\RequirePackage{fix-cm}
\documentclass[english]{article}
\usepackage[T1]{fontenc}
\usepackage[latin9]{inputenc}
\usepackage{geometry}
\usepackage{fancyhdr}
\usepackage{units}
\usepackage{textcomp}
\usepackage{amsthm}
\usepackage{amsmath}
\usepackage{amssymb}
\usepackage{graphicx}
\usepackage{esint}

\makeatletter

\providecommand{\tabularnewline}{\\}

\numberwithin{equation}{section}
\numberwithin{figure}{section}
\numberwithin{table}{section}
\theoremstyle{plain}
\newtheorem{thm}{\protect\theoremname}[section]
  \theoremstyle{definition}
  \newtheorem{defn}[thm]{\protect\definitionname}
  \theoremstyle{remark}
  \newtheorem*{rem*}{\protect\remarkname}
  \theoremstyle{plain}
  \newtheorem*{thm*}{\protect\theoremname}
  \theoremstyle{remark}
  \newtheorem*{acknowledgement*}{\protect\acknowledgementname}
  \theoremstyle{remark}
  \newtheorem{rem}[thm]{\protect\remarkname}
  \theoremstyle{plain}
  \newtheorem{prop}[thm]{\protect\propositionname}
  \theoremstyle{plain}
  \newtheorem{lem}[thm]{\protect\lemmaname}
  \theoremstyle{plain}
  \newtheorem{cor}[thm]{\protect\corollaryname}
\newenvironment{lyxlist}[1]
{\begin{list}{}
{\settowidth{\labelwidth}{#1}
 \setlength{\leftmargin}{\labelwidth}
 \addtolength{\leftmargin}{\labelsep}
 }}
{\end{list}}

\usepackage{amsfonts}
\usepackage{amsthm}
\usepackage{amscd}
\usepackage[all]{xy}%\usepackage{newlfont}
\usepackage[perpage,para,symbol]{footmisc}
\usepackage[all]{xy}
\usepackage{graphicx}
\usepackage{setspace}
\usepackage{bbm}

\DeclareMathOperator{\dist}{dist}
\DeclareMathOperator{\im}{im}

\DeclareMathOperator{\Spec}{Spec}
\DeclareMathOperator{\Prob}{Prob}
\DeclareMathOperator{\supp}{supp}

\newcommand{\one}{\mathbbm{1}}

\theoremstyle{remark}
\newtheorem*{rems*}{Remarks}

\DefineFNsymbols*{lamportnostar}[math]{\dagger\ddagger\S\P\|{\dagger\dagger}{\ddagger\ddagger}}
\setfnsymbol{lamportnostar}

\theoremstyle{plain}
\newtheorem{claim}[thm]{\protect\claimname}

\makeatother

\usepackage{babel}
  \providecommand{\acknowledgementname}{Acknowledgement}
  \providecommand{\corollaryname}{Corollary}
  \providecommand{\definitionname}{Definition}
  \providecommand{\lemmaname}{Lemma}
  \providecommand{\propositionname}{Proposition}
  \providecommand{\remarkname}{Remark}
  \providecommand{\theoremname}{Theorem}
\providecommand{\theoremname}{Theorem}

\begin{document}

\title{Simplicial complexes:\\
spectrum, homology and random walks}

\author{Ori Parzanchevski%
\thanks{Supported by an Advanced ERC Grant, and the ISF.%
}\ \ and Ron Rosenthal%
\thanks{Supported by ERC StG 239990.%
}\\
\\
Hebrew University of Jerusalem}
\maketitle
\begin{abstract}
Random walks on a graph reflect many of its topological and spectral
properties, such as connectedness, bipartiteness and spectral gap
magnitude. In the first part of this paper we define a stochastic
process on simplicial complexes of arbitrary dimension, which reflects
in an analogue way the existence of higher dimensional homology, and
the magnitude of the high-dimensional spectral gap originating in
the works of Eckmann and Garland.

The second part of the paper is devoted to infinite complexes. We
present a generalization of Kesten's result on the spectrum of regular
trees, and of the connection between return probabilities and spectral
radius. We study the analogue of the Alon-Boppana theorem on spectral
gaps, and exhibit a counterexample for its high-dimensional counterpart.
We show, however, that under some assumptions the theorem does hold
- for example, if the codimension-one skeletons of the complexes in
question form a family of expanders.

Our study suggests natural generalizations of many concepts from graph
theory, such as amenability, recurrence/transience, and bipartiteness.
We present some observations regarding these ideas, and several open
questions.
\end{abstract}
\tableofcontents{}

\section{Introduction}

There are well known connections between dynamical, topological and
spectral properties of graphs: The random walk on a graph reflects
both its topological and algebraic connectivity, which are reflected
by the $0^{\mathrm{th}}$-homology and the spectral gap, respectively.

In this paper we present a stochastic process which takes place on
simplicial complexes of arbitrary dimension and generalizes these
connections. In particular, for a finite $d$-dimensional complex,
the asymptotic behavior of the process reflects the existence of a
nontrivial $\left(d-1\right)$-homology, and its rate of convergence
is dictated by the $\left(d-1\right)$-dimensional spectral gap%
\footnote{The high-dimensional spectral gap originates in the classic works
of Eckmann \cite{Eck44} and Garland \cite{Gar73}, and appears in
Definition \ref{The-spectral-gap} here.%
}. The study of the process on finite complexes occupies the first
half of the paper. In the second half we turn to infinite complexes,
studying the high-dimensional analogues of classic properties and
theorems regarding infinite graphs. Both in the finite and the infinite
cases, one encounters new phenomena along the familiar ones, which
reveal that graphs present only a degenerated case of a broader theory.

In order to give a flavor of the results without plunging into the
most general definitions, we present in §\ref{sub:Example-regular-triangle},
without proofs, the special case of regular triangle complexes. A
summary of the paper and its main results both for finite and infinite
complexes is presented in §\ref{sub:Summary-of-results}.

This manuscript is part of an ongoing research seeking to understand
the notion of \emph{high-dimensional expanders}. Namely, the analogue
of expander graphs in the realm of simplicial complexes of general
dimension. Here we discuss the dynamical aspect of expansion, i.e.\ asymptotic
behavior of random walks, and its relation to spectral expansion and
homology. In a previous paper we studied expansion from combinatorial
and isoperimetric points of view \cite{parzanchevski2012isoperimetric}.
The study of high-dimensional expanders is currently a very active
one, comprising the notions of geometric and topological expansion
in \cite{Gromov2010,fox2011,Matouvsek2011}, $\mathbb{F}_{2}$-coboundary
expansion in \cite{Linial2006,meshulam2009homological,dotterrer2010coboundary,Gundert2012,mukherjee2012cheeger},
and Ramanujan complexes in \cite{cartwright2003ramanujan,li2004ramanujan,Lubotzky2005a}.
We refer the reader to \cite{Lubotzky2013} for a survey of the current
state of the field.

\subsection{\label{sub:Example-regular-triangle}Example - regular triangle complexes}

First, let us observe the $\frac{1}{2}$-lazy random walk on a $k$-regular
graph $G=\left(V,E\right)$: the walker starts at a vertex $v_{0}$,
and at each step remains in place with probability $\frac{1}{2}$
or moves to each of its $k$ neighbors with probability $\frac{1}{2k}$.
Let $\mathbf{p}_{n}^{v_{0}}\left(v\right)$ denote the probability
of finding the walker at the vertex $v$ at time $n$. The following
observations are classic:
\begin{enumerate}
\item If $G$ is finite, then $\mathbf{p}_{\infty}^{v_{0}}=\lim_{n\rightarrow\infty}\mathbf{p}_{n}^{v_{0}}$
exists, and it is constant if and only if $G$ is connected.
\item Furthermore, the rate of convergence is given by
\[
\left\Vert \mathbf{p}_{n}^{v_{0}}-\mathbf{const}\right\Vert =O\left(\left(1-\frac{1}{2}\lambda\left(G\right)\right)^{n}\right),
\]
where $\lambda\left(G\right)$ is the spectral gap of $G$ (the definition
follows below).
\item When $G$ is infinite and connected, the spectral gap is related to
the return probability of the walk by
\begin{equation}
\lim_{n\rightarrow\infty}\sqrt[n]{\mathbf{p}_{n}^{v_{0}}\left(v_{0}\right)}=1-\frac{1}{2}\lambda\left(G\right).\label{eq:return_prob_inf_graph}
\end{equation}

\end{enumerate}
We recall the basic definitions: the \emph{Laplacian }of $G$, which
we denote by $\Delta^{+}$, is the operator which acts on $\mathbb{R}^{V}$
by 
\[
\left(\Delta^{+}f\right)\left(v\right)=f\left(v\right)-\frac{1}{k}\sum\limits _{w\sim v}f\left(w\right)
\]
(where $\sim$ denotes neighboring vertices). If $G$ is finite, then
its \emph{spectral gap} $\lambda\left(G\right)$ is defined as the
minimal Laplacian eigenvalue on a function whose sum on $V$ vanishes.
When $G$ is infinite, its spectral gap is defined to be $\lambda\left(G\right)=\min\Spec\left(\Delta^{+}\big|_{L^{2}\left(V\right)}\right)$
(for more on this see §\ref{sub:Infinite-graphs}).

\medskip{}

Moving one dimension higher, let $X=\left(V,E,T\right)$ be a \emph{$k$-regular
triangle complex}. This means that $E\subseteq{V \choose 2}$ (i.e.\ $E$
consists of subsets of $V$ of size $2$, the \emph{edges} of $X$),
$T\subseteq{V \choose 3}$ (the \emph{triangles}), every edge is contained
in exactly $k$ triangles in $T$, and for every triangle $\left\{ u,v,w\right\} $
in $T$, the edges forming its boundary, $\left\{ u,v\right\} $,
$\left\{ u,w\right\} $ and $\left\{ v,w\right\} $, are in $E$.

For $\left\{ v,w\right\} \in E$ we denote the directed edge $\xymatrix@1@C=20pt{
\overset{v}{\bullet}\ar@{-}[r]|(0.6)*=0@{>}         & \overset{w}{\bullet} }$ by $\left[v,w\right]$, and the set of all directed edges by $E_{\pm}$
(so that $\left|E_{\pm}\right|=2\left|E\right|$). For $e\in E_{\pm}$,
$\overline{e}$ denotes the edge with the same vertices and opposite
direction, i.e.\ $\vphantom{\overset{a}{a}}\overline{\left[v,w\right]}=\left[w,v\right]$. 

The following definition is the basis of the process which we shall
study:
\begin{defn}
Two directed edges $e,e'\in E_{\pm}$ are called \emph{neighbors}
(indicated by $e\sim e'$) if they have the same origin or the same
terminus, and the triangle they form is in the complex. Namely, if
$e=\left[v,w\right]$ and $e'=\left[v',w'\right]$, then $e\sim e'$
means that either $v=v'$ and $\left\{ v,w,w'\right\} \in T$, or
$w=w'$ and $\left\{ v,v',w'\right\} \in T$.
\end{defn}
We study the following lazy random walk on $E_{\pm}$: The walk starts
at some directed edge $e_{0}\in E_{\pm}$. At every step, the walker
stays put with probability $\frac{1}{2}$, or else move to a uniformly
chosen neighbor. Figure \ref{fig:One-step-of-triangle} illustrates
one step of the process, in two cases (the right one is non-regular,
but the walk is defined in the same manner). 

\begin{figure}[h]
\begin{centering}
\includegraphics[width=0.8\columnwidth]{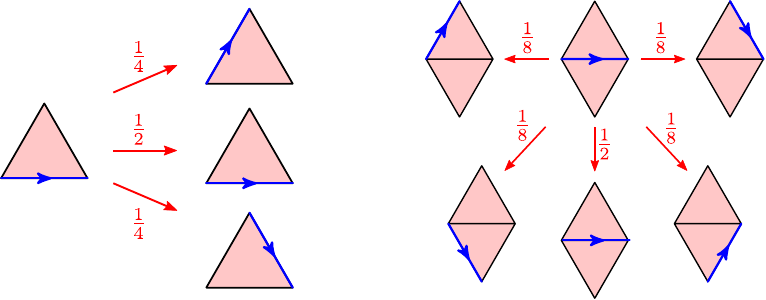}
\par\end{centering}

\caption{One step of the edge walk.\label{fig:One-step-of-triangle}}
\end{figure}

As in the random walk on a graph, this process induces a sequence
of distributions on $E_{\pm}$,
\[
\mathbf{p}_{n}\left(e\right)=\mathbf{p}_{n}^{e_{0}}\left(e\right),
\]
describing the probability of finding the walker at the directed edge
$e$ at time $n$ (having started from $e_{0}$). However, studying
the evolution of $\mathbf{p}_{n}$ amounts to studying the traditional
random walk on the graph with vertices $E_{\pm}$ and edges defined
by $\sim$. This will not take us very far, and in particular will
not reveal the presence or absence of first homology in $X$. Instead,
we study the evolution of what we call the ``expectation process''
on $X$:
\[
\mathcal{E}_{n}\left(e\right)=\mathcal{E}_{n}^{e_{0}}\left(e\right)=\mathbf{p}_{n}^{e_{0}}\left(e\right)-\mathbf{p}_{n}^{e_{0}}\left(\overline{e}\right),
\]
i.e.\ the probability of finding the walker at time $n$ at $e$,
minus the probability of finding it at the opposite edge $\overline{e}$
(for the reasons behind the name see Remark \ref{rem:process-name}). 

It is tempting to look at $\mathcal{E}_{\infty}^{e_{0}}=\lim_{n\rightarrow\infty}\mathcal{E}_{n}^{e_{0}}$
as is done in graphs, but a moment of reflection will show the reader
that $\mathcal{E}_{\infty}^{e_{0}}\equiv0$ for any finite triangle
complex, and any starting point $e_{0}$. Namely, the probabilities
of reaching $e$ and $\overline{e}$ become arbitrarily close, for
every $e$. While this might cause initial worry, it turns out that
the rate of decay of $\mathcal{E}_{n}$ is always the same: for any
finite triangle complex one has $\left\Vert \mathcal{E}_{n}^{e_{0}}\right\Vert =\Theta\left(\left(\frac{3}{4}\right)^{n}\right)$.
It is therefore reasonable to turn our attention to the \emph{normalized
expectation process,}
\[
\widetilde{\mathcal{E}}_{n}^{e_{0}}\left(e\right)=\left({\textstyle \frac{4}{3}}\right)^{n}\mathcal{E}_{n}^{e_{0}}\left(e\right)=\left({\textstyle \frac{4}{3}}\right)^{n}\left[\mathbf{p}_{n}^{e_{0}}\left(e\right)-\mathbf{p}_{n}^{e_{0}}\left(\overline{e}\right)\vphantom{\Big|}\right],
\]
and observe its limit, 
\[
\widetilde{\mathcal{E}}_{\infty}^{e_{0}}=\lim_{n\rightarrow\infty}\widetilde{\mathcal{E}}_{n}^{e_{0}}.
\]
For a finite triangle complex this limit always exists, and is nonzero.
This is the object which reveals the first homology of the complex. 

To see how, we need the following definition: We say that $f:E_{\pm}\rightarrow\mathbb{R}$
is \emph{exact} if its sum along every closed path vanishes; namely,
if 
\[
v_{0}\sim v_{1}\sim\ldots\sim v_{n}=v_{0}\qquad\Longrightarrow\qquad\sum_{i=0}^{n-1}f\left(\left[v_{i},v_{i+1}\right]\right)=0.
\]
This is the one-dimensional analogue of constant functions (for reasons
which will become clear in §\ref{sub:Simplicial-complexes-and}),
and the following holds:
\begin{enumerate}
\item For a finite $X$, $\mathcal{\widetilde{E}}_{\infty}^{e_{0}}$ is
exact for every $e_{0}\in E_{\pm}$ if and only if $G$ has a trivial
first homology.
\item Furthermore, the rate of convergence is given by
\[
\left\Vert \mathcal{\widetilde{E}}_{n}^{e_{0}}-\mathbf{exact}\right\Vert =O\left(\left(1-\frac{1}{3}\lambda\left(X\right)\right)^{n}\right),
\]
where $\lambda\left(X\right)$ is the spectral gap of $X$ (see Definition
\ref{The-spectral-gap}).
\item If $X$ is infinite and every vertex in $X$ is of infinite degree,
then its spectral gap (which is defined in §\ref{sub:General-dimension})
is revealed by the ``return expectation'': 
\[
\sup_{e_{0}\in E_{\pm}}\lim_{n\rightarrow\infty}\sqrt[n]{\widetilde{\mathcal{E}}_{n}^{e_{0}}\left(e_{0}\right)}=1-\frac{1}{3}\lambda\left(X\right).
\]

\end{enumerate}
What if one is interested not only in the existence of a first homology,
but also in its dimension? The answer is manifested in the walk as
well. In graphs the number of connected components is given by the
dimension of $\mathrm{Span}\left\{ \mathbf{p}_{\infty}^{v_{0}}\,\middle|\, v_{0}\in V\right\} $,
and an analogue statement holds here (see Theorem \ref{thm:walk-and-spec}).
\begin{rem*}
If the non-lazy\emph{ }walk on a finite graph is observed, then apart
from disconnectedness there is another obstruction for convergence
to the uniform distribution: \emph{bipartiteness}. We shall see that
this is a special case of an obstruction in general dimension, which
we call \emph{disorientability} (see Definition \ref{def:A-disorientation}).
In our example we have avoided this problem by considering the lazy
walk, both on graphs and on triangle complexes.
\end{rem*}

\subsection{\label{sub:Summary-of-results}Summary of results}

We give now a brief summary of the paper and its main results. The
definitions of the terms which appear in this section are explained
throughout the paper.

In §\ref{sub:The-walk} we define a $p$-lazy random walk on the oriented
$\left(d-1\right)$-cells of a $d$-dimensional complex $X$, and
associate with this walk the normalized expectation process $\widetilde{\mathcal{E}}_{n}^{\sigma_{0}}$.
In §\ref{sub:Walk-and-spectrum} it is shown that the limit of this
process $\widetilde{\mathcal{E}}_{\infty}^{\sigma_{0}}=\lim_{n}\widetilde{\mathcal{E}}_{n}^{\sigma_{0}}$
always exists and captures various properties of $X$, according to
the amount of laziness $p$ (this is an abridged version of Theorem
\ref{thm:walk-and-spec}):
\begin{thm*}
When $\frac{d-1}{3d-1}<p<1$, $\widetilde{\mathcal{E}}_{\infty}^{\sigma_{0}}$
is exact for every starting point $\sigma_{0}$ if and only if the
$\left(d-1\right)$-homology of $X$ is trivial. If furthermore $p\geq\frac{1}{2}$
then the rate of convergence is controlled by the spectral gap of
$X$:
\[
\dist\left(\widetilde{\mathcal{E}}_{n}^{\sigma_{0}},\widetilde{\mathcal{E}}_{\infty}^{\sigma_{0}}\right)=O\left(\left(1-\frac{1-p}{p\left(d-1\right)+1}\lambda\left(X\right)\right)^{n}\right).
\]

When $p=\frac{d-1}{3d-1}$, $\widetilde{\mathcal{E}}_{\infty}^{\sigma_{0}}$
is exact for every starting point $\sigma_{0}$ if and only if the
$\left(d-1\right)$-homology of $X$ is trivial, and in addition $X$
has no disorientable $\left(d-1\right)$-components (see Definitions
\ref{def:connectedness}, \ref{def:A-disorientation}).
\end{thm*}
Next, we move on to discuss infinite complexes, showing that they
present new aspects which do not appear in infinite graphs. In §\ref{sub:Example---arboreal}
we define a family of simplicial complexes (which we call \emph{arboreal
complexes}) generalizing the notion of trees. In Theorem \ref{thm:The-spectrum-of-trees}
we compute their spectra, generalizing Kesten's classic result on
the spectrum of regular trees \cite{MR0109367}. The spectra of the
regular arboreal complexes of high dimension and low regularity exhibit
a surprising new phenomenon - an isolated eigenvalue. 

Sections \ref{sub:Continuity-of-the} and \ref{sub:Alon-Boppana-type-theorems}
are devoted to study the behavior of the spectrum with respect to
a limit in the space of complexes. In particular we are interested
in the high-dimensional analogue of the Alon-Bopanna theorem, which
states that if a sequence of graphs $G_{n}$ convergences to a graph
$G$, then $\liminf_{n\rightarrow\infty}\lambda\left(G_{n}\right)\leq\lambda\left(G\right)$.
We first show that in general this need not hold in higher dimension
(Theorem \ref{thm:T_2_2-counterexample}). This uses the isolated
eigenvalue of the $2$-regular arboreal complex of dimension two,
which is shown in Figure \ref{fig:T_2_2}, as well as a study of the
spectrum of balls in this complex (shown in Figure \ref{fig:The-orientation-T_2_2}). 

Even though the Alon-Bopanna theorem does not hold in general in high
dimension, we show that under a variety of conditions it does hold
(Theorem \ref{thm:high-alon-boppana}):
\begin{thm*}
If \textup{$X_{n}\overset{{\scriptscriptstyle n\rightarrow\infty}}{\longrightarrow}X$},
and one of the following holds:
\begin{enumerate}
\item The spectral gap of $X$ is nonzero,
\item zero is a non-isolated point in the spectrum of $X$, or
\item the $\left(d-1\right)$-skeletons of the complexes $X_{n}$ form a
family of $\left(d-1\right)$-expanders,
\end{enumerate}

then 
\[
\liminf_{n\rightarrow\infty}\lambda\left(X_{n}\right)\leq\lambda\left(X\right).
\]

\end{thm*}
In §\ref{sub:Spectral-radius-and} we show that the connection between
the spectrum of a graph, and the return probability of the random
walk on it (see e.g.\ \cite[Lemma 2.2]{MR0109367}), generalizes
to higher dimensions (Proposition \ref{prop:spectral_radius}).

The final section on infinite complexes addresses the high-dimensional
analogues of the concepts of amenability, recurrence and transience,
proving some properties of these (Proposition \ref{pro:amen-trans}),
and raising many open questions.
\begin{acknowledgement*}
We would like to thank Alex Lubotzky and Gil Kalai who prompted this
research, and Jonathan Breuer for many helpful discussions. We are
also grateful to Noam Berger, Emmanuel Farjoun, Nati Linial, Doron
Puder and Andrzej \.{Z}uk for their insightful comments. 
\end{acknowledgement*}

\section{Finite complexes}

Throughout this section $X$ is a finite $d$-dimensional simplicial
complex on a finite vertex set $V$. This means that $X$ is comprised
of subsets of $V$, called \emph{cells}, and the subset of every cell
is also a cell. A cell of size $j+1$ (where $-1\leq j$) is said
to be of \emph{dimension} $j$, and $X^{j}$ denotes the set of $j$-cells
- cells of dimension $j$. The dimension of $X$ is the maximal dimension
of a cell in it. The \emph{degree} of a $j$-cell $\sigma$, denoted
$\deg\left(\sigma\right)$, is the number of $\left(j+1\right)$-cells
containing it. We shall assume that $X$ is \emph{uniform}, meaning
that every cell is contained in some cell of dimension $d=\dim X$.

For $j\geq1$, every $j$-cell $\sigma=\left\{ \sigma_{0},\ldots,\sigma_{j}\right\} $
has two possible orientations, corresponding to the possible orderings
of its vertices, up to an even permutation ($1$-cells and the empty
cell have only one orientation). We denote an oriented cell by square
brackets, and a flip of orientation by an overbar. For example, one
orientation of $\sigma=\left\{ x,y,z\right\} $ is $\left[x,y,z\right]$,
which is the same as $\left[y,z,x\right]$ and $\left[z,x,y\right]$.
The other orientation of $\sigma$ is $\overline{\left[x,y,z\right]}=\left[y,x,z\right]=\left[x,z,y\right]=\left[z,y,x\right]$.
We denote by $X_{\pm}^{j}$ the set of oriented $j$-cells (so that
$\left|X_{\pm}^{j}\right|=2\left|X^{j}\right|$ for $j\geq1$ and
$X_{\pm}^{j}=X^{j}$ for $j=-1,0$), and we shall occasionally denote
by $X_{+}^{j}$ a choice of orientation for $X^{j}$, i.e.\ a subset
of $X_{\pm}^{j}$ such that $X_{\pm}^{j}$ is the disjoint union of
$X_{+}^{j}$ and $\left\{ \overline{\sigma}\,\middle|\,\sigma\in X_{+}^{j}\right\} $.

The \emph{faces} of a $j$-cell $\sigma=\left\{ v_{0},\ldots,v_{j}\right\} $
are the $\left(j-1\right)$-cells $\sigma\backslash v_{0},\sigma\backslash v_{1},\ldots,\sigma\backslash v_{j}$.
An oriented $j$-cell $\sigma=\left[v_{0},\ldots,v_{j}\right]$ ($2\leq j\leq d$)
induces an orientation on its faces as follows: the face $\left\{ v_{0},\ldots,v_{i-1},v_{i+1},\ldots,v_{j}\right\} $
is oriented as $\left(-1\right)^{i}\left[v_{1},\ldots,v_{i-1},v_{i+1},\ldots,v_{k}\right]$,
where $\left(-1\right)^{i}$ means taking the opposite orientation
when $\left(-1\right)^{i}$ is $-1$. 

Finally, we define the space of \emph{$j$-forms }on $X$: these are
functions on $X_{\pm}^{j}$ which are antisymmetric w.r.t.\ a flip
of orientation:
\[
\Omega^{j}=\Omega^{j}\left(X\right)=\left\{ f:X_{\pm}^{j}\rightarrow\mathbb{R}\,\middle|\, f\left(\overline{\sigma}\right)=-f\left(\sigma\right)\quad\forall\sigma\in X_{\pm}^{d-1}\right\} .
\]
For $j=-1,0$ there are no flips; $\Omega^{0}$ is just the space
of functions on the vertices, and $\Omega^{-1}=\left\{ f:\left\{ \varnothing\right\} \rightarrow\mathbb{R}\right\} $
can be identified in a natural way with $\mathbb{R}$. With every
oriented $j$-cell $\sigma\in X^{j}$ we associate the Dirac $j$-form
$\mathbbm{1}_{\sigma}$ defined by 
\[
\mathbbm{1}_{\sigma}\left(\sigma'\right)=\begin{cases}
1 & \sigma'=\sigma\\
-1 & \sigma'=\overline{\sigma}\\
0 & \mathrm{otherwise}
\end{cases}
\]
(for $j=0$ this is the standard Dirac function, and $\mathbbm{1}_{\varnothing}$
is the constant $1$).

\subsection{\label{sub:The-walk}The $\left(d-1\right)$-walk and expectation
process}

Let $X$ be a uniform $d$-dimensional complex and $0\leq p<1$. The
following process is the generalization of the edge walk from §\ref{sub:Example-regular-triangle}:
\begin{defn}
\label{def:The--lazy--walk}The $p$-lazy $\left(d-1\right)$-walk
on a $d$-complex $X$ is defined as follows: 
\begin{itemize}
\item Two oriented $\left(d-1\right)$-cells $\sigma,\sigma'\in X_{\pm}^{d-1}$
are said to be \emph{neighbors} (denoted $\sigma\sim\sigma'$) if
there exists an oriented $d$-cell $\tau$, such that both $\sigma$
and $\overline{\sigma'}$ are faces of $\tau$ with the orientations
induced by it (see Figure \ref{fig:One-step-of-triangle}).
\item The walk starts at an initial oriented $\left(d-1\right)$-cell $\sigma_{0}$,
and at each step the walker stays in place with probability $p$,
and with probability $\left(1-p\right)$ chooses uniformly one of
its neighbors and moves to it.
\end{itemize}
\end{defn}
Put differently, it is the Markov chain on $X_{\pm}^{d-1}$ with transition
probabilities 
\[
\Prob\left(X_{n+1}=\sigma'|X_{n}=\sigma\right)=\begin{cases}
p & \quad\sigma'=\sigma\\
\frac{1-p}{d\deg(\sigma)} & \quad\sigma'\sim\sigma\\
0 & \quad\mbox{otherwise}
\end{cases},
\]
(note that $\sigma$ is contained in $\deg\left(\sigma\right)$ $d$-cells,
and thus has $d\cdot\deg\sigma$ neighbors!)

We remark that neighboring cells can also be described in the following
way: if $\sigma,\sigma'\in X_{\pm}^{j}$ and $j\geq2$, then $\sigma\sim\sigma'$
iff the unoriented cell $\sigma\cup\sigma'$ is in $X^{d}$, and the
intersection $\sigma\cap\sigma'$ inherits the same orientation from
both $\sigma$ and $\sigma'$. For $j=1$, this can be interpreted
as follows: two edges $e,e'\in X_{\pm}^{1}$ are neighbors if they
bound a triangle in the complex, and the vertex at which they intersect
``inherits the same orientation from both of them'': it is either
the origin of both $e$ and $e'$, or the terminus of both. Finally,
for $j=0$ Definition \ref{def:The--lazy--walk} gives the standard
neighboring relation and $p$-lazy random walk on a graph.
\begin{defn}
\label{def:connectedness}We say that $X$ is \emph{$\left(d-1\right)$-connected}
if the $\left(d-1\right)$-walk on it is irreducible, i.e., for every
pair of oriented $\left(d-1\right)$-cells $\sigma$ and $\sigma'$
there exist a chain $\sigma=\sigma_{0}\sim\sigma_{1}\sim\ldots\sim\sigma_{n}=\sigma'$.
Moreover, having such a chain defines an equivalence relation on the
$\left(d-1\right)$-cells of $X$, whose classes we call the \emph{$\left(d-1\right)$-components}
of $X$.\end{defn}
\begin{rem*}
Due to the assumption of uniformity, it is enough to observe unoriented
cells - $X$ is $\left(d-1\right)$-connected iff for every $\sigma,\sigma'\in X^{d-1}$
there exists a chain of unoriented $\left(d-1\right)$-cells $\sigma=\sigma_{0},\sigma_{1},\ldots,\sigma_{n}=\sigma'$
with $\sigma_{i}\cup\sigma_{i+1}\in X^{d}$ for all $i$. This is
also equivalent to the assertion that for any $\tau,\tau'\in X^{d}$
there is a chain $\tau=\tau_{0},\tau_{1},\ldots,\tau_{m}=\tau'$ of
$d$-cells with $\tau_{i}\cap\tau_{i-1}\in X^{d-1}$ for all $i$
(this is sometimes referred to as a \emph{chamber complex}). We note
that it follows from uniformity that a $\left(d-1\right)$-connected
complex is connected as a topological space. The other direction does
not hold: the complex $\blacktriangleright\!\blacktriangleleft$ is
not $1$-connected, even though it is connected (and uniform).
\end{rem*}
Observing the $\left(d-1\right)$-walk on $X$, we denote by $\mathbf{p}_{n}^{\sigma_{0}}\left(\sigma\right)$
the probability that the random walk starting at $\sigma_{0}$ reaches
$\sigma$ at time $n$. We then have:
\begin{defn}
\label{def:norm-exp-proc}For $d\geq2$, the \emph{expectation process}
on $X$ starting at $\sigma_{0}$ is the sequence of $\left(d-1\right)$-forms
$\left\{ \mathcal{E}_{n}^{\sigma_{0}}\right\} _{n=0}^{\infty}$ defined
by
\[
\mathcal{E}_{n}^{\sigma_{0}}\left(\sigma\right)=\mathbf{p}_{n}^{\sigma_{0}}\left(\sigma\right)-\mathbf{p}_{n}^{\sigma_{0}}\left(\overline{\sigma}\right).
\]
For $d=1$ (i.e.\ graphs) we simply define $\mathcal{E}_{n}^{v_{0}}=\mathbf{p}_{n}^{v_{0}}.$%
\footnote{The results in the paper hold for graphs as well, using this definition
of $\mathcal{E}_{n}^{v_{0}}$, but they are all familiar theorems.
In some cases the proofs are slightly different, and we will not trouble
to handle this special case.%
}

The \emph{normalized expectation process }is defined to be 
\[
\widetilde{\mathcal{E}}_{n}^{\sigma_{0}}\left(\sigma\right)=\left(\frac{d}{p\left(d-1\right)+1}\right)^{n}\mathcal{E}_{n}^{\sigma_{0}}\left(\sigma\right)=\left(\frac{d}{p\left(d-1\right)+1}\right)^{n}\left[\mathbf{p}_{n}^{\sigma_{0}}\left(\sigma\right)-\mathbf{p}_{n}^{\sigma_{0}}\left(\overline{\sigma}\right)\right],
\]
where $p$ is the laziness of the walk. In particular, for $d=1$
one has $\widetilde{\mathcal{E}}_{n}^{v_{0}}=\mathcal{E}_{n}^{v_{0}}=\mathbf{p}_{n}^{v_{0}}$
for all $p$.
\end{defn}
The reason for this particular normalization is that for a lazy enough
process (in particular for $p\geq\frac{1}{2}$) one has $\left\Vert \mathcal{E}_{n}^{\sigma_{0}}\right\Vert =\Theta\left(\left(\frac{p\left(d-1\right)+1}{d}\right)^{n}\right)$
(see (\ref{eq:growth-exp-proc})). Note that $\widetilde{\mathcal{E}}_{0}^{\sigma_{0}}=\mathcal{E}_{0}^{\sigma_{0}}=\mathbbm{1}_{\sigma_{0}}$.
\begin{rem}
\label{rem:process-name}The name ``expectation process'' comes
from the fact that for any $\left(d-1\right)$-form $f$
\[
\mathbb{E}_{n}^{\sigma_{0}}\left[f\right]=\sum_{\sigma\in X_{\pm1}^{d-1}}\mathbf{p}_{n}^{\sigma_{0}}\left(\sigma\right)f\left(\sigma\right)=\sum_{\sigma\in X^{d-1}}\mathcal{E}_{n}^{\sigma_{0}}\left(\sigma\right)f\left(\sigma\right)
\]
where, as implied by the notation, $\mathcal{E}_{n}^{\sigma_{0}}\left(\sigma\right)f\left(\sigma\right)$
does not depend on the orientation of $\sigma$.
\end{rem}
The evolution of the expectation process over time is given by $\mathcal{E}_{n+1}^{\sigma_{0}}=A\mathcal{E}_{n}^{\sigma_{0}}$,
where $A=A\left(X,p\right)$ is the \emph{transition operator }acting
on $\Omega^{d-1}$ by 
\begin{equation}
\left(Af\right)\left(\sigma\right)=pf\left(\sigma\right)+\frac{\left(1-p\right)}{d}\sum_{\sigma'\sim\sigma}\frac{f\left(\sigma'\right)}{\deg\left(\sigma'\right)}\qquad\left(f\in\Omega^{d-1},\sigma\in X^{d-1}\right).\label{eq:transition-operator}
\end{equation}
Note that the evolution of $\mathbf{p}_{n}^{\sigma_{0}}$ is given
by the same $A$, acting on all functions from $X_{\pm}^{d}$ to $\mathbb{R}$,
and not only on forms. 

It is sometimes useful to observe the Markov operator $M=M\left(X,p\right)$
associated with this evolution, which is characterized by 
\[
\mathbb{E}_{n+1}^{\sigma_{0}}\left[f\right]=\mathbb{E}_{n}^{\sigma_{0}}\left[Mf\right],
\]
and is given explicitly by 
\[
\left(Mf\right)\left(\sigma\right)=pf\left(\sigma\right)+\frac{1-p}{d\deg\left(\sigma\right)}\sum_{\sigma'\sim\sigma}f\left(\sigma'\right)\qquad\left(f\in\Omega^{d-1},\sigma\in X^{d-1}\right).
\]
This is the transpose of $A$, w.r.t.\ to a natural choice of basis
for $\Omega^{d-1}\left(X\right)$.

\subsection{\label{sub:Simplicial-complexes-and}Simplicial complexes and Laplacians}

For a cell $\sigma$ and a vertex $v\notin\sigma$, we write $v\triangleleft\sigma$
if $v\sigma=\left\{ v\right\} \cup\sigma$ is a cell in $X$. If $\sigma$
is oriented, $\sigma=\left[\sigma_{0},\ldots,\sigma_{k}\right]$,
and $v\triangleleft\sigma$, then $v\sigma$ denotes the oriented
cell $\left[v,\sigma_{0},\ldots,\sigma_{k}\right]$.

For $0\leq k\leq d$, the $k^{th}$ \emph{boundary operator $\partial_{k}:\Omega^{k}\rightarrow\Omega^{k-1}$}
is defined by 
\[
\left(\partial_{k}f\right)\left(\sigma\right)=\sum_{v\triangleleft\sigma}f\left(v\sigma\right).
\]
In particular $\partial_{0}:\Omega^{0}\rightarrow\Omega^{-1}$ is
defined by $\left(\partial_{0}f\right)\left(\varnothing\right)=\sum_{v\in X^{0}}f(v).$ 

The sequence $\left(\Omega^{k},\partial_{k}\right)$ is the \emph{simplicial
chain complex} of $X$, meaning that $\partial_{k}\partial_{k+1}=0$
for all $k$, giving rise to the $k$\emph{-cycles} $Z_{k}=\ker\partial_{k}$,
the \emph{$k$-boundaries} $B_{k}=\im\partial_{k+1}$ and the (real)
$k^{th}$\emph{-homology} $H_{k}=\nicefrac{Z_{k}}{B_{k}}$.

Given a weight function $w:X\rightarrow\left(0,\infty\right)$, $\Omega^{k}$
become inner product spaces (for $-1\leq k\leq d$) with 
\[
\left\langle f,g\right\rangle =\sum_{\sigma\in X^{k}}w(\sigma)f(\sigma)g(\sigma)\quad\quad\forall f,g\in\Omega^{k}.
\]
Note that the sum is over $X^{k}$ and not $X_{\pm}^{k}$ - this is
well defined since the value of $f\left(\sigma\right)g\left(\sigma\right)$
is independent of the orientation of $\sigma$. 

Since $X$ is finite the spaces $\Omega^{k}$ are finite dimensional,
and there exist adjoint operators to the boundary operators $\partial_{k}$.
These are the \emph{co-boundary operators}, which are denoted by $\delta_{k}=\partial_{k}^{*}:\Omega^{k-1}\rightarrow\Omega^{k}$,
and are given by 
\begin{align*}
\left(\delta_{k}f\right)\left(\sigma\right)=\left(\partial_{k}^{*}f\right)\left(\sigma\right) & =\frac{1}{w\left(\sigma\right)}\sum_{i=0}^{k}\left(-1\right)^{i}w\left(\sigma\backslash\sigma_{i}\right)f\left(\sigma\backslash\sigma_{i}\right)\quad0\leq k\leq d.
\end{align*}

We will stick with the notation $\partial_{k}^{*}$ until we get to
infinite complexes, where sometimes $\delta_{k}$ is defined even
though $\partial_{k}$ is not. The \emph{simplicial cochain complex}
of $X$ is $\left(\Omega_{k},\delta_{k}\right)$, and $Z^{k}=\ker\delta_{k+1}$,
$B^{k}=\im\delta_{k}$, $H^{k}=\nicefrac{Z^{k}}{B^{k}}$ are the \emph{cocycles,
coboundaries }and \emph{cohomology}, respectively. Cocycles are also
known as \emph{closed forms}, and coboundaries as \emph{exact forms}.

The following weight functions will be used throughout this paper%
\footnote{Another natural weight function is the constant one. The obtained
Laplacians are more convenient for isoperimetric analysis. For more
details see \cite{parzanchevski2012isoperimetric}.%
}: 
\[
w\left(\sigma\right)=\begin{cases}
\frac{1}{\deg\sigma} & \sigma\in X^{d-1}\\
1 & \sigma\in X\backslash X^{d-1}
\end{cases}.
\]
Notice that for $\sigma\in X^{d-1}$ 
\[
\frac{1}{w\left(\sigma\right)}=\deg\left(\sigma\right)=\left|\left\{ \tau\in X^{d}\,\middle|\,\sigma\subset\tau\right\} \right|=\left|\left\{ v\,\middle|\, v\triangleleft\sigma\right\} \right|=\frac{1}{d}\left|\left\{ \sigma'\in X^{d-1}\,\middle|\,\sigma'\sim\sigma\right\} \right|.
\]

Due to our choice of weights, the inner product and coboundary operators
are given by
\begin{align}
\left\langle f,g\right\rangle  & =\begin{cases}
\sum\limits _{\sigma\in X^{k}}f\left(\sigma\right)g\left(\sigma\right) & f,g\in\Omega^{k},k\neq d-1\\
\sum\limits _{\sigma\in X^{d-1}}\frac{f\left(\sigma\right)g\left(\sigma\right)}{\deg\sigma} & f,g\in\Omega^{d-1}
\end{cases}\label{eq:our-inner-product}\\
\left(\delta_{k}f\right)\left(\sigma\right)=\left(\partial_{k}^{*}f\right)\left(\sigma\right) & =\begin{cases}
\sum\limits _{i=0}^{k}\left(-1\right)^{i}f\left(\sigma\backslash\sigma_{i}\right) & k\leq d-2\\
\deg\left(\sigma\right)\sum\limits _{i=0}^{d-1}\left(-1\right)^{i}f\left(\sigma\backslash\sigma_{i}\right)\vspace{0.1cm} & k=d-1\\
\sum\limits _{i=0}^{d}\dfrac{\left(-1\right)^{i}f\left(\sigma\backslash\sigma_{i}\right)}{\deg\left(\sigma\backslash\sigma_{i}\right)} & k=d.
\end{cases}\label{eq:coboundary-operator}
\end{align}

Finally, the \emph{upper, lower, and full} \emph{Laplacians} $\Delta^{+},\Delta^{-},\Delta:\Omega^{d-1}\rightarrow\Omega^{d-1}$
are defined by:
\begin{align*}
\Delta^{+} & =\partial_{d}\delta_{d}=\partial_{d}\partial_{d}^{*}\\
\Delta^{-} & =\delta_{d-1}\partial_{d-1}=\partial_{d-1}^{*}\partial_{d-1}\\
\Delta^{\phantom{+}} & =\Delta^{+}+\Delta^{-}.
\end{align*}
An explicit calculation gives 
\begin{gather}
\begin{aligned}\begin{aligned}\left(\Delta^{+}f\right)\left(\sigma\right)\end{aligned}
 & =\sum_{v\triangleleft\sigma}\left(\partial_{d}^{*}f\right)\left(v\sigma\right)=\sum_{v\triangleleft\sigma}\sum_{i=0}^{d}\frac{\left(-1\right)^{i}f\left(v\sigma\backslash\left(v\sigma\right)_{i}\right)}{\deg\left(v\sigma\backslash\left(v\sigma\right)_{i}\right)}\\
 & =f\left(\sigma\right)-\sum_{v\triangleleft\sigma}\sum_{i=0}^{d-1}\frac{\left(-1\right)^{i}f\left(v\left(\sigma\backslash\sigma_{i}\right)\right)}{\deg\left(v\left(\sigma\backslash\sigma_{i}\right)\right)}=f\left(\sigma\right)-\sum_{\sigma'\sim\sigma}\frac{f\left(\sigma'\right)}{\deg\left(\sigma'\right)}
\end{aligned}
\label{eq:LaplacianUp}
\end{gather}
and 
\begin{equation}
\left(\Delta^{-}f\right)\left(\sigma\right)=\deg\sigma\sum_{i=0}^{d-1}\left(-1\right)^{i}\sum_{v\triangleleft\sigma\backslash\sigma_{i}}f\left(v\sigma\backslash\sigma_{i}\right).\label{eq:LaplacianDown}
\end{equation}
More generally, one can define the $k$-th upper, lower and full Laplacians
$\Delta_{k}^{+}=\partial_{k+1}\delta_{k+1}$, $\Delta_{k}^{-}=\delta_{k}\partial_{k}$
and $\Delta_{k}=\Delta_{k}^{+}+\Delta_{k}^{-}$. Apart from $k=d-1$,
these will only make a brief appearance in §\ref{sub:Alon-Boppana-type-theorems}.
The kernel of $\Delta_{k}$ is the space of \emph{harmonic $k$-forms},
denoted by $\mathcal{H}^{k}=\mathcal{H}^{k}\left(X\right)$. 

The spaces defined so far are related by 
\begin{gather*}
\begin{aligned}Z_{k} & =\ker\partial_{k}=\ker\Delta_{k}^{-} & \quad B_{k} & =\left(Z^{k}\right)^{\perp}=\im\partial_{k+1}=\im\Delta_{k}^{+}\\
Z^{k} & =\ker\delta_{k+1}=\ker\Delta_{k}^{+} & \quad B^{k} & =Z_{k}^{\perp}=\im\delta_{k}=\im\Delta_{k}^{-}
\end{aligned}
\\
\mathcal{H}^{k}=\ker\Delta_{k}=Z_{k}\cap Z^{k}=\left(B_{k}\oplus B^{k}\right)^{\perp}\cong H_{k}\cong H^{k}
\end{gather*}
The isomorphism between harmonic functions, homology and cohomology,
which is sometimes called the discrete Hodge theorem, was first observed
in \cite{Eck44}. In a similar manner, there is a ``discrete Hodge
decomposition'' 
\begin{equation}
\Omega^{k}=\rlap{\ensuremath{\overbrace{\phantom{B_{k}\oplus\mathcal{H}^{k}}}^{Z_{k}}}}B_{k}\oplus\underbrace{\mathcal{H}^{k}\oplus B^{k}}_{Z^{k}},\label{eq:hodge-decomp}
\end{equation}
and all the Laplacians decompose with respect to it. All of these
claims follow by linear algebra, using the fact that $\Omega^{k}$
is finite-dimensional (see \cite[§2]{parzanchevski2012isoperimetric}
for the details). For infinite complexes the situation is more involved,
and is addressed in §\ref{sub:General-dimension}.

\subsection{The upper Laplacian spectrum}

In this section we study the spectrum of the upper Laplacian $\Delta^{+}$
of a finite complex $X$. First note that as $\Delta^{+}=\partial_{d}\partial_{d}^{*}$,
its spectrum is non-negative. Furthermore, zero is obtained precisely
on closed forms, i.e.\ $\ker\Delta^{+}=Z^{d-1}$. The space of closed
forms always contains the exact forms, $B^{d-1}=\im\partial_{d-1}^{*}$,
which are considered the \emph{trivial zeros} in the spectrum of $\Delta^{+}$.
The existence of \emph{nontrivial} \emph{zeros} in the spectrum of
$\Delta^{+}$, i.e.\ closed forms which are not exact, indicates
the existence of a nontrivial homology. This leads to the following
definition:
\begin{defn}
\label{The-spectral-gap}The \emph{spectral gap}\textbf{\emph{ }}of
a finite $d$-dimensional complex $X$, denoted $\lambda\left(X\right)$,
is 
\[
\lambda(X)=\min\Spec\left(\Delta^{+}|_{Z_{d-1}}\right)=\min\Spec\left(\Delta|_{Z_{d-1}}\right).
\]
The \emph{essential gap} of $X$, denoted $\tilde{\lambda}(X)$, is
\[
\widetilde{\lambda}(X)=\min\Spec\left(\Delta^{+}|_{B_{d-1}}\right)=\min\Spec\left(\Delta|_{B_{d-1}}\right)
\]
(the transition from $\Delta$ to $\Delta^{+}$ follows from the fact
that $\Delta^{-}$ vanishes on $Z_{d-1}$.)
\end{defn}
Since $\lambda$ vanishes precisely when the $\left(d-1\right)$-homology
of $X$ is nontrivial, it should be thought of as giving a quantitative
measure for the ``triviality of the homology''. For example, in
graphs, having $\lambda\left(X\right)$ far away from zero is a measure
of high-connectedness, or ``high triviality of the $0^{\mathrm{th}}$-homology''.

In contrast, $\widetilde{\lambda}$ never vanishes, as $B_{d-1}=\left(Z^{d-1}\right)^{\perp}=\left(\ker\Delta^{+}\right)^{\perp}$.
If the $\left(d-1\right)$-homology is nontrivial then $\lambda=\widetilde{\lambda}$,
so that $\widetilde{\lambda}$ is only of additional interest when
the homology vanishes. In a disconnected graph $\widetilde{\lambda}$
controls the mixing rate as $\lambda$ does for a connected graph,
and we will see that the same happens in higher dimension (see (\ref{eq:convergence-rate-tilde})).

Until now we have studied one extremity of $\Spec\Delta^{+}$. The
other side relates to the following definition:
\begin{defn}
\label{def:A-disorientation}A \emph{disorientation} of a $d$-complex
$X$ is a choice of orientation $X_{+}^{d}$ of its $d$-cells, so
that whenever $\sigma,\sigma'\in X_{+}^{d}$ intersect in a $\left(d-1\right)$-cell
they induce the same orientation on it. If $X$ has a disorientation
it is said to be \emph{disorientable}.
\end{defn}
\begin{rems*}$ $ \vspace{-6pt}
\begin{enumerate}
\item A disorientable $1$-complex is precisely a bipartite graph, and thus
disorientability should be thought of as a high-dimensional analogue
of bipartiteness. Another natural analogue is ``$\left(d+1\right)$-partiteness'':
having some partition $A_{0},\ldots,A_{d}$ of $V$ so that every
$d$-cell contains one vertex from each $A_{i}$. A $\left(d+1\right)$-partite
complex is easily seen to be disorientable, but the opposite does
not necessarily hold for $d\geq2$.
\item Notice the similarity to the notion of orientability: a $d$-complex
is orientable if there is a choice of orientations $ $of its $d$-cells,
so that cells intersecting in a codimension one cell induce \emph{opposite}
orientations on it. However, orientability implies that $\left(d-1\right)$-cells
have degrees at most two, where as disorientability impose no such
restrictions. Note that a complex can certainly be both orientable
and disorientable (e.g.\ Figure \ref{fig:two-periodic}(a)).
\end{enumerate}
\end{rems*}
\begin{prop}
\label{prop:Laplacian-spec}Let $X$ be a finite complex of dimension
$d$.
\begin{enumerate}
\item \textup{$\Spec\Delta^{+}\left(X\right)$} is the disjoint union of
$\Spec\Delta^{+}\left(X_{i}\right)$ where $X_{i}$ are the $\left(d-1\right)$-components
of $X$.
\item The spectrum of $\Delta^{+}=\Delta^{+}\left(X\right)$ is contained
in $\left[0,d+1\right]$.
\item Zero is achieved on the closed $\left(d-1\right)$-forms, $Z^{d-1}$.
\item If $X$ is $\left(d-1\right)$-connected, then $d+1$ is in the spectrum
iff $X$ is disorientable, and is achieved on the boundaries of disorientations
(see (\ref{eq:disorientation-form})).
\end{enumerate}
\end{prop}
\begin{proof}
\emph{(1)} follows from the observation that $\Delta^{+}$ decomposes
w.r.t.\ the decomposition $\Omega^{d-1}\left(X\right)=\bigoplus_{i}\Omega^{d-1}\left(X_{i}\right)$.
We have already seen \emph{(3)}, and the fact that the spectrum is
nonnegative. Now, assume that $\Delta^{+}f=\lambda f$. Choose $\sigma\in X^{d-1}$
which maximize $\frac{\left|f\left(\sigma\right)\right|}{\deg\left(\sigma\right)}$.
By (\ref{eq:LaplacianUp}), 
\[
\lambda f\left(\sigma\right)=\left(\Delta^{+}f\right)\left(\sigma\right)=f\left(\sigma\right)-\sum_{\sigma'\sim\sigma}\frac{f\left(\sigma'\right)}{\deg\left(\sigma'\right)}
\]
and therefore
\[
\left|\lambda f\left(\sigma\right)\right|\leq\left|f\left(\sigma\right)\right|+\sum_{\sigma'\sim\sigma}\frac{\left|f\left(\sigma'\right)\right|}{\deg\left(\sigma'\right)}\leq\left(d+1\right)\left|f\left(\sigma\right)\right|,
\]
(since $\#\left\{ \sigma'\,\middle|\,\sigma'\sim\sigma\right\} =d\deg\sigma$),
hence $\lambda\leq d+1$ and \emph{(2)} is obtained. 

Next, assume that $X$ is $\left(d-1\right)$-connected and that $X_{+}^{d}$
is a disorientation. Define
\begin{equation}
F\left(\tau\right)=\begin{cases}
1 & \tau\in X_{+}^{d}\\
-1 & \tau\in X_{\pm}^{d}\backslash X_{+}^{d}
\end{cases},\label{eq:disorientation-form}
\end{equation}
and $f=\partial_{d}F$. For any $\sigma\in X_{\pm}^{d-1}$, there
exists some vertex $v$ with $v\triangleleft\sigma$ (since $X$ is
uniform). Furthermore, by the assumption on $X_{+}^{d}$, if $v\triangleleft\sigma$
and $v'\triangleleft\sigma$ for vertices $v,v'$ then $v\sigma\in X_{+}^{d}$
if and only if $v'\sigma\in X_{+}^{d}$, and thus 
\[
f\left(\sigma\right)=\left(\partial_{d}F\right)\left(\sigma\right)=\sum_{v\triangleleft\sigma}F\left(v\sigma\right)=\deg\left(\sigma\right)F\left(\tau\right)
\]
where $\tau$ is any $d$-cell containing $\sigma$. If $\sigma$
and $\sigma'$ are neighboring $\left(d-1\right)$-faces in $X_{\pm}^{d-1}$,
then by definition, for some $\tau\in X_{\pm}^{d}$, $\sigma$ is
a face of $\tau$ and $\sigma'$ is a face of $\overline{\tau}$,
so that 
\[
\frac{f\left(\sigma\right)}{\deg\sigma}+\frac{f\left(\sigma'\right)}{\deg\sigma'}=F\left(\tau\right)+F\left(\overline{\tau}\right)=0,
\]
and consequently for any $\sigma\in X_{\pm}^{d-1}$ 
\[
\left(\Delta^{+}f\right)\left(\sigma\right)=f\left(\sigma\right)-\sum_{\sigma'\sim\sigma}\frac{f\left(\sigma'\right)}{\deg\left(\sigma'\right)}=f\left(\sigma\right)-\sum_{\sigma'\sim\sigma}\frac{-f(\sigma)}{\deg\left(\sigma\right)}=\left(d+1\right)f\left(\sigma\right),
\]
so that $f$ is a $\Delta^{+}$-eigenform with eigenvalue $d+1$.

In the other direction, assume that $X$ is $\left(d-1\right)$-connected
and that $\Delta^{+}f=\left(d+1\right)f$ for some $f\in\Omega^{d-1}\left(X\right)\backslash\left\{ 0\right\} $.
Fix some $\widetilde{\sigma}\in X_{\pm}^{d-1}$ which maximize $\frac{\left|f\left(\sigma\right)\right|}{\deg\sigma}$,
normalize $f$ so that $\frac{\left|f\left(\widetilde{\sigma}\right)\right|}{\deg\widetilde{\sigma}}=1$,
and define
\[
F=\frac{\partial_{d}^{*}f}{d+1},\quad X_{+}^{d}=\left\{ \tau\in X_{\pm}^{d}\,\middle|\, F\left(\tau\right)>0\right\} .
\]
We have $f=\frac{\Delta^{+}f}{d+1}=\frac{\partial_{d}\partial_{d}^{*}f}{d+1}=\partial_{d}F$
by assumption, and we proceed to show that $X_{+}^{d}$ is a disorientation
with $F$ the corresponding form as in (\ref{eq:disorientation-form}).
By the definition of $\Delta^{+}$  
\[
\deg\widetilde{\sigma}=\left|f\left(\widetilde{\sigma}\right)\right|=\frac{1}{d}\left|\sum_{\sigma\sim\widetilde{\sigma}}\frac{f\left(\sigma\right)}{\deg\left(\sigma\right)}\right|\leq\frac{1}{d}\sum_{\sigma\sim\widetilde{\sigma}}\frac{\left|f\left(\sigma\right)\right|}{\deg\left(\sigma\right)}\leq\frac{1}{d}\sum_{\sigma\sim\widetilde{\sigma}}1=\deg\widetilde{\sigma},
\]
so that $\frac{\left|f\left(\sigma\right)\right|}{\deg\sigma}=1$
for every $\sigma\sim\widetilde{\sigma}$. Continuing in this manner,
$\left(d-1\right)$-connectedness implies that $\frac{\left|f\left(\sigma\right)\right|}{\deg\sigma}\equiv1$
on all $X_{\pm}^{d}$. Using again the definition of $\Delta^{+}$,
for any $\sigma$ in $X_{\pm}^{d}$ 
\[
\frac{f\left(\sigma\right)}{\deg\sigma}=-\frac{1}{\deg\sigma\cdot d}\sum_{\sigma'\sim\sigma}\frac{f\left(\sigma'\right)}{\deg\left(\sigma'\right)}.
\]
Since the r.h.s is an average over terms whose absolute value is that
of the l.h.s this gives $\frac{f\left(\sigma'\right)}{\deg\sigma'}=-\frac{f\left(\sigma\right)}{\deg\sigma}$
whenever $\sigma\sim\sigma'$, hence 
\[
F\left(\tau\right)=\frac{1}{d+1}\sum_{i=0}^{d}\frac{\left(-1\right)^{i}f\left(\tau\backslash\tau_{i}\right)}{\deg\left(\tau\backslash\tau_{i}\right)}=\frac{f\left(\tau\backslash\tau_{0}\right)}{\deg\left(\tau\backslash\tau_{0}\right)}
\]
is always of absolute value one. Furthermore, if $\tau,\tau'\in X_{\pm}^{d}$
intersect in a face $\sigma$ and induce opposite orientations on
it, then $\tau=v\sigma$ and $\tau'=\overline{v'\sigma}$ for some
vertices $v,v'$, hence 
\[
F\left(\tau\right)=F\left(v\sigma\right)=\frac{f\left(\sigma\right)}{\deg\sigma}=F\left(v'\sigma\right)=-F\left(\overline{v'\sigma}\right)=-F\left(\tau'\right)
\]
which concludes the proof.
\end{proof}

\subsection{\label{sub:Walk-and-spectrum}Walk and spectrum}

The $\left(d-1\right)$-walk defined in §\ref{sub:The-walk} is related
to the Laplacians from §\ref{sub:Simplicial-complexes-and} as follows:
\begin{prop}
\label{Prop:random_Laplacian_connection}Observe the $p$-lazy $\left(d-1\right)$-walk
on $X$ starting at $\sigma_{0}\in X_{\pm}^{d-1}$. Then
\begin{enumerate}
\item The transition operator $A=A\left(X,p\right)$ is given by 
\[
A=\frac{p(d-1)+1}{d}\cdot I-\frac{1-p}{d}\cdot\Delta^{+},
\]
so that 
\[
\mathcal{E}_{n}^{\sigma_{0}}=A^{n}\mathcal{E}_{0}^{\sigma_{0}}=\left(\frac{p(d-1)+1}{d}\cdot I-\frac{1-p}{d}\cdot\Delta^{+}\right)^{n}\mathcal{E}_{0}^{\sigma_{0}}.
\]

\item The spectrum of $A$ is contained in $\left[2p-1,\frac{p\left(d-1\right)+1}{d}\right]$,
with $2p-1$ achieved by disorientations, and $\frac{p\left(d-1\right)+1}{d}$
by closed forms ($Z^{d-1}$). 
\item The expectation process satisfies
\begin{equation}
\frac{1}{\sqrt{K_{d-2}K_{d-1}}}\left(\frac{p\left(d-1\right)+1}{d}\right)^{n}\leq\left\Vert \mathcal{E}_{n}^{\sigma_{0}}\right\Vert \leq\max\left(\left|2p-1\right|,\frac{p\left(d-1\right)+1}{d}\right)^{n}\label{eq:growth-exp-proc}
\end{equation}
where $K_{j}$ is the maximal degree of a $j$-cell in $X$.
\end{enumerate}
\end{prop}
\begin{proof}
\emph{(1)} follows trivially from (\ref{eq:transition-operator})
and (\ref{eq:LaplacianUp}), and Proposition \ref{prop:Laplacian-spec}
then implies \emph{(2)}. The upper bound in \emph{(3)} follows from
\emph{(2)} by $\mathcal{E}_{n}^{\sigma_{0}}=A^{n}\mathcal{E}_{0}^{\sigma_{0}}$
and $\left\Vert \mathcal{\mathcal{E}}_{0}^{\sigma_{0}}\right\Vert =\left\Vert \mathbbm{1}_{\sigma_{0}}\right\Vert =\frac{1}{\sqrt{\deg\sigma_{0}}}\leq1$.
For the lower bound, let $v$ be a vertex in $\sigma_{0}$, and $\sigma_{0},\ldots,\sigma_{k}$
the $\left(d-1\right)$-cells containing $\sigma_{0}\backslash v$.
Define $f=\partial_{d}^{*}\mathbbm{1}_{\sigma_{0}\backslash v}=\sum_{i=0}^{k}\deg\sigma_{i}\cdot\mathbbm{1}_{\sigma_{i}}$,
so that $f\in Z^{d-1}$ and $\left\Vert f\right\Vert ^{2}=\sum_{i=0}^{k}\deg\sigma_{i}\leq K_{d-2}K_{d-1}$.
Since $\Delta^{+}$ decomposes w.r.t.\ the orthogonal sum $\Omega^{d-1}=Z^{d-1}\oplus B_{d-1}$
so does $A=\frac{p(d-1)+1}{d}\cdot I-\frac{1-p}{d}\cdot\Delta^{+}$,
hence  by \emph{(2) }
\begin{gather*}
\left\Vert \mathcal{E}_{n}^{\sigma_{0}}\right\Vert =\left\Vert A^{n}\mathbbm{1}_{\sigma_{0}}\right\Vert \geq\left({\textstyle \frac{p\left(d-1\right)+1}{d}}\right)^{n}\left\Vert \mathbb{P}_{Z^{d-1}}\left(\mathbbm{1}_{\sigma_{0}}\right)\right\Vert \geq\left({\textstyle \frac{p\left(d-1\right)+1}{d}}\right)^{n}\left|\left\langle {\textstyle \frac{f}{\left\Vert f\right\Vert }},\mathbbm{1}_{\sigma_{0}}\right\rangle \right|\\
=\left({\textstyle \frac{p\left(d-1\right)+1}{d}}\right)^{n}\frac{\left|f\left(\sigma_{0}\right)\right|}{\left\Vert f\right\Vert \deg\sigma_{0}}\geq\frac{1}{\sqrt{K_{d-2}K_{d-1}}}\left({\textstyle \frac{p\left(d-1\right)+1}{d}}\right)^{n}.
\end{gather*}

\end{proof}
This proposition leads to the connection between the asymptotic behavior
of the $\left(d-1\right)$-walk and the homology and spectrum of the
complex:
\begin{thm}
\label{thm:walk-and-spec}Let $\mathcal{\widetilde{E}}_{n}^{\sigma}$
be the normalized expectation process associated with the $p$-lazy
$\left(d-1\right)$-walk on $X$ starting from $\sigma$ (see Definitions
\ref{def:The--lazy--walk}, \ref{def:norm-exp-proc}). Then $\widetilde{\mathcal{E}}_{\infty}^{\sigma}=\lim_{n\rightarrow\infty}\mathcal{\widetilde{E}}_{n}^{\sigma}$
exists and satisfies the following:
\begin{enumerate}
\item If $\frac{d-1}{3d-1}<p<1$, then $\widetilde{\mathcal{E}}_{\infty}^{\sigma}$
is exact for every starting point $\sigma$ if and only if \textup{$H_{d-1}(X)=0$.}%
\footnote{Note that the first value of $p$ for which the homology can be studied
via the walk in every dimension is $p=\frac{1}{3}$.%
}\textup{ If furthermore $p\geq\frac{1}{2}$ then}
\begin{equation}
\dist\left(\widetilde{\mathcal{E}}_{n}^{\sigma},B^{d-1}\right)=O\left(\left(1-\frac{1-p}{p\left(d-1\right)+1}\lambda\left(X\right)\right)^{n}\right).\label{eq:convergence-rate}
\end{equation}

\item More generally, the dimension of $H_{d-1}\left(X\right)$ equals the
dimension of $\mathrm{Span}\left\{ \mathbb{P}_{Z_{d-1}}\left(\widetilde{\mathcal{E}}_{\infty}^{\sigma}\right)\,\middle|\,\sigma\in X^{d-1}\right\} $.
\item If $p=\frac{d-1}{3d-1}$ then $\widetilde{\mathcal{E}}_{\infty}^{\sigma}$
is exact for all $\sigma$ if and only if $X$ has a trivial $\left(d-1\right)$-homology
and no disorientable $\left(d-1\right)$-components.
\item More generally, if $\frac{d-1}{3d-1}<p<1$ then $\widetilde{\mathcal{E}}_{\infty}^{\sigma}$
is closed, and likewise for $p=\frac{d-1}{3d-1}$, unless $X$ has
a disorientable $\left(d-1\right)$-component. If $p\geq\frac{1}{2}$
then 
\begin{equation}
\dist\left(\widetilde{\mathcal{E}}_{n}^{\sigma},Z^{d-1}\right)=O\left(\left(1-\frac{1-p}{p\left(d-1\right)+1}\widetilde{\lambda}\left(X\right)\right)^{n}\right).\label{eq:convergence-rate-tilde}
\end{equation}

\end{enumerate}
\end{thm}
\begin{proof}
\textbf{}

\textbf{Case $\boldsymbol{\left(i\right)\:-\:\frac{d-1}{3d-1}<p<1}$}:
We have $\left|2p-1\right|<\frac{p\left(d-1\right)+1}{d}$, so that
$\left\Vert A\right\Vert =\max\Spec A=\frac{p\left(d-1\right)+1}{d}$.
Thus, 
\[
\Spec A\big|_{B_{d-1}}\subseteq\left[2p-1,\frac{p\left(d-1\right)+1}{d}\right)\subseteq\left(-\left\Vert A\right\Vert ,\left\Vert A\right\Vert \right).
\]
Since $A$ decomposes w.r.t.\ $\Omega^{d-1}=B_{d-1}\oplus Z^{d-1}$,
and $A\big|_{Z^{d-1}}=\left\Vert A\right\Vert \cdot I\big|_{Z^{d-1}}$,
this means that $\left(\frac{A}{\left\Vert A\right\Vert }\right)^{n}$
converges to the orthogonal projection $\mathbb{P}_{Z^{d-1}}$. Now
$\widetilde{\mathcal{E}}_{n}^{\sigma}=\left(\frac{d}{p\left(d-1\right)+1}\right)^{n}\mathcal{E}_{n}^{\sigma}=\left(\frac{A}{\left\Vert A\right\Vert }\right)^{n}\mathcal{E}_{0}^{\sigma}$,
which shows that
\begin{equation}
\widetilde{\mathcal{E}}_{\infty}^{\sigma}=\mathbb{P}_{Z^{d-1}}\left(\widetilde{\mathcal{E}}_{0}^{\sigma}\right)=\mathbb{P}_{Z^{d-1}}\left(\mathcal{E}_{0}^{\sigma}\right)=\mathbb{P}_{Z^{d-1}}\left(\one_{\sigma}\right).\label{eq:inf-is-cocycle}
\end{equation}
In particular $\widetilde{\mathcal{E}}_{\infty}^{\sigma}$ is closed,
so that if the homology of $X$ is trivial then it is exact. On the
other hand, assume that $\widetilde{\mathcal{E}}_{\infty}^{\sigma}$
is exact for all $\sigma$: then 
\[
\widetilde{\mathcal{E}}_{\infty}^{\sigma}=\mathbb{P}_{Z^{d-1}}\left(\mathcal{E}_{0}^{\sigma}\right)=\mathbb{P}_{Z^{d-1}}\left(\mathbbm{1}_{\sigma}\right)=\mathbb{P}_{B^{d-1}}\left(\mathbbm{1}_{\sigma}\right)+\mathbb{P}_{\mathcal{H}^{d-1}}\left(\mathbbm{1}_{\sigma}\right)
\]
so that $\mathbb{P}_{\mathcal{H}^{d-1}}\left(\mathbbm{1}_{\sigma}\right)=0$
by (\ref{eq:hodge-decomp}). As $\left\{ \one_{\sigma}\right\} $
span $\Omega^{d-1}$, this shows that $H_{d-1}\cong\mathcal{\mathcal{H}}^{d-1}=0$.
To further understand the dimension of the homology, observe that
\[
\mathrm{Span}\left\{ \mathbb{P}_{Z_{d-1}}\left(\widetilde{\mathcal{E}}_{\infty}^{\sigma}\right)\,\middle|\,\sigma\in X^{d-1}\right\} =\mathcal{H}^{d-1}\left(X\right),
\]
which follows from 
\[
\mathbb{P}_{Z_{d-1}}\left(\widetilde{\mathcal{E}}_{\infty}^{\sigma}\right)=\mathbb{P}_{Z_{d-1}}\left(\mathbb{P}_{Z^{d-1}}\left(\one_{\sigma}\right)\right)=\mathbb{P}_{\mathcal{H}^{d-1}}\left(\one_{\sigma}\right).
\]
If $p\geq\frac{1}{2}$ then we know not only that $\left\Vert A\right\Vert =\max\Spec A$
but also that $\left\Vert A\big|_{Z_{d-1}}\right\Vert =\max\Spec\left(A\big|_{Z_{d-1}}\right)$,
which allows us to say more. In this case $A$ is positive semidefinite,
so that (\ref{eq:convergence-rate-tilde}) follows by 
\begin{multline*}
\left\Vert \left(\frac{d}{p\left(d-1\right)+1}A\right)^{n}-\mathbb{P}_{Z^{d-1}}\right\Vert =\left\Vert \left(\frac{d}{p\left(d-1\right)+1}A\big|_{B_{d-1}}\right)^{n}\right\Vert \\
=\left\Vert \left(I-\frac{1-p}{p\left(d-1\right)+1}\cdot\Delta^{+}\right)^{n}\big|_{B_{d-1}}\right\Vert =\left(1-\frac{1-p}{p\left(d-1\right)+1}\widetilde{\lambda}\left(X\right)\right)^{n},
\end{multline*}
which gives (\ref{eq:convergence-rate}) as well when the homology
is trivial. 

\textbf{Case $\boldsymbol{\left(ii\right)\:-\: p=\frac{d-1}{3d-1}}$}:
Now, $\left|2p-1\right|=\frac{p\left(d-1\right)+1}{d}=\left\Vert A\right\Vert $.
If $X$ has no disorientable $\left(d-1\right)$-components then again
$\Spec A\big|_{B_{d-1}}\subseteq\left(-\left\Vert A\right\Vert ,\left\Vert A\right\Vert \right)$,
which gives (\ref{eq:inf-is-cocycle}), and everything is as before.
On the other hand, let us assume that $\widetilde{\mathcal{E}}_{\infty}^{\sigma}$
is closed for all $\sigma$. Denoting by $\Omega_{\lambda}^{d-1}$
the $\lambda$-eigenspace of $A$, now $\left(\frac{d}{p\left(d-1\right)+1}A\right)^{2n}$
converges to $\mathbb{P}_{Z^{d-1}}+\mathbb{P}_{\Omega_{2p-1}^{d-1}}$
($\Delta^{+}$ is diagonalizable and consequently so is $A$). Since
$\widetilde{\mathcal{E}}_{\infty}^{\sigma}$ is closed this shows
that $\mathbb{P}_{\Omega_{2p-1}^{d-1}}\left(\mathbbm{1}_{\sigma}\right)=0$,
and consequently that $\Omega_{2p-1}^{d-1}=0$, i.e.\ X has no disorientable
$\left(d-1\right)$-components.
\end{proof}
\nopagebreak \begin{rems*}$ $\vspace{-6pt} \nopagebreak
\begin{enumerate}
\item \nopagebreak The study of complexes via $\left(d-1\right)$-walk
gives a conceptual reason to the fact that the high-dimensional case
is harder than that of graphs: while graphs are studied by the evolution
of probabilities, analogue properties of high-dimensional complexes
are reflected in the expectation process. As this is given by the
difference of two probability vectors, it is much harder to analyze.
Several examples of this appear in the open questions in §\ref{sec:Open-Questions}.
\item In order to study the connectedness of a graph it is enough to observe
the walk starting at one vertex. If $\mathbf{p}_{\infty}^{v_{0}}$
is not exact (i.e.\ not proportional to the degree function) for
even one $v_{0}$, then the graph is necessarily disconnected. In
general dimension, however, this is not enough: there are complexes
(even $\left(d-1\right)$-connected ones!) with nontrivial $\left(d-1\right)$-homology,
such that $\widetilde{\mathcal{E}}_{\infty}^{\sigma}$ is exact for
a carefully chosen $\sigma$.
\item If one starts the process with a general initial distribution $\mathbf{p}_{0}$\textbf{
}instead of the Dirac probability $\one_{\sigma}$, then Theorem \ref{thm:walk-and-spec}
holds for the corresponding expectation process (i.e.\ $\mathcal{E}_{0}\left(\sigma\right)=\mathbf{p}_{0}\left(\sigma\right)-\mathbf{p}_{0}\left(\overline{\sigma}\right)$,
$\mathcal{E}_{n+1}=A\mathcal{E}_{n}$). Furthermore, in these settings
a disorientable component corresponds to a distribution for which
$\widetilde{\mathcal{E}}_{n}$ is $2$-periodic for $p=\frac{d-1}{3d-1}$
(see Figure \ref{fig:two-periodic}(a)); a nontrivial homology corresponds
to a distribution which induces a stationery non-exact $\widetilde{\mathcal{E}}_{n}$
for $p\geq\frac{d-1}{3d-1}$ (see Figure \ref{fig:two-periodic}(b)).
\end{enumerate}
\end{rems*}

\begin{figure}[h]
\centering{}%
\begin{minipage}[t]{0.45\columnwidth}%
\begin{center}
\includegraphics{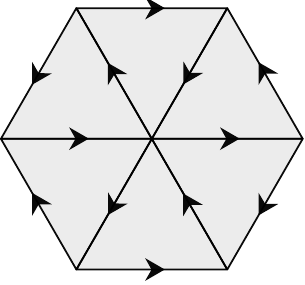}\\
(a)
\par\end{center}%
\end{minipage}%
\begin{minipage}[t]{0.45\columnwidth}%
\begin{center}
\includegraphics{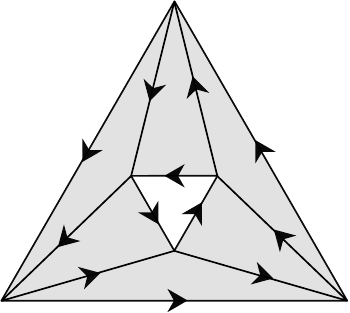}\\
(b)
\par\end{center}%
\end{minipage}\caption{\label{fig:two-periodic}Two distributions on the edges of $2$-complexes
(the orientations drawn have uniform probability, and their inverses
probability zero). (a) is a distribution for which $\widetilde{\mathcal{E}}_{n}=\left(\frac{5}{3}\right)^{n}\mathcal{E}_{n}$
is $2$-periodic under the $\frac{1}{5}$-lazy walk; (b) is a distribution
for which $\widetilde{\mathcal{E}}_{n}$ is stable and non exact (under
the $p$-lazy walk, $p>\frac{1}{5}$).}
\end{figure}

\section{Infinite complexes}

\subsection{\label{sub:Infinite-graphs}Infinite graphs}

We move to the case of infinite complexes, starting with infinite
graphs. Recall that for a finite graph $G=\left(V,E\right)$, we observed
$\Delta^{+}=\Delta^{+}\left(G\right)$, and defined 
\[
\lambda\left(G\right)=\min\Spec\Delta^{+}\big|_{\left(B^{0}\right)^{\bot}}=\min\Spec\Delta^{+}\big|_{Z_{0}}.
\]
In contrast, when $G$ is an infinite graph (i.e.\ $\left|V\right|=\infty$)
one usually restrict his attention to $L^{2}\left(V\right)$ and define
\begin{equation}
\lambda\left(G\right)=\min\Spec\Delta^{+}\big|_{L^{2}\left(V\right)}.\label{eq:gap-infinite-graph}
\end{equation}
Here there is no restriction to $Z_{0}$, nor to $\left(B^{0}\right)^{\bot}$.
These two spaces, which coincide in the finite dimensional case, since
\begin{equation}
Z_{0}=\ker\partial_{0}=\left(\im\partial_{0}^{*}\right)^{\bot}=\left(B^{0}\right)^{\bot},\label{eq:Z0-B0-coincide}
\end{equation}
fail to do so in the infinite settings. First, $Z_{0}$ is not even
defined, as $\left(\partial_{0}f\right)\left(\varnothing\right)=\sum_{v\in V}f\left(v\right)$
has no meaning for general $f\in L^{2}\left(V\right)$. One can observe
$B^{0}=\im\delta_{0}$, taking (\ref{eq:coboundary-operator}) as
the definition of $\delta_{0}$ (as $\partial_{0}$ is not defined).
With this definition, $B^{0}$ consists of the scalar multiples of
the degree function. Since these are never in $L^{2}\left(V\right)$
(assuming as always that there are no isolated vertices), we have
$B^{0}=0$ and $\left(B^{0}\right)^{\bot}=L^{2}\left(V\right)$, justifying
(\ref{eq:gap-infinite-graph}). Another thing which fails here is
the chain complex property $\partial_{0}\partial_{1}=0$: there may
exist $f\in\Omega^{1}\left(G\right)$ such that $\partial_{0}\partial_{1}f$
is defined and nonzero. For example, take $V=\mathbb{Z}$, $E=\left\{ \left\{ i,i\!+\!1\right\} \,\middle|\, i\in\mathbb{Z}\right\} $,
and $f\left(\left[i,i\!+\!1\right]\right)=\begin{cases}
0 & i<0\\
1 & 0\leq i
\end{cases}$. Here $\partial_{1}f=\mathbbm{1}_{0}$, and thus $\left(\partial_{0}\partial_{1}f\right)\left(\varnothing\right)=1$.
If $G$ is transient, e.g.\ the $\mathbb{Z}^{3}$ graph, or a $k$-regular
tree with $k\geq3$, there are even such $f$ in $L^{2}$ - see §\ref{sub:Amenability-transiency}.

\subsection{\label{sub:General-dimension}Infinite complexes of general dimension }

For a complex $X$ of dimension $d$, and $-1\leq k\leq d$, we denote
\[
\Omega_{L^{2}}^{k}=\Omega_{L^{2}}^{k}\left(X\right)=\left\{ f\in\Omega^{k}\left(X\right)\,\middle|\,\left\Vert f\right\Vert ^{2}<\infty\right\} \subseteq\Omega^{k}\left(X\right),
\]
where we recall that 
\[
\left\Vert f\right\Vert ^{2}=\sum_{\sigma\in X^{k}}w\left(\sigma\right)f\left(\sigma\right)^{2}=\begin{cases}
\sum_{\sigma\in X^{k}}f\left(\sigma\right)^{2} & k\neq d-1\\
\sum_{\sigma\in X^{k}}\frac{f\left(\sigma\right)^{2}}{\deg\sigma} & k=d-1
\end{cases}.
\]

Whenever referring to infinite complexes, the domain of all operators
(i.e.\ $\partial,\delta,\Delta^{+},\Delta^{-},\Delta$) is assumed
to be $\Omega_{L^{2}}^{k}$, unless explicitly stated that we are
interested in $\Omega^{k}$. 

Let us examine these operators. We shall always assume that the $\left(d-1\right)$-cells
in $X$ have globally bounded degrees, which ensures that the boundary
and coboundary operators $\partial_{d}:\Omega^{d}\rightarrow\Omega^{d-1}$,
$\delta_{d}:\Omega^{d-1}\rightarrow\Omega^{d}$ are defined, bounded,
and adjoint to one another, so that $\Delta^{+}=\partial_{d}\delta_{d}=\partial_{d}\partial_{d}^{*}$
is bounded and self-adjoint. We do not assume that the degrees in
other dimensions are bounded, as this would rule out infinite graphs,
for example. This means that in general $\delta_{k}$ does not take
$\Omega_{L^{2}}^{k-1}$ into $\Omega_{L^{2}}^{k}$ but only to $\Omega^{k}$,
and $\partial_{k}$ need not even be defined. In particular, one cannot
always define $\Delta^{-}$.

The cochain property $\delta_{k}\delta_{k-1}=0$ always holds, whereas
in general $\partial_{k-1}\partial_{k}\left(f\right)$ can be defined
and nonzero for some $f\in\Omega_{L^{2}}^{k}$. If the degrees of
$\left(k-1\right)$-cells are bounded, then $\delta_{k}$ and $\partial_{k}$
are bounded and $\delta_{k}=\partial_{k}^{*}$. Thus, if the degrees
of $\left(k-1\right)$-cells and $\left(k-2\right)$-cells are globally
bounded one has $\partial_{k-1}\partial_{k}=\left(\delta_{k}\delta_{k-1}\right)^{*}=0^{*}=0$
as well.

In contrast with infinite graphs, an infinite $d$-complex may have
$\left(d-2\right)$-cells of finite degree, so that the image of $\delta_{d-1}$
may contain $L^{2}$-coboundaries. For example, if $v$ is a vertex
of finite degree in an infinite triangle complex, then the ``star''
$\delta_{1}\mathbbm{1}_{v}$ is an $L^{2}$-coboundary. We denote
by $B^{d-1}$ the $L^{2}$-coboundaries, i.e.\ $B^{d-1}=\im\delta_{d-1}\cap\Omega_{L^{2}}^{d-1}$.
In order to avoid trivial zeros in the spectrum of $\Delta^{+}$,
we define $Z_{d-1}=\left(B^{d-1}\right)^{\bot}$ (the orthogonal complement
in $\Omega_{L^{2}}^{d-1}$), and 
\[
\lambda\left(X\right)=\min\Spec\Delta^{+}\big|_{Z_{d-1}}.
\]
We stress out that $Z_{d-1}$ is not necessarily the kernel of $\partial_{d-1}$
(which is not even defined in general). If the $\left(d-2\right)$-degrees
are globally bounded then $\partial_{d-1}$ is defined and dual to
$\delta_{d-1}$, and this gives inclusion in one direction:
\begin{equation}
Z_{d-1}=\left(B^{d-1}\right)^{\bot}=\left(\im\delta_{d-1}\right)^{\bot}\subseteq\ker\partial_{d-1}.\label{eq:d-2-finite-cycles}
\end{equation}
For finite complexes there is an equality here (as in (\ref{eq:Z0-B0-coincide}))
due to dimension considerations.

\medskip{}

In infinite graphs we had $B^{0}=0$, $Z_{0}=\Omega_{L^{2}}^{0}=L^{2}\left(V\right)$
and $\lambda=\min\Spec\Delta^{+}\big|_{L^{2}\left(V\right)}$. The
following lemma shows that this happens whenever all $\left(d-2\right)$-cells
are of infinite degree:
\begin{lem}
\label{lem:d-2-of-inf-deg}If $X$ is a $d$-complex whose $\left(d-2\right)$-cells
are all of infinite degree, then $B^{d-1}=0$ and thus $\lambda\left(X\right)=\min\Spec\Delta^{+}$.\end{lem}
\begin{proof}
Let $f\in\Omega^{d-2}$ be such that $\delta_{d-1}f\in\Omega_{L^{2}}^{d-1}\backslash\left\{ 0\right\} $.
Choose $\tau\in X_{\pm}^{d-2}$ for which $f\left(\tau\right)>0$,
and let $\left\{ \sigma_{i}\right\} _{i=1}^{\infty}$ be a sequence
of $\left(d-1\right)$-cells containing $\tau$. Since $\sum_{i=1}^{\infty}\left(\delta_{d-1}f\right)^{2}\left(\sigma_{i}\right)\leq\left\Vert \delta_{d-1}f\right\Vert ^{2}<\infty$,
for infinitely many $i$ we have $\left|\left(\delta_{d-1}f\right)\left(\sigma_{i}\right)\right|\leq\frac{f\left(\tau\right)}{2}$.
Since $\tau$ contributes $f\left(\tau\right)$ to $\left(\delta_{d-1}f\right)\left(\sigma_{i}\right)$,
one of the other faces of $\sigma_{i}$ must be of absolute value
at least $\frac{f\left(\tau\right)}{2\left(d-1\right)}$. Since these
faces are all different $\left(d-2\right)$-cells (if $\sigma_{i}\cap\sigma_{j}$
contains $\tau$ and another $\left(d-2\right)$-cell, then $\sigma_{i}=\sigma_{j}$),
we have $\left\Vert f\right\Vert =\infty$.
\end{proof}

\subsection{\label{sub:Example---arboreal}Example - arboreal complexes}
\begin{defn}
We say that a $d$-complex is \emph{arboreal }if it is $\left(d-1\right)$-connected,
and has no simple $d$-loops. That is, there are no non-backtracking
closed chains of $d$-cells, $\sigma_{0},\sigma_{1},\ldots,\sigma_{n}=\sigma_{0}$
s.t.\ $\dim\left(\sigma_{i}\cap\sigma_{i+1}\right)=d-1$ ($\sigma_{i}$
and $\sigma_{i+1}$ are adjacent) and $\sigma_{i}\neq\sigma_{i+2}$
(the chain is non-backtracking).
\end{defn}
For $d=1$, these are simply trees. As in trees, there is a unique
$k$-regular arboreal $d$-complex for every $k\in\mathbb{N}$, and
we denote it by $T_{k}^{d}$. It can be constructed as follows: start
with a $d$-cell, and attach to each of its faces $k-1$ new $d$-cells.
Continue by induction, adding to each face of a $d$-cell in the boundary
$k-1$ new $d$-cells at every step. For example, the $2$-regular
arboreal triangle complex $T_{2}^{2}$ can be thought of as an ideal
triangulation of the hyperbolic plane, depicted in Figure \ref{fig:T_2_2}.

\begin{figure}[h]
\centering{}\includegraphics[clip,scale=0.22]{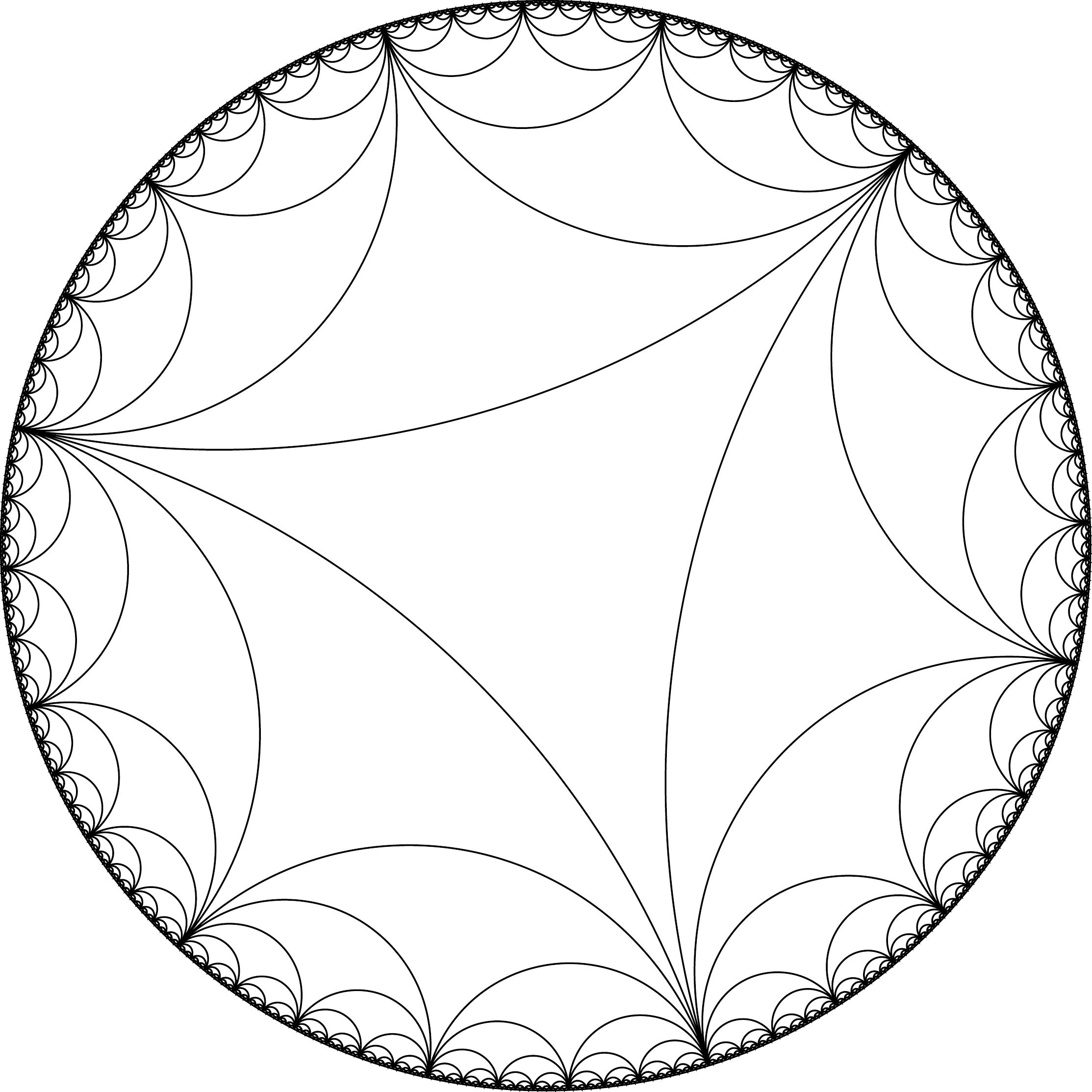}\caption{\label{fig:T_2_2}The $2$-regular arboreal triangle complex $T_{2}^{2}$.}
\end{figure}

The following theorem describes the spectrum of regular arboreal complexes:
\begin{thm}
\label{thm:The-spectrum-of-trees}The spectrum of the non-lazy transition
operator on \textup{the }$k$-regular arboreal $d$-complex\textup{
is} 
\begin{equation}
\begin{aligned}\Spec A\left(T_{k}^{d},0\right) & =\begin{cases}
\left[\frac{1-d-2\sqrt{d(k-1)}}{kd},\frac{1-d+2\sqrt{d(k-1)}}{kd}\right]\cup\left\{ \frac{1}{d}\right\} \vspace{2pt} & 2\leq k\leq d\\
\left[\frac{1-d-2\sqrt{d(k-1)}}{kd},\frac{1-d+2\sqrt{d(k-1)}}{kd}\right] & d<k.
\end{cases}\end{aligned}
\label{eq:tree-spec}
\end{equation}

\end{thm}
$ $\begin{rems*}$ $ \vspace{-6pt}$ $
\begin{enumerate}
\item For $d=1$ this gives the spectrum of the $k$-regular tree, which
is a famous result of Kesten \cite{MR0109367}: 
\[
\Spec A\left(T_{k}^{1},0\right)=\left[-\frac{2\sqrt{k-1}}{k},\frac{2\sqrt{k-1}}{k}\right].
\]

\item Since for $2\leq k\leq d$ the value $\frac{1}{d}$ is an isolated
value of the spectrum of $T_{k}^{d}$, it follows that it is in fact
an eigenvalue. This is a major difference from the case of graphs,
where the value $\frac{1}{d}=1$ cannot be an eigenvalue for infinite
graphs. This phenomena will play a crucial role in the counterexample
for the Alon-Boppana theorem in general dimension (see §\ref{sub:Alon-Boppana-type-theorems}-\ref{sub:Analysis-of-balls}).
\item Another phenomena which does not occur in the case of graphs, is that
in the region $2\leq k\leq d$ the spectrum expands as $k$ becomes
larger. The spectrum is maximal (as a set) for $k=d+1$, where $\Spec A\left(T_{d+1}^{d},0\right)=\left[-\frac{3d-1}{d\left(d+1\right)},\frac{1}{d}\right]$,
merging with the isolated eigenvalue which appear for smaller $k$.
\item The spectra of the Laplacian $\Delta^{+}=\Delta^{+}\left(T_{k}^{d}\right)$,
and of the $p$-lazy transition operator $A_{p}=A\left(T_{k}^{d},p\right)$,
are obtained from (\ref{eq:tree-spec}) using $\Delta^{+}=I-d\cdot A$
and $ $$A_{p}=p\cdot I+\left(1-p\right)\cdot A$.
\end{enumerate}
\end{rems*}

In order to prove Theorem \ref{thm:The-spectrum-of-trees} we will
need the following lemma, for the idea of which we are indebted to
Jonathan Breuer:
\begin{lem}
\label{lem:tree_lemma}Let $X$ be any set, and $L^{2}(X)$ the Hilbert
space of complex functions of finite $L^{2}$-norm on $X$ (with respect
to the counting measure). Let $A$ be a bounded self adjoint operator
on $L^{2}(X)$, and $a<b\in\mathbb{R}$, such that the following hold:
\begin{enumerate}
\item For every $x\in X$ and $a\leq\lambda\leq b$, there exists $\psi_{x}^{\lambda}\in L^{2}(X)$
such that $\left(A-\lambda I\right)\psi_{x}^{\lambda}=\mathbbm{1}_{x}$.
\item The integral $\int_{a}^{b}c\left(\lambda\right)^{2}d\lambda$ is finite,
where $c\left(\lambda\right)=\sup\limits _{x\in X}\left\Vert \psi_{x}^{\lambda}\right\Vert $
.
\end{enumerate}

Then $\left(a,b\right)\cap\Spec\left(A\right)=\varnothing$.

\end{lem}
\begin{proof}
We show that $\mathbb{P}_{\left[a,b\right]}$, the spectral projection
of $A$ on the interval $\left[a,b\right]$, is zero, and the conclusion
$\left(a,b\right)\cap\Spec\left(A\right)=\varnothing$ follows by
the spectral theorem. Stone's formula states that 
\[
\underset{\varepsilon\downarrow0}{\left(s\right)\mathrm{lim}}\frac{1}{2\pi i}\int_{a}^{b}\left[\left(A-\lambda-i\varepsilon\right)^{-1}-\left(A-\lambda+i\varepsilon\right)^{-1}\right]d\lambda=\mathbb{P}_{\left(a,b\right)}+\frac{1}{2}\mathbb{P}_{\left\{ a,b\right\} }
\]
where $\mathbb{P}_{\left(a,b\right)}$ and $\mathbb{P}_{\left\{ a,b\right\} }$
the spectral projections of $A$ on $\left(a,b\right)$ and $\left\{ a,b\right\} $
respectively, and $\left(s\right)\mathrm{lim}$ denotes a limit in
the strong sense. Denoting $\mathbb{P}=\mathbb{P}_{\left(a,b\right)}+\frac{1}{2}\mathbb{P}_{\left\{ a,b\right\} }$,
this gives for every $x\in X$ 
\[
\lim_{\varepsilon\downarrow0}\frac{1}{2\pi i}\int_{a}^{b}\left\langle \left[\left(A-\lambda-i\varepsilon\right)^{-1}-\left(A-\lambda+i\varepsilon\right)^{-1}\right]\mathbbm{1}_{x},\mathbbm{1}_{x}\right\rangle d\lambda=\left\langle \mathbb{P}\mathbbm{1}_{x},\mathbbm{1}_{x}\right\rangle 
\]
Evaluating the right hand side we get 
\begin{align*}
\left\langle \mathbb{P}\mathbbm{1}_{x},\mathbbm{1}_{x}\right\rangle  & =\lim_{\varepsilon\downarrow0}\frac{1}{2\pi i}\int_{a}^{b}\left\langle \left[\left(A-\lambda-i\varepsilon\right)^{-1}-\left(A-\lambda+i\varepsilon\right)^{-1}\right]\mathbbm{1}_{x},\mathbbm{1}_{x}\right\rangle d\lambda\\
 & =\lim_{\varepsilon\downarrow0}\frac{1}{2\pi i}\int_{a}^{b}\left\langle \left[\left(A-\lambda-i\varepsilon\right)^{-1}-\left(A-\lambda+i\varepsilon\right)^{-1}\right]\left(A-\lambda\right)\psi_{x}^{\lambda},\left(A-\lambda\right)\psi_{x}^{\lambda}\right\rangle d\lambda\\
 & =\lim_{\varepsilon\downarrow0}\frac{1}{2\pi i}\int_{a}^{b}\left\langle \left(A-\lambda+i\varepsilon\right)^{-1}\left[A-\lambda+i\varepsilon-A+\lambda+i\varepsilon\right]\left(A-\lambda-i\varepsilon\right)^{-1}\left(A-\lambda\right)^{2}\psi_{x}^{\lambda},\psi_{x}^{\lambda}\right\rangle d\lambda\\
 & =\lim_{\varepsilon\downarrow0}\frac{\varepsilon}{\pi}\int_{a}^{b}\left\langle \left((A-\lambda)^{2}+\varepsilon^{2}\right)^{-1}\left(A-\lambda\right)^{2}\psi_{x}^{\lambda},\psi_{x}^{\lambda}\right\rangle d\lambda\\
 & \leq\lim_{\varepsilon\downarrow0}\frac{\varepsilon}{\pi}\int_{a}^{b}\left\Vert \left((A-\lambda)^{2}+\varepsilon^{2}\right)^{-1}\left(A-\lambda\right)^{2}\right\Vert c\left(\lambda\right)^{2}d\lambda.
\end{align*}
Defining $f_{\varepsilon,\lambda}\left(t\right)=\frac{\left(t-\lambda\right)^{2}}{\left(t-\lambda\right)^{2}+\varepsilon^{2}}$,
we have $\left|f_{\varepsilon,\lambda}\left(t\right)\right|\leq1$
for every $t,\lambda\in\mathbb{R}$ and $\varepsilon>0$, and thus
$\left\Vert f_{\varepsilon,\lambda}\left(A\right)\right\Vert \leq1$.
Therefore, using \emph{(2)}, the last limit above is zero. Consequently,
for any $x,y\in X$ 
\[
\left|\left\langle \mathbb{P}\mathbbm{1}_{x},\mathbbm{1}_{y}\right\rangle \right|=\left|\left\langle \mathbb{P}\mathbbm{1}_{x},\mathbb{P}\mathbbm{1}_{y}\right\rangle \right|\leq\left\langle \mathbb{P}\mathbbm{1}_{x},\mathbb{P}\mathbbm{1}_{x}\right\rangle ^{\frac{1}{2}}\cdot\left\langle \mathbb{P}\mathbbm{1}_{y},\mathbb{P}\mathbbm{1}_{y}\right\rangle ^{\frac{1}{2}}=0.
\]
It follows that for general $f\in L^{2}(X)$ 
\[
\left\langle \mathbb{P}f,f\right\rangle =\left\langle \mathbb{P}\left(\sum_{x\in X}f(x)\mathbbm{1}_{x}\right),\sum_{y\in X}f(y)\mathbbm{1}_{y}\right\rangle =\sum_{x,y\in X}f(v)f(w)\left\langle \mathbb{P}\mathbbm{1}_{x},\mathbbm{1}_{y}\right\rangle =0,
\]
which implies that $\mathbb{P}=0$, hence also $\mathbb{P}_{\left(a,b\right)}$
and $\mathbb{P}_{\left\{ a,b\right\} }$, and therefore also $\mathbb{P}_{\left[a,b\right]}$.
\end{proof}

\begin{proof}[Proof of Theorem \ref{thm:The-spectrum-of-trees}]
Let $X=T_{k}^{d}$, and $\Lambda_{\pm}=\frac{1-d\pm2\sqrt{d(k-1)}}{kd}$.
The proof is separated into two parts. First we prove that every $\Lambda_{-}\leq\lambda\leq\Lambda_{+}$,
and also $\lambda=\frac{1}{d}$ when $k\leq d$, is in the spectrum,
by exhibiting an appropriate eigenform or an approximate one. In the
second part we use Lemma \ref{lem:tree_lemma} to prove that there
are no other points in the spectrum.

Define an orientation $X_{+}^{d-1}$ as follows: choose an arbitrary
$\left(d-1\right)$-cell $\sigma_{0}\in X_{\pm}^{d-1}$ and place
it in $X_{+}^{d-1}$. Then add to $X_{+}^{d-1}$ all the $k\cdot d$
neighbors of $\sigma_{0}$. Next, for every neighbor $\tau$ of the
recently added $k\cdot d$ cells, add $\tau$ to $X_{+}^{d-1}$, unless
$\tau$ or $\overline{\tau}$ is already there. Continue expanding
in this manner, adding at each stage the neighbors of the last ``layer''
which are further away from the starting cell $\sigma_{0}$. Apart
from orientation, this process gives $X_{+}^{d-1}$ a layer structure:
$\left\{ \sigma_{0}\right\} $ is the $0^{\mathrm{th}}$ layer, its
neighbors the $1^{\mathrm{st}}$ layer, and so on. We denote by $S_{n}\left(X,\sigma_{0}\right)$
the $n^{\mathrm{th}}$ layer, and also write $B_{n}\left(X,\sigma_{0}\right)=\bigcup_{k\leq n}S_{k}\left(X,\sigma_{0}\right)$
for the ``$n^{\mathrm{th}}$ ball'' around $\sigma_{0}$. Figure
\ref{fig:The-orientation-T_2_2} demonstrates this for the first four
layers of $T_{2}^{2}$.

\begin{figure}[h]
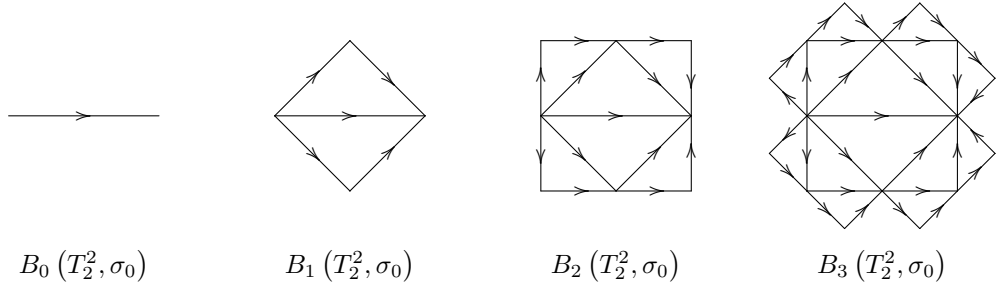

\hfill{}\xy (10,-20)*{B_{0}\left(T_2^2,\sigma_{0}\right)};
(0,0)*{}="00";
(20,0)*{}="01";
(10,10)*{}="10";
(10,-10)*{}="11";
(0,10)*{}="20";
(20,10)*{}="21";
(0,-10)*{}="22";
(20,-10)*{}="23";
(5,15)*{}="30";
(15,15)*{}="31";
(25,5)*{}="32";
(25,-5)*{}="33";
(15,-15)*{}="34";
(5,-15)*{}="35";
(-5,-5)*{}="36";
(-5,5)*{}="37";
"00",{\ar@{-}|(0.54)*=0@{>} "01"};
\endxy\hfill{}\xy (10,-20)*{B_{1}\left(T_2^2,\sigma_{0}\right)};
(0,0)*{}="00";
(20,0)*{}="01";
(10,10)*{}="10";
(10,-10)*{}="11";
(0,10)*{}="20";
(20,10)*{}="21";
(0,-10)*{}="22";
(20,-10)*{}="23";
(5,15)*{}="30";
(15,15)*{}="31";
(25,5)*{}="32";
(25,-5)*{}="33";
(15,-15)*{}="34";
(5,-15)*{}="35";
(-5,-5)*{}="36";
(-5,5)*{}="37";
"00",{\ar@{-}|(0.54)*=0@{>} "01"};
"00",{\ar@{-}|(0.59)*=0@{>} "10"};
"10",{\ar@{-}|(0.59)*=0@{>} "01"};
"01",{\ar@{-}|(0.41)*=0@{<} "11"};
"11",{\ar@{-}|(0.41)*=0@{<} "00"};
\endxy\hfill{}\xy (10,-20)*{B_{2}\left(T_2^2,\sigma_{0}\right)};
(0,0)*{}="00";
(20,0)*{}="01";
(10,10)*{}="10";
(10,-10)*{}="11";
(0,10)*{}="20";
(20,10)*{}="21";
(0,-10)*{}="22";
(20,-10)*{}="23";
(5,15)*{}="30";
(15,15)*{}="31";
(25,5)*{}="32";
(25,-5)*{}="33";
(15,-15)*{}="34";
(5,-15)*{}="35";
(-5,-5)*{}="36";
(-5,5)*{}="37";
"00",{\ar@{-}|(0.54)*=0@{>} "01"};
"00",{\ar@{-}|(0.59)*=0@{>} "10"};
"10",{\ar@{-}|(0.59)*=0@{>} "01"};
"01",{\ar@{-}|(0.41)*=0@{<} "11"};
"11",{\ar@{-}|(0.41)*=0@{<} "00"};
"00",{\ar@{-}|(0.63)*=0@{>} "20"};
"20",{\ar@{-}|(0.63)*=0@{>} "10"};
"11",{\ar@{-}|(0.37)*=0@{<} "22"};
"22",{\ar@{-}|(0.37)*=0@{<} "00"};
"10",{\ar@{-}|(0.63)*=0@{>} "21"};
"21",{\ar@{-}|(0.63)*=0@{>} "01"};
"01",{\ar@{-}|(0.37)*=0@{<} "23"};
"23",{\ar@{-}|(0.37)*=0@{<} "11"};
\endxy\hfill{}\xy (10,-20)*{B_{3}\left(T_2^2,\sigma_{0}\right)};
(0,0)*{}="00";
(20,0)*{}="01";
(10,10)*{}="10";
(10,-10)*{}="11";
(0,10)*{}="20";
(20,10)*{}="21";
(0,-10)*{}="22";
(20,-10)*{}="23";
(5,15)*{}="30";
(15,15)*{}="31";
(25,5)*{}="32";
(25,-5)*{}="33";
(15,-15)*{}="34";
(5,-15)*{}="35";
(-5,-5)*{}="36";
(-5,5)*{}="37";
"00",{\ar@{-}|(0.54)*=0@{>} "01"};
"00",{\ar@{-}|(0.59)*=0@{>} "10"};
"10",{\ar@{-}|(0.59)*=0@{>} "01"};
"01",{\ar@{-}|(0.41)*=0@{<} "11"};
"11",{\ar@{-}|(0.41)*=0@{<} "00"};
"00",{\ar@{-}|(0.63)*=0@{>} "20"};
"20",{\ar@{-}|(0.63)*=0@{>} "10"};
"11",{\ar@{-}|(0.37)*=0@{<} "22"};
"22",{\ar@{-}|(0.37)*=0@{<} "00"};
"10",{\ar@{-}|(0.63)*=0@{>} "21"};
"21",{\ar@{-}|(0.63)*=0@{>} "01"};
"01",{\ar@{-}|(0.37)*=0@{<} "23"};
"23",{\ar@{-}|(0.37)*=0@{<} "11"};
"20",{\ar@{-}|(0.63)*=0@{>} "30"};
"30",{\ar@{-}|(0.63)*=0@{>} "10"};
"10",{\ar@{-}|(0.63)*=0@{>} "31"};
"31",{\ar@{-}|(0.63)*=0@{>} "21"};
"21",{\ar@{-}|(0.63)*=0@{>} "32"};
"32",{\ar@{-}|(0.63)*=0@{>} "01"};
"01",{\ar@{-}|(0.37)*=0@{<} "33"};
"33",{\ar@{-}|(0.37)*=0@{<} "23"};
"23",{\ar@{-}|(0.37)*=0@{<} "34"};
"34",{\ar@{-}|(0.37)*=0@{<} "11"};
"11",{\ar@{-}|(0.37)*=0@{<} "35"};
"35",{\ar@{-}|(0.37)*=0@{<} "22"};
"22",{\ar@{-}|(0.37)*=0@{<} "36"};
"36",{\ar@{-}|(0.37)*=0@{<} "00"};
"00",{\ar@{-}|(0.63)*=0@{>} "37"};
"37",{\ar@{-}|(0.63)*=0@{>} "20"};
\endxy\hfill{}

\caption{\label{fig:The-orientation-T_2_2}The orientation at the zeroth, first,
second, and third layers of $X=T_{2}^{2}$.}
\end{figure}

We shall study $X_{+}^{d-1}$-spherical forms, i.e.\ forms in $\Omega^{d-1}\left(X\right)$
which are constant on each layer of $X_{+}^{d-1}$. For such a form
$f$ we will make some abuse of notation and write $f\left(n\right)$
for the value of $f$ on the cells in the $n^{\mathrm{th}}$ layer
of $X_{+}^{d-1}$. As in regular trees, if one allows forms which
are not in $L^{2}$, then for every $\lambda\in\mathbb{R}$ there
is a unique (up to a constant) $X_{+}^{d-1}$-spherical eigenform
$f$ with eigenvalue $\lambda$. This form is given explicitly by
\[
f(n)=\left(\frac{\lambda-\alpha_{-}}{\alpha_{+}-\alpha_{-}}\right)\cdot\alpha_{+}^{n}+\left(\frac{\alpha_{+}-\lambda}{\alpha_{+}-\alpha_{-}}\right)\cdot\alpha_{-}^{n},
\]
where 
\begin{equation}
\alpha_{\pm}=\frac{d-1+dk\lambda\pm\sqrt{\left(d-1+dk\lambda\right)^{2}-4d(k-1)}}{2d\left(k-1\right)},\label{eq:alpha_pm}
\end{equation}
except for the case $\alpha_{+}=\alpha_{-}$, which happens when $\lambda\in\left\{ \Lambda_{-},\Lambda_{+}\right\} $.
In this case $f$ is given by 
\[
f(n)=\left(1-n\right)\left(\frac{\left(d-1\right)+dk\lambda}{2d\left(k-1\right)}\right)^{n}+\lambda n\left(\frac{\left(d-1\right)+dk\lambda}{2d\left(k-1\right)}\right)^{n-1},
\]
but this will not concern us as the spectrum is closed, and it is
therefore enough to show that $\left(\Lambda_{-},\Lambda_{+}\right)$
is contained in it to deduce this for $\left[\Lambda_{-},\Lambda_{+}\right]$.

The term inside the root in (\ref{eq:alpha_pm}) is negative for $\Lambda_{-}<\lambda<\Lambda_{+},$
hence in this case $\left|\alpha_{+}\right|=\left|\alpha_{-}\right|=\frac{1}{\sqrt{d(k-1)}}.$
We claim the following: for any $\Lambda_{-}<\lambda<\Lambda_{+}$
there exist $0<c_{1}<c_{2}<\infty$ (which depend on $\lambda$) such
that 
\begin{enumerate}
\item For all $n\in\mathbb{N}$,
\begin{equation}
\left|f\left(n\right)\right|\leq c_{2}\left(\frac{1}{\sqrt{d(k-1)}}\right)^{n}.\label{eq:upper_bound_f_n}
\end{equation}

\item For infinitely many $n\in\mathbb{N}$,
\begin{equation}
c_{1}\left(\frac{1}{\sqrt{d(k-1)}}\right)^{n}\leq\left|f\left(n\right)\right|.\label{eq:lower_bound_f_n}
\end{equation}

\end{enumerate}
Indeed, (1) follows from $\left|f\left(n\right)\right|\leq\left[\left|\frac{\lambda-\alpha_{-}}{\alpha_{+}-\alpha_{-}}\right|+\left|\frac{\alpha_{+}-\lambda}{\alpha_{+}-\alpha_{-}}\right|\right]\left(\frac{1}{\sqrt{d(k-1)}}\right)^{n}$
(as $\alpha_{+}\neq\alpha_{-}$ for $\Lambda_{-}<\lambda<\Lambda_{+}$).
Next, denote $\gamma=\frac{\lambda-\alpha_{-}}{\alpha_{+}-\alpha_{-}}$
and observe that
\[
\left|f\left(n\right)\right|\left[d(k-1)\right]^{\frac{n}{2}}=\left|\gamma\alpha_{+}^{n}+\overline{\gamma}\alpha_{-}^{n}\right|\left[d(k-1)\right]^{\frac{n}{2}}=2\Re\left(\gamma\left(\alpha_{+}\sqrt{d\left(k-1\right)}\right)^{n}\right).
\]
If (2) fails, then $\left|f\left(n\right)\right|\left[d(k-1)\right]^{\frac{n}{2}}\overset{{\scriptscriptstyle n\rightarrow\infty}}{\longrightarrow}0$.
Since $\left|\alpha_{+}\sqrt{d\left(k-1\right)}\right|=1$, this means
that $n\arg\alpha_{+}\overset{{\scriptscriptstyle n\rightarrow\infty}}{\longrightarrow}\frac{\pi}{2}-\arg\gamma\:\left(\mathrm{mod}\,\pi\right)$,
hence $\alpha_{+}\in\mathbb{R}$, which is false. 

Even though $f$ is not in $\Omega_{L^{2}}^{d-1}\left(X\right)$
it induces a natural sequence of approximate eigenforms: 
\[
f_{n}\left(\sigma\right)=\begin{cases}
f\left(k\right) & \sigma\in S_{k}\left(X,\sigma_{0}\right)\:\mathrm{and}\: k\leq n\\
-f\left(k\right) & \overline{\sigma}\in S_{k}\left(X,\sigma_{0}\right)\:\mathrm{and}\: k\leq n\\
0 & \mbox{otherwise}.
\end{cases}
\]
To see this, observe that $\left(A_{0}-\lambda\right)f=0$, and that
$f_{n}$ coincides with $f$ on $B_{n}\left(X,\sigma_{0}\right)$
for $k\leq n$ and vanishes on $\left(T_{d}^{k}\right)^{d-1}\backslash B_{n}\left(X,\sigma_{0}\right)$.
It follows that $\left(A_{0}-\lambda\right)f_{n}$ is supported on
$S_{n}\left(X,\sigma_{0}\right)\cup S_{n+1}\left(X,\sigma_{0}\right)$,
and by $\left|S_{n}\left(X,\sigma_{0}\right)\right|=d^{n}k\left(k-1\right)^{n-1}$,
the definition of $A_{0}$, and (\ref{eq:upper_bound_f_n}) 
\begin{align*}
\frac{\left\Vert \left(A_{0}-\lambda\right)f_{n}\right\Vert ^{2}}{\left\Vert f_{n}\right\Vert ^{2}} & =\frac{\left|S_{n}\left(X,\sigma_{0}\right)\right|\left(\frac{1}{dk}\left[f\left(n-1\right)-\left(d-1\right)f\left(n\right)\right]-\lambda f(n)\right)^{2}+\left|S_{n+1}\left(X,\sigma_{0}\right)\right|\left(\frac{1}{dk}f\left(n\right)\right)^{2}}{\sum_{j=0}^{n}\left|S_{j}\left(X,\sigma_{0}\right)\right|f^{2}(j)}\\
 & =\frac{d^{n}k\left(k-1\right){}^{n-1}\cdot\left(-\frac{k-1}{k}f\left(n+1\right)\right)^{2}+d^{n+1}k\left(k-1\right){}^{n}\left(\frac{1}{dk}f\left(n\right)\right)^{2}}{f^{2}\left(0\right)+\sum_{j=1}^{n}d^{j}k\left(k-1\right)^{j-1}f^{2}\left(j\right)}\\
 & =\frac{d^{n}k^{-1}\left(k-1\right){}^{n+1}f\left(n+1\right)^{2}+d^{n-1}k^{-1}\left(k-1\right){}^{n}f\left(n\right)^{2}}{f^{2}\left(0\right)+\sum_{j=1}^{n}d^{j}k\left(k-1\right)^{j-1}f^{2}\left(j\right)}\\
 & \leq\frac{\frac{2c_{2}^{2}}{dk}}{f^{2}\left(0\right)+\frac{k}{k-1}\sum_{j=1}^{n}\left[d\left(k-1\right)\right]^{j}f\left(j\right)^{2}}.
\end{align*}
By (\ref{eq:lower_bound_f_n}), the denominator becomes arbitrarily
large as $n$ grows, and therefore $\frac{\left\Vert \left(A_{0}-\lambda\right)f_{n}\right\Vert ^{2}}{\left\Vert f_{n}\right\Vert ^{2}}\rightarrow0$
and $\lambda\in\Spec A_{0}$.

Turning to the isolated eigenvalues in (\ref{eq:tree-spec}), one
can easily check that $f\left(n\right)=\frac{1}{d^{n}}$ is an eigenform
with eigenvalue $\frac{1}{d}$, and for $2\leq k\leq d$ it is in
$L^{2}$. This concludes the first part of the proof.

\medskip{}

Next assume that $\lambda\in\left(-1,\frac{1}{d}\right)\backslash\left[\Lambda_{-},\Lambda_{+}\right]$.
We show that in this case Lemma \ref{lem:tree_lemma} can be applied.
Let $\sigma_{0}$ and $X_{+}^{d-1}$ be as before, including the layer
structure. Define the following $X_{+}^{d-1}$-spherical forms: 
\begin{equation}
\psi_{\sigma_{0}}^{\lambda}\left(n\right)=\frac{\alpha_{+}^{n}}{\alpha_{+}-\lambda},\qquad\varphi_{\sigma_{0}}^{\lambda}\left(n\right)=\frac{\alpha_{-}^{n}}{\alpha_{-}-\lambda}.\label{eq:psi_phi_lambda}
\end{equation}
The functions $\psi_{\sigma_{0}}^{\lambda}$ is defined whenever
$\lambda\neq\alpha_{+}$, which holds unless $\lambda=\frac{1}{d}$
and $k\leq d+1$ (see (\ref{eq:alpha_pm})). Similarly, $\varphi_{\sigma_{0}}^{\lambda}$
is defined unless $\lambda=-1$, or $\lambda=\frac{1}{d}$ and $k\leq d+1$.
It is straightforward to verify that 
\[
\left(A_{0}-\lambda I\right)\psi_{\sigma_{0}}^{\lambda}=\left(A_{0}-\lambda I\right)\varphi_{\sigma_{0}}^{\lambda}=\mathbbm{1}_{\sigma_{0}}
\]
whenever the functions are defined. For every $X_{+}^{d-1}$-spherical
form $f$ one has 
\begin{equation}
\left\Vert f\right\Vert ^{2}=\sum_{n=0}^{\infty}\left|S_{n}\left(X,\sigma_{0}\right)\right|f^{2}\left(n\right)=f^{2}\left(0\right)+\frac{k}{k-1}\sum_{n=1}^{\infty}\left[\left(k-1\right)d\right]^{n}f^{2}\left(n\right).\label{eq:spherical_norm}
\end{equation}
One can verify that $0<d\left(k-1\right)\alpha_{+}^{2}<1$ holds for
all $\lambda<\Lambda_{-}$, and thus by (\ref{eq:psi_phi_lambda})
and (\ref{eq:spherical_norm}) $\left\Vert \psi_{\sigma_{0}}^{\lambda}\right\Vert $
is finite. In fact, $\left\Vert \psi_{\sigma_{0}}^{\lambda}\right\Vert $
is continuous w.r.t.\ $\lambda$ in this region, so that it is bounded
on every interval $\left[a,b\right]\subseteq\left(-\infty,\Lambda_{-}\right)$.
Furthermore, for any $\sigma\in X^{d-1}$ there is an isometry of
$T_{k}^{d}$ which takes $\sigma_{0}$ to $\sigma$, and thus $\psi_{\sigma_{0}}^{\lambda}$
to a form $\psi_{\sigma}^{\lambda}$ with the same $L^{2}$-norm as
$\psi_{\sigma_{0}}^{\lambda}$, and which satisfies $\left(A_{0}-\lambda I\right)\psi_{\sigma}^{\lambda}=\mathbbm{1}_{\sigma}$.
We can now invoke Lemma \ref{lem:tree_lemma} for $\left[a,b\right]\subseteq\left(-\infty,\Lambda_{-}\right)$,
using $\psi_{\sigma_{0}}^{\lambda}$ and its translations by isometries,
and obtain that $\left(a,b\right)\cap\Spec A_{0}=\varnothing$. Thus,
$\Spec A_{0}$ does not intersect $\left(-\infty,\Lambda_{-}\right)$.

Similarly, $0<d\left(k-1\right)\alpha_{-}^{2}<1$ holds for all $\lambda>\Lambda_{+}$,
so that the same argumentation for $\varphi_{\sigma_{0}}^{\lambda}$
shows that $\Spec A_{0}$ does not intersect $\left(\Lambda_{+},\infty\right)$,
provided that $d+1<k$. When $k\leq d+1$ we know that $\frac{1}{d}\in\Spec A_{0}$,
and we need to show that $\Spec A_{0}$ does not intersect $\left(\Lambda_{+},\infty\right)\backslash\left\{ \frac{1}{d}\right\} $.
This is done in the same manner, observing intervals $\left[a,b\right]\subseteq\left(\Lambda_{+},\frac{1}{d}\right)$
and $\left[a,b\right]\subseteq\left(\frac{1}{d},\infty\right)$ separately.
\end{proof}

\subsection{\label{sub:Continuity-of-the}Continuity of the spectral measure}

In this section we generalize parts of Grigorchuk and \.{Z}uk's work
on graphs \cite{grigorchuk1999asymptotic} to general simplicial complexes.
We assume throughout the section that all $d$-complexes referred
to are $\left(d-1\right)$-connected, and that families and sequences
of $d$-complexes we encounter have globally bounded $\left(d-1\right)$-degrees.

For a uniform $d$-complex $X$ we define the distance between two
$\left(d-1\right)$-cells to be the minimal length of a $\left(d-1\right)$-chain
connecting them:
\[
\dist\left(\sigma,\sigma'\right)=\min\left\{ n\,\middle|\,{\exists\sigma_{0},\sigma_{1},\ldots,\sigma_{n}=\sigma_{0}\in X^{d-1}\;\mathrm{s.t.}\atop \sigma_{i}\cup\sigma_{i+1}\in X^{d}\quad\forall i}\right\} .
\]
We denote by $B_{n}\left(X,\sigma\right)$ the ball of radius $n$
around $ $$\sigma$ in $X$, which is the maximal subcomplex of $X$
all of whose $\left(d-1\right)$-cells are of distance at most $n$
from $\sigma$%
\footnote{this is similar to $B_{n}\left(X,\sigma\right)$ defined in the proof
of theorem \ref{thm:The-spectrum-of-trees}, but there $B_{n}\left(X,\sigma\right)$
referred only to the $\left(d-1\right)$-cells, and here to the entire
subcomplex%
}. A \emph{marked $d$-complex} $\left(X,\sigma\right)$ is a $d$-complex
with a choice of a $\left(d-1\right)$-cell $\sigma$. On the space
of marked $d$-complexes with finite $\left(d-1\right)$-degrees one
can define a metric by 
\[
\dist\left(\left(X_{1},\sigma_{1}\right),\left(X_{2},\sigma_{2}\right)\right)=\inf\left\{ \frac{1}{n+1}\,:\, B_{n}\left(X_{1},\sigma_{1}\right)\,\mbox{is isometric to }B_{n}\left(X_{2},\sigma_{2}\right)\right\} 
\]
\begin{rems*}$ $ \vspace{-6pt}$ $
\begin{enumerate}
\item A limit $\left(X,\sigma\right)$ of a sequence $\left(X_{n},\sigma_{n}\right)$
in this space is unique up to isometry. 
\item For every $K\in\mathbb{N}$, the subspace of $d$-complexes with $\left(d-1\right)$-degrees
bounded by $K$ is compact. This is due to the fact that there is
only a finite number of possibilities for a ball of radius $n$, so
that every sequence has a converging subsequence by a diagonal argument
(see \cite{grigorchuk1999asymptotic} for details).
\end{enumerate}
\end{rems*}

\textbf{}Our next goal is to study the relation of this metric to
the spectra of complexes. We use some standard spectral theoretical
results which we summarize as follows: Let $X$ be a countable set
with a weighted counting measure $w$, i.e., $\int_{X}f=\sum_{x\in X}w\left(x\right)f\left(x\right)$,
and $A$ a self-adjoint operator on $L^{2}\left(X,w\right)$. For
every $x\in X$, the spectral measure $\mu_{x}$ is the unique regular
Borel measure on $\mathbb{C}$ such that for every polynomial $P\left(t\right)\in\mathbb{C}\left[t\right]$
\[
\left\langle P(A)\mathbbm{1}_{x},\mathbbm{1}_{x}\right\rangle =\int_{\mathbb{C}}P(z)d\mu_{x}(z),
\]
where $\mathbbm{1}_{x}$ is the Dirac function of the point $x$.
For $x,y\in X$ the spectral measure $\mu_{x,y}$ is the unique regular
Borel measure on $\mathbb{C}$ such that for every polynomial $P$
\[
\left\langle P(A)\mathbbm{1}_{x},\mathbbm{1}_{y}\right\rangle =\int_{\mathbb{C}}P(z)d\mu_{x,y}(z).
\]
The spectrum of $A$ can be inferred from the spectral measures by
\begin{equation}
\Spec A=\bigcup_{x,y\in X}\supp\mu_{x,y}=\bigcup_{x\in X}\supp\mu_{x}.\label{eq:spectral_stuff}
\end{equation}
We wish to apply this mechanism to the analysis of the action of $A=A\left(X,0\right)=\frac{I-\Delta^{+}}{d}$
on $\Omega_{L^{2}}^{d-1}$ (with the inner product as in (\ref{eq:our-inner-product})),
and this is justified by observing that for any choice of orientation
$X_{+}^{d-1}$ of $X^{d-1}$, we have an isometry $\Omega_{L^{2}}^{d-1}\cong L^{2}\left(X_{+}^{d-1},w\right)$,
where $w\left(\sigma\right)=\frac{1}{\deg\sigma}$. For any $\sigma\in X^{d-1}$
we denote by $\mu_{\sigma}^{X}$ the spectral measure of $A$ w.r.t.\ $\mathbbm{1}_{\sigma}$.
Similarly, $\mu_{\sigma,\sigma'}^{X}$ denotes the spectral measure
of $A$ w.r.t.\ $\mathbbm{1}_{\sigma}$ and $\mathbbm{1}_{\sigma'}$.
\begin{lem}
\label{Lemma-convergence-of-spectral-measures}If $\lim\limits _{n\rightarrow\infty}\left(X_{n},\sigma_{n}\right)=\left(X,\sigma\right)$
then $\mu_{\sigma_{n}}^{X_{n}}$ converges weakly to $\mu_{\sigma}^{X}$.\end{lem}
\begin{proof}
For regular finite Borel measures on $\mathbb{R}$ with compact support,
weak convergence follows from convergence of the moments of the measures
(see e.g.\ \cite[§VIII.1]{feller1966introduction}). For $m\geq0$
the $m^{\mathrm{th}}$ moment of $\mu_{\sigma}^{X}$, denoted $\left(\mu_{\sigma}^{X}\right)^{\left(m\right)}$,
is given by 
\[
\left(\mu_{\sigma}^{X}\right)^{\left(m\right)}=\int_{\mathbb{C}}z^{m}d\mu_{\sigma}^{X}(z)=\left\langle A^{m}\mathbbm{1}_{\sigma},\mathbbm{1}_{\sigma}\right\rangle =\left\langle A^{m}\mathcal{E}_{0}^{\sigma},\mathbbm{1}_{\sigma}\right\rangle =\left\langle \mathcal{E}_{m}^{\sigma},\mathbbm{1}_{\sigma}\right\rangle =\frac{\mathcal{E}_{m}^{\sigma}\left(\sigma\right)}{\deg\sigma},
\]
where $\mathcal{E}_{m}^{\sigma}$ is the $0$-lazy expectation process
starting at $\sigma$, at time $m$. However, 
\[
\mathcal{E}_{m}^{\sigma}\left(\sigma\right)=\mathbf{p}_{m}^{\sigma}\left(\sigma\right)-\mathbf{p}_{m}^{\sigma}\left(\overline{\sigma}\right)
\]
is determined by the structure of the complex in the ball $B_{m}\left(X,\sigma\right)$.
For large enough $n$, $B_{m}\left(X,\sigma\right)$ is isometric
to $B_{m}\left(X_{n},\sigma_{n}\right)$, which implies that $\left(\mu_{\sigma_{n}}^{X_{n}}\right)^{(m)}=\left(\mu_{\sigma}^{X}\right)^{(m)}$. 
\end{proof}

\subsection{\label{sub:Alon-Boppana-type-theorems}Alon-Boppana type theorems}
\begin{defn}
\label{def:convergence}A sequence of $d$-complexes $X_{n}$, whose
$\left(d-1\right)$-degrees are bounded globally, is said to converge
to the complex $X$ (written $X_{n}\overset{{\scriptscriptstyle n\rightarrow\infty}}{\longrightarrow}X$)
if $\left(X_{n},\sigma_{n}\right)$ converges to $\left(X,\sigma\right)$
for some choice of $\sigma_{n}\in X_{n}^{d-1}$ and $\sigma\in X^{d-1}$.
\end{defn}
In particular, if $X$ is an infinite $d$-complex with bounded $\left(d-1\right)$-degrees,
and $\left\{ X_{n}\right\} $ is a sequence of quotients of $X$ whose
injectivity radii approach infinity, then $X_{n}\overset{{\scriptscriptstyle n\rightarrow\infty}}{\longrightarrow}X$. 

\medskip{}
The following is (one form of) the classic Alon-Boppana theorem:
\begin{thm}[Alon-Boppana]
\label{Thm:alon-boppana}Let $G_{n}$ be a sequence of graphs whose
degrees are globally bounded, and $G$ a graph s.t. $G_{n}\overset{{\scriptscriptstyle n\rightarrow\infty}}{\longrightarrow}G$.
Then
\[
\liminf_{n\rightarrow\infty}\lambda\left(G_{n}\right)\leq\lambda\left(G\right).
\]

\end{thm}
In the literature one encounters many variations on this formulation:
some refer only to quotients of $G$, some only to regular graphs,
and some are quantitative (e.g.\ \cite{nilli1991second}). 

In this section we study the analogue question for complexes of general
dimension. We start with the following:
\begin{thm}
\label{thm:cont-spec}If \textup{$X_{n}\overset{{\scriptscriptstyle n\rightarrow\infty}}{\longrightarrow}X$}
and $\lambda\in\Spec A\left(X,0\right)$, there exist $\lambda_{n}\in\Spec A\left(X_{n},0\right)$
with $\lim\limits _{n\rightarrow\infty}\lambda_{n}=\lambda$. The
same holds for the corresponding Laplacians $\Delta_{X}^{+}$ and
$\Delta_{X_{n}}^{+}$.\end{thm}
\begin{proof}
Let $\sigma_{n},\sigma$ be as in Definition \ref{def:convergence}.
Since $\lambda\in\Spec A\left(X,0\right)$, for every $\varepsilon>0$
there exists $\sigma'\in X^{d-1}$ such that $\mu_{\sigma'}^{X}\left(\left(\lambda-\varepsilon,\lambda+\varepsilon\right)\right)>0$.
We denote $r=\dist\left(\sigma,\sigma'\right)$, and restrict our
attention to the tail of $\left\{ \left(X_{n},\sigma_{n}\right)\right\} $
in which $B_{r}\left(X_{n},\sigma_{n}\right)$ is isometric to $B_{r}\left(X,\sigma\right)$.
If $\sigma_{n}'$ is the image of $\sigma'$ under such an isometry,
and $d_{n}=\max\left\{ k\,\middle|\, B_{k}\left(X_{n},\sigma_{n}\right)\cong B_{k}\left(X,\sigma\right)\right\} $,
then $B_{d_{n}-r}\left(X_{n},\sigma_{n}'\right)\cong B_{d_{n}-r}\left(X,\sigma'\right)$,
and since $d_{n}-r\rightarrow\infty$ we have $\left(X_{n},\sigma_{n}'\right)\rightarrow\left(X,\sigma'\right)$.
By Lemma \ref{Lemma-convergence-of-spectral-measures}, $\mu_{\sigma'_{n}}^{X_{n}}\left(\left(\lambda-\varepsilon,\lambda+\varepsilon\right)\right)>0$
for large enough $n$ and therefore $\Spec A\left(X_{n},0\right)$
intersects $\left(\lambda-\varepsilon,\lambda+\varepsilon\right)$.
The result for the Laplacians follows from the fact that $\Delta^{+}=I-d\cdot A$.
\end{proof}
In particular this gives:
\begin{cor}
If \textup{$X_{n}\overset{{\scriptscriptstyle n\rightarrow\infty}}{\longrightarrow}X$}
then $\Spec A_{X}\subseteq\overline{\bigcup_{n}\Spec A_{X_{n}}}$.
\end{cor}
This is an analogue of \cite[Thm.\ 4.3]{li2004ramanujan}, which is
also regarded sometimes as an Alon-Boppana theorem. In \cite{li2004ramanujan}
the same statement is proved for the Hecke operators acting on $X=\mathcal{B}_{n,F}$,
the Bruhat-Tits building of type $\widetilde{A}_{n}$, and on a sequence
of quotients of $X$ whose injectivity radii approach infinity.

Returning to the formulation of Alon-Boppana with spectral gaps, Theorem
\ref{thm:cont-spec} yields as an immediate result that if $X_{n}\overset{{\scriptscriptstyle n\rightarrow\infty}}{\longrightarrow}X$
then 
\begin{equation}
\liminf_{n\rightarrow\infty}\min\Spec\Delta_{X_{n}}^{+}\leq\min\Spec\Delta_{X}^{+}\leq\lambda\left(X\right).\label{eq:simple-alon-boppana}
\end{equation}
In order to obtain the higher dimensional analogue of the Alon-Boppana
theorem one would like to verify that this holds also when the spectrum
of $\Delta_{X_{n}}^{+}$ is restricted to $Z_{d-1}=\left(B^{d-1}\right)^{\perp}$.
But while this holds for graphs, the situation is more involved in
general dimension. First of all, it does not hold in general:
\begin{thm}
\label{thm:T_2_2-counterexample}Let $T_{2}^{2}$ be the arboreal
2-regular triangle complex (Figure \ref{fig:T_2_2}), and $X_{r}=B_{r}\left(T_{2}^{2},e_{0}\right)$
be the ball of radius $r$ around an edge in it (as in Figure \ref{fig:The-orientation-T_2_2}).
Then $\lim\limits _{r\rightarrow\infty}\lambda\left(X_{r}\right)=\frac{3}{2}-\sqrt{2}$,
while $\lambda\left(T_{2}^{2}\right)=0$.
\end{thm}
The proof follows in the next section. Before we delve into this counterexample,
let us exhibit first several cases in which the Alon-Boppana analogue
does hold:
\begin{thm}
\label{thm:high-alon-boppana}If \textup{$X_{n}\overset{{\scriptscriptstyle n\rightarrow\infty}}{\longrightarrow}X$},
and one of the following holds:
\begin{enumerate}
\item Zero is not in $\mathrm{\Spec}\,\Delta_{X}^{+}\big|_{Z_{d-1}}$ (i.e.\ $\lambda\left(X\right)\neq0$),
\item zero is a non-isolated point in $\mathrm{\Spec}\,\Delta_{X}^{+}\big|_{Z_{d-1}}$,
or
\item the $\left(d-1\right)$-skeletons of the complexes $X_{n}$ form a
family of $\left(d-1\right)$-expanders,
\end{enumerate}

then 
\[
\liminf_{n\rightarrow\infty}\lambda\left(X_{n}\right)\leq\lambda\left(X\right).
\]

\end{thm}
\begin{proof}
By Theorem \ref{thm:cont-spec} there exist $\lambda_{n}\in\Spec\Delta_{X_{n}}^{+}$
with $\lambda_{n}\rightarrow\lambda\left(X\right)$. If \emph{(1)}
holds, then $\lambda_{n}>0$ for large enough $n$, which implies
that $\lambda_{n}\in\Spec\Delta_{X_{n}}^{+}\big|_{Z_{d-1}}$, hence
$\lambda\left(X_{n}\right)=\min\Spec\Delta_{X_{n}}^{+}\big|_{Z_{d-1}}\leq\lambda_{n}$.
Thus, $\liminf_{n\rightarrow\infty}\lambda\left(X_{n}\right)\leq\liminf_{n\rightarrow\infty}\lambda_{n}=\lambda\left(X\right)$.
If \emph{(2)} holds then there are $\mu_{n}\in\Spec\Delta_{X}^{+}\backslash\left\{ 0\right\} $
with $\mu_{n}\rightarrow\lambda\left(X\right)$. For every $\mu_{n}$
there is a sequence $\lambda_{n,m}\in\Spec\Delta_{X_{m}}^{+}\big|_{Z_{d-1}}$
with $\lambda_{n,m}\overset{{\scriptscriptstyle m\rightarrow\infty}}{\longrightarrow}\mu_{n}$,
and $\lambda_{n,n}\rightarrow\lambda\left(X\right)$. 

In \emph{(3)} we mean that the $\left(d-2\right)$-cells in $X_{n}$
have globally bounded degrees, and the $\left(d-2\right)$-dimensional
spectral gaps 
\[
\lambda_{d-2}\left(X_{n}\right)=\min\Spec\Delta_{d-2}^{+}\big|_{Z_{d-2}\left(X_{n}\right)}
\]
are bounded away from zero (see Remark (1) after the proof). For example,
if $X_{n}$ are triangle complexes, this means that their underlying
graphs form a family of expander graphs in the classical sense. By
the previous cases, we can assume that $\lambda\left(X\right)=0$,
and furthermore that zero is an isolated point in $\mathrm{\Spec}\,\Delta_{X}^{+}\big|_{Z_{d-1}}$.
This implies that it is an eigenvalue, so that there exists $0\neq f\in Z_{d-1}\left(X\right)=B^{d-1}\left(X\right)^{\bot}$
with $\Delta_{X}^{+}f=0$.

Since $X_{n}\overset{{\scriptscriptstyle n\rightarrow\infty}}{\longrightarrow}X$
there exist $\sigma_{n}\in X_{n}$, $ $$\sigma_{\infty}\in X$, a
sequence $r\left(n\right)\rightarrow\infty$, and isometries $\psi_{n}:B_{r\left(n\right)}\left(X_{n},\sigma_{n}\right)\overset{\cong}{\longrightarrow}B_{r\left(n\right)}\left(X,\sigma_{\infty}\right)$.
Define $f_{n}\in\Omega_{L^{2}}^{d-1}\left(X_{n}\right)$ by 
\[
f_{n}\left(\tau\right)=\begin{cases}
f\left(\psi_{n}\left(\tau\right)\right) & \dist\left(\tau,\sigma_{n}\right)\leq r\left(n\right)\\
0 & r\left(n\right)<\dist\left(\tau,\sigma_{n}\right).
\end{cases}
\]
We first claim that $\left\Vert \Delta^{+}f_{n}\right\Vert $ and
$\left\Vert \Delta^{-}f_{n}\right\Vert $ converge to zero ($\Delta^{-}=\Delta^{-}\left(X_{n}\right)$
are defined since the $\left(d-2\right)$-degrees are bounded). Since
$f_{n}$ is zero outside $B_{r\left(n\right)}\left(X_{n},\sigma_{n}\right)$
and coincide with $f$ on it, by $\Delta^{+}f=0$ we have 
\begin{align*}
\left\Vert \Delta^{+}f_{n}\right\Vert ^{2} & =\sum_{\sigma\in X_{n}^{d-1}}\left|\Delta^{+}f_{n}\left(\sigma\right)\right|^{2}=\sum_{\sigma:r\left(n\right)\leq\dist\left(\sigma,\sigma_{n}\right)\leq r\left(n\right)+1}\left|\Delta^{+}f_{n}\left(\sigma\right)\right|^{2}\\
 & =\sum_{\sigma:r\left(n\right)\leq\dist\left(\sigma,\sigma_{n}\right)\leq r\left(n\right)+1}\left|f_{n}\left(\sigma\right)-\sum_{\sigma'\sim\sigma}\frac{f_{n}\left(\sigma'\right)}{\deg\sigma'}\right|^{2}.
\end{align*}
Using $\left(\sum_{i=1}^{k}a_{i}\right)^{2}\leq k\sum_{i=1}^{k}a_{i}^{2}$
this gives
\[
\left\Vert \Delta^{+}f_{n}\right\Vert ^{2}\leq\left(dK+1\right)\sum_{\sigma:r\left(n\right)\leq\dist\left(\sigma,\sigma_{n}\right)\leq r\left(n\right)+1}\left[\left|f_{n}\left(\sigma\right)\right|^{2}+\sum_{\sigma'\sim\sigma}\left|f_{n}\left(\sigma'\right)\right|^{2}\right],
\]
where $K$ is a bound on the degree of $\left(d-1\right)$-cells in
$X$ and $X_{n}$. Since every $\left(d-1\right)$-cell has at most
$dK$ neighbors, we have
\begin{align*}
\left\Vert \Delta^{+}f_{n}\right\Vert ^{2} & \leq dK\left(dK+1\right)\sum_{\sigma:r\left(n\right)-1\leq\dist\left(\sigma,\sigma_{n}\right)\leq r\left(n\right)+2}\left|f_{n}\left(\sigma\right)\right|^{2}\\
 & \leq dK\left(dK+1\right)\left\Vert f\big|_{X\backslash B_{X}\left(\sigma,r\left(n\right)-2\right)}\right\Vert ^{2}\overset{{\scriptscriptstyle n\rightarrow\infty}}{\longrightarrow}0.
\end{align*}

The reasoning for $\left\Vert \Delta^{-}f_{n}\right\Vert \rightarrow0$
(see (\ref{eq:LaplacianDown})) is analogous: (\ref{eq:d-2-finite-cycles})
gives $\Delta^{-}f=0$, and the assumptions that $\left(d-2\right)$-degrees
are globally bounded yields similar bounds as done for $\Delta^{+}$.

For every $n$ write $f_{n}=z_{n}+b_{n}$, with $z_{n}\in Z_{d-1}\left(X_{n}\right)$
and $b_{n}\in B^{d-1}\left(X_{n}\right)$. It is enough to show that
$\left\Vert z_{n}\right\Vert $ are bounded away from zero, since
then $\frac{\left\Vert \Delta^{+}z_{n}\right\Vert }{\left\Vert z_{n}\right\Vert }=\frac{\left\Vert \Delta^{+}f_{n}\right\Vert }{\left\Vert z_{n}\right\Vert }\rightarrow0$,
showing that $\lambda\left(X_{n}\right)=\min\Spec\left(\Delta^{+}\big|_{Z_{d-1}\left(X_{n}\right)}\right)$
converge to zero.

Assume therefore that there are arbitrarily small $\left\Vert z_{n}\right\Vert $,
and by passing to a subsequence that $\left\Vert z_{n}\right\Vert \rightarrow0$.
Then $\left\Vert b_{n}\right\Vert \rightarrow\left\Vert f\right\Vert >0$,
giving $\frac{\left\Vert \Delta^{-}b_{n}\right\Vert }{\left\Vert b_{n}\right\Vert }=\frac{\left\Vert \Delta^{-}f_{n}\right\Vert }{\left\Vert b_{n}\right\Vert }\rightarrow0$.
This implies that $\lambda'_{n}=\min\Spec\left(\Delta^{-}\big|_{B^{d-1}\left(X_{n}\right)}\right)$
converge to zero. However, 
\begin{align*}
\lambda'_{n} & =\min\Spec\left(\Delta^{-}\big|_{B^{d-1}\left(X_{n}\right)}\right)=\min\Spec\left(\partial_{d-1}^{*}\partial_{d-1}\big|_{B^{d-1}\left(X_{n}\right)}\right)\\
 & \overset{\star}{=}\min\Spec\left(\partial_{d-1}\partial_{d-1}^{*}\big|_{B_{d-2}\left(X_{n}\right)}\right)=\min\Spec\left(\Delta_{d-2}^{+}\big|_{B_{d-2}\left(X_{n}\right)}\right)\\
 & \geq\min\Spec\left(\Delta_{d-2}^{+}\big|_{Z_{d-2}\left(X_{n}\right)}\right)=\lambda_{d-2}\left(X_{n}\right)
\end{align*}
where $\star$ is due to the fact that $B^{d-1}$ and $B_{d-1}$ are
the orthogonal complements of $\ker\partial_{d-1}$ and $\ker\partial_{d-1}^{*}$
respectively. This is a contradiction, since $\lambda_{d-2}\left(X_{n}\right)$
are bounded away from zero.
\end{proof}
\begin{rems*}$ $ \vspace{-6pt}
\begin{enumerate}
\item If $X^{\left(j\right)}$ denote the $j$-skeleton of a complex $X$,
i.e.\ the subcomplex consisting of all cells of dimension $\leq j$,
then one can look at $\lambda\left(X_{n}^{\left(d-1\right)}\right)$
instead of at $\lambda_{d-2}\left(X_{n}\right)$. Since we have different
weight functions in codimension one, these are not equal. However,
since we assumed that all $\left(d-1\right)$ and $\left(d-2\right)$
degrees are globally bounded (and nonzero), the norms induced by these
choices of weight functions are equivalent, and thus $\lambda\left(X_{n}^{\left(d-1\right)}\right)$
are bounded away from zero iff $\lambda_{d-2}\left(X_{n}\right)$
are.
\item The Alon-Boppana theorem for graphs follows from condition \emph{(2)}
in this Proposition (as done in \cite{grigorchuk1999asymptotic}),
since zero is never an isolated point in the spectrum of the Laplacian
of an infinite connected graph. Otherwise, it would correspond to
an eigenfunction, which is some multiple of the degree function, hence
not in $L^{2}$.
\end{enumerate}
\end{rems*}

\subsection{\label{sub:Analysis-of-balls}Analysis of balls in $T_{2}^{2}$}

In this section we analyze the spectrum of balls in the 2-regular
triangle complex $T_{2}^{2}$, proving in particular that they constitute
a counterexample for the higher-dimensional analogue of Alon-Boppana
(Theorem \ref{thm:T_2_2-counterexample}). We denote here $X_{r}=B_{r}\left(T_{2}^{2},e_{0}\right)$
the ball of radius $r$ around an edge $e_{0}$ in $T_{2}^{2}$: $X_{0}$
is a single edge, $X_1=\raisebox{1.5pt}{\scalebox{0.4}{$\xy (0,0)*{}="00";
(10,0)*{}="01";
(5,5)*{}="10";
(5,-5)*{}="11";
(0,5)*{}="20";
(10,5)*{}="21";
(0,-5)*{}="22";
(10,-5)*{}="23";
"00";"01" **\dir{-};
"00";"10" **\dir{-};
"00";"11" **\dir{-};
"01";"10" **\dir{-};
"01";"11" **\dir{-};
\endxy$}}$, $X_2=\raisebox{1.5pt}{\scalebox{0.4}{$\xy (0,0)*{}="00";
(10,0)*{}="01";
(5,5)*{}="10";
(5,-5)*{}="11";
(0,5)*{}="20";
(10,5)*{}="21";
(0,-5)*{}="22";
(10,-5)*{}="23";
"00";"01" **\dir{-};
"00";"10" **\dir{-};
"00";"11" **\dir{-};
"01";"10" **\dir{-};
"01";"11" **\dir{-};
"00";"20" **\dir{-};
"20";"10" **\dir{-};
"00";"22" **\dir{-};
"22";"11" **\dir{-};
"01";"21" **\dir{-};
"21";"10" **\dir{-};
"01";"23" **\dir{-};
"23";"11" **\dir{-};
\endxy$}}$ , $X_3=\raisebox{1.5pt}{\scalebox{0.4}{$\xy (0,0)*{}="00";
(10,0)*{}="01";
(5,5)*{}="10";
(5,-5)*{}="11";
(0,5)*{}="20";
(10,5)*{}="21";
(0,-5)*{}="22";
(10,-5)*{}="23";
(2.5,7.5)*{}="30";
(7.5,7.5)*{}="31";
(12.5,2.5)*{}="32";
(12.5,-2.5)*{}="33";
(7.5,-7.5)*{}="34";
(2.5,-7.5)*{}="35";
(-2.5,-2.5)*{}="36";
(-2.5,2.5)*{}="37";
"00";"01" **\dir{-};
"00";"10" **\dir{-};
"00";"11" **\dir{-};
"01";"10" **\dir{-};
"01";"11" **\dir{-};
"00";"20" **\dir{-};
"20";"10" **\dir{-};
"00";"22" **\dir{-};
"22";"11" **\dir{-};
"01";"21" **\dir{-};
"21";"10" **\dir{-};
"01";"23" **\dir{-};
"23";"11" **\dir{-};
"20";"30" **\dir{-};
"30";"10" **\dir{-};
"10";"31" **\dir{-};
"31";"21" **\dir{-};
"21";"32" **\dir{-};
"32";"01" **\dir{-};
"01";"33" **\dir{-};
"33";"23" **\dir{-};
"23";"34" **\dir{-};
"34";"11" **\dir{-};
"11";"35" **\dir{-};
"35";"22" **\dir{-};
"22";"36" **\dir{-};
"36";"00" **\dir{-};
"00";"37" **\dir{-};
"37";"20" **\dir{-};
\endxy$}}$, and so on. For $r\geq1$ we define three $r\times r$ matrices denoted
$M_{++}^{\left(r\right)},M_{+-}^{\left(r\right)},M_{--}^{\left(r\right)}$,
and for $r\geq0$ a $\left(r+1\right)\times\left(r+1\right)$ matrix
$M_{-+}^{\left(r\right)}$, as follows:
\begin{gather*}
M_{-+}^{\left(0\right)}=\left(\begin{matrix}1\end{matrix}\right),\qquad M_{++}^{\left(1\right)}=M_{+-}^{\left(1\right)}=M_{--}^{\left(1\right)}=\left(\begin{matrix}0\end{matrix}\right)\\
M_{-+}^{\left(1\right)}=\left(\begin{smallmatrix}1 & -2\\
-1 & 2
\end{smallmatrix}\right),\qquad M_{++}^{\left(2\right)}=M_{+-}^{\left(2\right)}=\left(\begin{smallmatrix}\frac{1}{2} & -1\\
-1 & 2
\end{smallmatrix}\right),\qquad M_{--}^{\left(2\right)}=\left(\begin{smallmatrix}\frac{3}{2} & -1\\
-1 & 2
\end{smallmatrix}\right)
\end{gather*}
\begin{align*}
M_{++}^{\left(r\right)}=M_{+-}^{\left(r\right)} & =\left.\left(\begin{smallmatrix}\frac{1}{2} & -1\\
-\frac{1}{2} & \frac{3}{2} & -1\\
 &  &  & \ddots\\
 &  &  &  & -\frac{1}{2} & \frac{3}{2} & -1\\
 &  &  &  &  & -1 & 2
\end{smallmatrix}\right)\right\} r\\
M_{--}^{\left(r\right)} & =\left.\left(\begin{smallmatrix}\frac{3}{2} & -1\\
-\frac{1}{2} & \frac{3}{2} & -1\\
 &  &  & \ddots\\
 &  &  &  & -\frac{1}{2} & \frac{3}{2} & -1\\
 &  &  &  &  & -1 & 2
\end{smallmatrix}\right)\right\} r\\
M_{-+}^{\left(r\right)} & =\left.\left(\begin{smallmatrix}1 & -2\\
-\frac{1}{2} & \frac{3}{2} & -1\\
 &  &  & \ddots\\
 &  &  &  & -\frac{1}{2} & \frac{3}{2} & -1\\
 &  &  &  &  & -1 & 2
\end{smallmatrix}\right)\right\} r+1
\end{align*}

\begin{thm}
The spectrum of $X_{r}=B_{r}\left(T_{2}^{2}\right)$ is given (including
multiplicities) by 
\[
\Spec\Delta^{+}\left(X_{r}\right)=\Spec M_{++}^{\left(r\right)}\cup\Spec M_{+-}^{\left(r\right)}\cup\Spec M_{--}^{\left(r\right)}\cup\Spec M_{-+}^{\left(r\right)}\cup\bigcup_{j=1}^{r-1}\left[\Spec M_{++}^{\left(j\right)}\right]^{2^{r-j+1}}
\]
where $ $$\left[X\right]^{i}$ means that $X$ is repeated $i$ times.
\end{thm}
To make this clear, this gives 
\[
\left|\Spec\Delta^{+}\left(X_{r}\right)\right|=4r+1+\sum_{j=1}^{r-1}2^{r-j+1}\cdot j=2^{r+2}-3=\left|X_{r}^{1}\right|=\dim\Omega^{1}\left(X_{r}\right),
\]
as ought to be. 
\begin{proof}
The symmetry group of $X_{r}$ (for $r\geq1$) is $G=\left\{ id,\tau_{h},\tau_{v},\sigma\right\} $,
where $\tau_{h}$ is the horizontal reflection, $\tau_{v}$ is the
vertical reflection (around the middle edge $e_{0}$), and $\sigma=\tau_{h}\circ\tau_{v}=\tau_{v}\circ\tau_{h}$
is a rotation by $\pi$. The irreducible representations of $G$ are
given in Table \ref{tab:irreducible-representations}.

\begin{table}[h]
\begin{centering}
\begin{tabular}{|c|c|c|c|c|}
\hline 
 & $e$ & $\tau_{h}$ & $\tau_{v}$ & $\sigma$\tabularnewline
\hline 
\hline 
$V_{++}$ & $1$ & $1$ & $1$ & $1$\tabularnewline
\hline 
$V_{+-}$ & $1$ & $1$ & $-1$ & $-1$\tabularnewline
\hline 
$V_{-+}$ & $1$ & $-1$ & $1$ & $-1$\tabularnewline
\hline 
$V_{--}$ & $1$ & $-1$ & $-1$ & $1$\tabularnewline
\hline 
\end{tabular}
\par\end{centering}

\caption{\label{tab:irreducible-representations}The irreducible representations
of $G=\mathrm{Sym}\left(X_{r}\right)$.}
\end{table}

We define four orientations for $X_{r}$, denoted $X_{r}^{\pm\pm}$,
demonstrated in Figure \ref{fig:The-four-choices}. In all of them
$e_{0}$ is oriented from left to right, and the first (top right)
quadrant is oriented clockwise. Each of the other quadrants is then
oriented according to the corresponding representation, e.g.\ $X_{r}^{+-}$
satisfies the following: for every oriented edge $e$, if $e\in X_{r}^{+-}$
then $\tau_{h}e\in X_{r}^{+-}$, while $\tau_{v}e,\sigma e\notin X_{r}^{+-}$
(so that $\overline{\tau_{v}e},\overline{\sigma e}\in X_{r}^{+-}$). 

\begin{figure}[h]
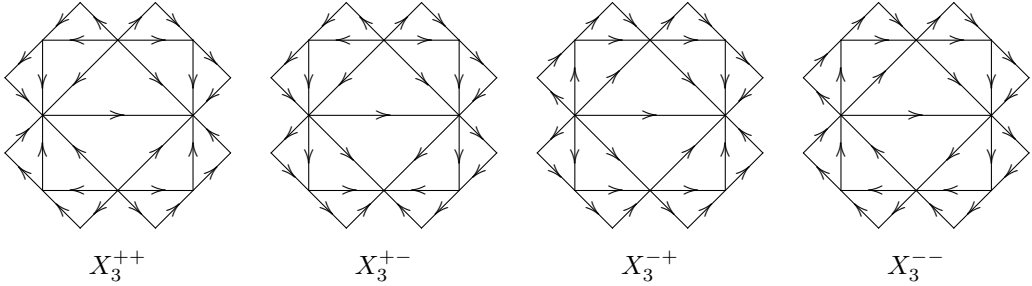

\hfill{}\xy (10,-20)*{X_3^{++}};
(0,0)*{}="00";
(20,0)*{}="01";
(10,10)*{}="10";
(10,-10)*{}="11";
(0,10)*{}="20";
(20,10)*{}="21";
(0,-10)*{}="22";
(20,-10)*{}="23";
(5,15)*{}="30";
(15,15)*{}="31";
(25,5)*{}="32";
(25,-5)*{}="33";
(15,-15)*{}="34";
(5,-15)*{}="35";
(-5,-5)*{}="36";
(-5,5)*{}="37";
"00",{\ar@{-}|(0.54)*=0@{>} "01"};
"00",{\ar@{-}|(0.41)*=0@{<} "10"};
"10",{\ar@{-}|(0.59)*=0@{>} "01"};
"01",{\ar@{-}|(0.41)*=0@{<} "11"};
"11",{\ar@{-}|(0.59)*=0@{>} "00"};
"00",{\ar@{-}|(0.37)*=0@{<} "20"};
"20",{\ar@{-}|(0.37)*=0@{<} "10"};
"11",{\ar@{-}|(0.63)*=0@{>} "22"};
"22",{\ar@{-}|(0.63)*=0@{>} "00"};
"10",{\ar@{-}|(0.63)*=0@{>} "21"};
"21",{\ar@{-}|(0.63)*=0@{>} "01"};
"01",{\ar@{-}|(0.37)*=0@{<} "23"};
"23",{\ar@{-}|(0.37)*=0@{<} "11"};
"00",{\ar@{-}|(0.37)*=0@{<} "37"};
"37",{\ar@{-}|(0.37)*=0@{<} "20"};
"20",{\ar@{-}|(0.37)*=0@{<} "30"};
"30",{\ar@{-}|(0.37)*=0@{<} "10"};
"10",{\ar@{-}|(0.63)*=0@{>} "31"};
"31",{\ar@{-}|(0.63)*=0@{>} "21"};
"21",{\ar@{-}|(0.63)*=0@{>} "32"};
"32",{\ar@{-}|(0.63)*=0@{>} "01"};
"01",{\ar@{-}|(0.37)*=0@{<} "33"};
"33",{\ar@{-}|(0.37)*=0@{<} "23"};
"23",{\ar@{-}|(0.37)*=0@{<} "34"};
"34",{\ar@{-}|(0.37)*=0@{<} "11"};
"11",{\ar@{-}|(0.63)*=0@{>} "35"};
"35",{\ar@{-}|(0.63)*=0@{>} "22"};
"22",{\ar@{-}|(0.63)*=0@{>} "36"};
"36",{\ar@{-}|(0.63)*=0@{>} "00"};
\endxy\hfill{}\xy (10,-20)*{X_3^{+-}};
(0,0)*{}="00";
(20,0)*{}="01";
(10,10)*{}="10";
(10,-10)*{}="11";
(0,10)*{}="20";
(20,10)*{}="21";
(0,-10)*{}="22";
(20,-10)*{}="23";
(5,15)*{}="30";
(15,15)*{}="31";
(25,5)*{}="32";
(25,-5)*{}="33";
(15,-15)*{}="34";
(5,-15)*{}="35";
(-5,-5)*{}="36";
(-5,5)*{}="37";
"00",{\ar@{-}|(0.54)*=0@{>} "01"};
"00",{\ar@{-}|(0.41)*=0@{<} "10"};
"10",{\ar@{-}|(0.59)*=0@{>} "01"};
"01",{\ar@{-}|(0.59)*=0@{>} "11"};
"11",{\ar@{-}|(0.41)*=0@{<} "00"};
"00",{\ar@{-}|(0.37)*=0@{<} "20"};
"20",{\ar@{-}|(0.37)*=0@{<} "10"};
"11",{\ar@{-}|(0.37)*=0@{<} "22"};
"22",{\ar@{-}|(0.37)*=0@{<} "00"};
"10",{\ar@{-}|(0.63)*=0@{>} "21"};
"21",{\ar@{-}|(0.63)*=0@{>} "01"};
"01",{\ar@{-}|(0.63)*=0@{>} "23"};
"23",{\ar@{-}|(0.63)*=0@{>} "11"};
"00",{\ar@{-}|(0.37)*=0@{<} "37"};
"37",{\ar@{-}|(0.37)*=0@{<} "20"};
"20",{\ar@{-}|(0.37)*=0@{<} "30"};
"30",{\ar@{-}|(0.37)*=0@{<} "10"};
"10",{\ar@{-}|(0.63)*=0@{>} "31"};
"31",{\ar@{-}|(0.63)*=0@{>} "21"};
"21",{\ar@{-}|(0.63)*=0@{>} "32"};
"32",{\ar@{-}|(0.63)*=0@{>} "01"};
"01",{\ar@{-}|(0.63)*=0@{>} "33"};
"33",{\ar@{-}|(0.63)*=0@{>} "23"};
"23",{\ar@{-}|(0.63)*=0@{>} "34"};
"34",{\ar@{-}|(0.63)*=0@{>} "11"};
"11",{\ar@{-}|(0.37)*=0@{<} "35"};
"35",{\ar@{-}|(0.37)*=0@{<} "22"};
"22",{\ar@{-}|(0.37)*=0@{<} "36"};
"36",{\ar@{-}|(0.37)*=0@{<} "00"};
\endxy\hfill{}\xy (10,-20)*{X_3^{-+}};
(0,0)*{}="00";
(20,0)*{}="01";
(10,10)*{}="10";
(10,-10)*{}="11";
(0,10)*{}="20";
(20,10)*{}="21";
(0,-10)*{}="22";
(20,-10)*{}="23";
(5,15)*{}="30";
(15,15)*{}="31";
(25,5)*{}="32";
(25,-5)*{}="33";
(15,-15)*{}="34";
(5,-15)*{}="35";
(-5,-5)*{}="36";
(-5,5)*{}="37";
"00",{\ar@{-}|(0.54)*=0@{>} "01"};
"00",{\ar@{-}|(0.59)*=0@{>} "10"};
"10",{\ar@{-}|(0.59)*=0@{>} "01"};
"01",{\ar@{-}|(0.41)*=0@{<} "11"};
"11",{\ar@{-}|(0.41)*=0@{<} "00"};
"00",{\ar@{-}|(0.63)*=0@{>} "20"};
"20",{\ar@{-}|(0.63)*=0@{>} "10"};
"11",{\ar@{-}|(0.37)*=0@{<} "22"};
"22",{\ar@{-}|(0.37)*=0@{<} "00"};
"10",{\ar@{-}|(0.63)*=0@{>} "21"};
"21",{\ar@{-}|(0.63)*=0@{>} "01"};
"01",{\ar@{-}|(0.37)*=0@{<} "23"};
"23",{\ar@{-}|(0.37)*=0@{<} "11"};
"20",{\ar@{-}|(0.63)*=0@{>} "30"};
"30",{\ar@{-}|(0.63)*=0@{>} "10"};
"10",{\ar@{-}|(0.63)*=0@{>} "31"};
"31",{\ar@{-}|(0.63)*=0@{>} "21"};
"21",{\ar@{-}|(0.63)*=0@{>} "32"};
"32",{\ar@{-}|(0.63)*=0@{>} "01"};
"01",{\ar@{-}|(0.37)*=0@{<} "33"};
"33",{\ar@{-}|(0.37)*=0@{<} "23"};
"23",{\ar@{-}|(0.37)*=0@{<} "34"};
"34",{\ar@{-}|(0.37)*=0@{<} "11"};
"11",{\ar@{-}|(0.37)*=0@{<} "35"};
"35",{\ar@{-}|(0.37)*=0@{<} "22"};
"22",{\ar@{-}|(0.37)*=0@{<} "36"};
"36",{\ar@{-}|(0.37)*=0@{<} "00"};
"00",{\ar@{-}|(0.63)*=0@{>} "37"};
"37",{\ar@{-}|(0.63)*=0@{>} "20"};
\endxy\hfill{}\xy (10,-20)*{X_3^{--}};
(0,0)*{}="00";
(20,0)*{}="01";
(10,10)*{}="10";
(10,-10)*{}="11";
(0,10)*{}="20";
(20,10)*{}="21";
(0,-10)*{}="22";
(20,-10)*{}="23";
(5,15)*{}="30";
(15,15)*{}="31";
(25,5)*{}="32";
(25,-5)*{}="33";
(15,-15)*{}="34";
(5,-15)*{}="35";
(-5,-5)*{}="36";
(-5,5)*{}="37";
"00",{\ar@{-}|(0.54)*=0@{>} "01"};
"00",{\ar@{-}|(0.59)*=0@{>} "10"};
"10",{\ar@{-}|(0.59)*=0@{>} "01"};
"01",{\ar@{-}|(0.59)*=0@{>} "11"};
"11",{\ar@{-}|(0.59)*=0@{>} "00"};
"00",{\ar@{-}|(0.63)*=0@{>} "20"};
"20",{\ar@{-}|(0.63)*=0@{>} "10"};
"11",{\ar@{-}|(0.63)*=0@{>} "22"};
"22",{\ar@{-}|(0.63)*=0@{>} "00"};
"10",{\ar@{-}|(0.63)*=0@{>} "21"};
"21",{\ar@{-}|(0.63)*=0@{>} "01"};
"01",{\ar@{-}|(0.63)*=0@{>} "23"};
"23",{\ar@{-}|(0.63)*=0@{>} "11"};
"20",{\ar@{-}|(0.63)*=0@{>} "30"};
"30",{\ar@{-}|(0.63)*=0@{>} "10"};
"10",{\ar@{-}|(0.63)*=0@{>} "31"};
"31",{\ar@{-}|(0.63)*=0@{>} "21"};
"21",{\ar@{-}|(0.63)*=0@{>} "32"};
"32",{\ar@{-}|(0.63)*=0@{>} "01"};
"01",{\ar@{-}|(0.63)*=0@{>} "33"};
"33",{\ar@{-}|(0.63)*=0@{>} "23"};
"23",{\ar@{-}|(0.63)*=0@{>} "34"};
"34",{\ar@{-}|(0.63)*=0@{>} "11"};
"11",{\ar@{-}|(0.63)*=0@{>} "35"};
"35",{\ar@{-}|(0.63)*=0@{>} "22"};
"22",{\ar@{-}|(0.63)*=0@{>} "36"};
"36",{\ar@{-}|(0.63)*=0@{>} "00"};
"00",{\ar@{-}|(0.63)*=0@{>} "37"};
"37",{\ar@{-}|(0.63)*=0@{>} "20"};
\endxy\hfill{}

\caption{\label{fig:The-four-choices}The four choices of orientations for
$X_{r}$, depicted for $r=3$.}
\end{figure}

The space of $1$-forms $\Omega^{1}\left(X_{r}\right)$ is naturally
a representation of $G=\mathrm{Sym}\left(X_{r}\right)$, by $\left(\gamma f\right)\left(e\right)=f\left(\gamma^{-1}e\right)$
(where $\gamma\in G$, $f\in\Omega^{1}\left(X_{r}\right)$, $e\in X_{r}^{1}$).
We denote by $\Omega_{\pm\pm}^{\left(r\right)}=\Omega_{\pm\pm}^{1}\left(X^{r}\right)$
its $V_{\pm\pm}$-isotypic components. For example, $f\in\Omega_{+-}^{\left(r\right)}$
if and only if it satisfies $\tau_{h}f=f$ and $\tau_{v}f=-f$ (which
implies that $\sigma f=\tau_{v}\tau_{h}f=-f$). 

We say that a $1$-form on $X_{r}$ is \emph{$++$-spherical}, denoted
$f\in\mathcal{S}_{++}^{\left(r\right)}$, if it is
\begin{enumerate}
\item spherical in absolute value (i.e.\ $\left|f\left(e\right)\right|=\left|f\left(e'\right)\right|$
whenever $\dist\left(e_{0},e\right)=\dist\left(e_{0},e'\right)$),
and 
\item $V_{++}$-isotypic (namely $f\in\Omega_{++}^{\left(r\right)}$, or
equivalently, $f$ is of constant sign on $X_{r}^{++}$). 
\end{enumerate}
The definition of $\mathcal{S}_{+-}^{\left(r\right)},\mathcal{S}_{-+}^{\left(r\right)},\mathcal{S}_{--}^{\left(r\right)}$
are analogue. 

Let $e_{1},\ldots,e_{r}$ be edges in the first quadrant of $X_{r}$
oriented as in $X_{r}^{\pm\pm}$, and with $\dist\left(e_{i},e_{0}\right)=i$.
Let $f$ be an eigenform of $\Delta^{+}$ with eigenvalue $\lambda$,
which is in one of the $\mathcal{S}_{\pm\pm}^{\left(r\right)}$. Then
for $2\leq i\leq r-1$
\[
\lambda f\left(e_{i}\right)=\left(\Delta^{+}f\right)\left(e_{i}\right)=f\left(e_{i}\right)-\frac{1}{2}\left[f\left(e_{i-1}\right)-f\left(e_{i}\right)+2f\left(e_{i+1}\right)\right]
\]
and 
\[
\lambda f\left(e_{r}\right)=\left(\Delta^{+}f\right)\left(e_{r}\right)=f\left(e_{r}\right)-\left[f\left(e_{r-1}\right)-f\left(e_{r}\right)\right].
\]
The behavior of $f$ around $e_{0},e_{1}$ depends on the isotypic
component. We assume $r\geq2$, and leave it to the reader to verify
the cases $r=0,1$. Every form in $\Omega_{++}^{\left(r\right)},\Omega_{+-}^{\left(r\right)},\Omega_{--}^{\left(r\right)}$
must vanish on the middle edge $e_{0}$: for the first two, since
\[
f\left(e_{0}\right)=\left(\tau_{h}f\right)\left(e_{0}\right)=f\left(\tau_{h}e_{0}\right)=f\left(\overline{e_{0}}\right)=-f\left(e_{0}\right),
\]
and for the last one since $f\left(e_{0}\right)=\left(-\tau_{v}f\right)\left(e_{0}\right)=-f\left(\tau_{v}e_{0}\right)=-f\left(e_{0}\right)$.
For a spherical $\left(-+\right)$-functions we have
\[
\lambda f\left(e_{0}\right)=\left(\Delta^{+}f\right)\left(e_{0}\right)=f\left(e_{0}\right)-\frac{1}{2}\left[4\cdot f\left(e_{1}\right)\right],
\]
and at $e_{1}$ we have (using the fact that $f\left(e_{0}\right)=0$
for $f\in\Omega_{++}^{\left(r\right)},\Omega_{+-}^{\left(r\right)},\Omega_{--}^{\left(r\right)}$)
\[
\lambda f\left(e_{1}\right)=\left(\Delta^{+}f\right)\left(e_{1}\right)=\begin{cases}
f\left(e_{1}\right)-\frac{1}{2}\left[f\left(e_{1}\right)+2f\left(e_{2}\right)\right] & f\in\Omega_{++}^{\left(r\right)},\Omega_{+-}^{\left(r\right)}\\
f\left(e_{1}\right)-\frac{1}{2}\left[f\left(e_{0}\right)-f\left(e_{1}\right)+2f\left(e_{2}\right)\right] & f\in\Omega_{-+}^{\left(r\right)}\\
f\left(e_{1}\right)-\frac{1}{2}\left[-f\left(e_{1}\right)+2f\left(e_{2}\right)\right] & f\in\Omega_{--}^{\left(r\right)}.
\end{cases}
\]
The matrices $M_{\pm\pm}^{\left(r\right)}$ represent these equations,
and thus the $++$-spherical spectrum of $X^{r}$ is $\Spec\Delta^{+}\big|_{\mathcal{S}_{++}^{\left(r\right)}}=\Spec M_{++}^{\left(r\right)}$,
and likewise for the other $\mathcal{S}_{\pm\pm}^{\left(r\right)}$.

Until now we have only accounted for the spherical part of $\Omega^{1}\left(X\right)$,
finding in total $4r+1$ eigenvalues. The other eigenvalues are obtained
by using spherical eigenforms of $X^{i}$ with $i<r$. 

Denote by $X_{r}^{\mathfrak{h}}$ the upper half of $X_{r}$, including
$e_{0}$, which is a fundamental domain for $\left\{ id,\tau_{v}\right\} $.
Observe that $X_{r}\backslash\overset{\circ}{X_{1}}$ (by which we
mean $X_{r}$ after deleting $e_{0}$ and the two triangles adjacent
to it, but not the other four edges), is comprised of four copies
of $X_{r-1}^{\mathfrak{h}}$, which intersect only in vertices. Denote
these four copies of $X_{r-1}^{\mathfrak{h}}$ by $Y_{1},\ldots,Y_{4}$.
Let $f\in\mathcal{S}_{++}^{\left(r-1\right)}$ be a $\left(++\right)$-spherical
$\lambda$-eigenform on $X_{r-1}$, and define $g\in\Omega^{1}\left(X_{r}\right)$
by $g\big|_{Y_{1}}=f\big|_{X_{r-1}^{\mathfrak{h}}}$ and $g\big|_{Y_{2}}=g\big|_{Y_{3}}=g\big|_{Y_{4}}=0$.
We show now that $g$ is a $\lambda$-eigenform of $X^{r}$. Since
$f\in\Omega_{++}^{\left(r-1\right)}$, $g\left(e_{1}\right)=f\left(e_{0}\right)=0$,
where $e_{1}$ is the edge incident to $e_{0}$ in $Y_{1}$. Therefore,
$\Delta^{+}g=\lambda g$ holds everywhere outside $Y_{1}$. It also
holds at $e_{1}$, since if $e_{2},e_{2}'$ are the two edges incident
to $e_{1}$ in $Y_{1}$, then $g\left(e_{2}\right)=-g\left(e_{2}'\right)$
since $f$ is symmetric with respect to $\tau_{h}$. Obviously, $\Delta^{+}g=\lambda g$
holds in $Y_{1}\backslash\left\{ e_{1}\right\} $, and we are done.
We could have taken $g\big|_{Y_{i}}=f\big|_{X_{r-1}^{\mathfrak{h}}}$
for any $i\in\left\{ 1,2,3,4\right\} $, and the resulting eigenforms
are independent. We remark that taking $f\in\Omega_{+-}^{\left(r-1\right)}$
would also work, but would give again the same eigenforms, while $f\in\Omega_{-+}^{\left(r\right)},\Omega_{--}^{\left(r\right)}$
would not define an eigenform on $X_{r}$.

More generally, $X_{r}\backslash\overset{\circ}{X_{j}}$ is comprised
of $2^{j+1}$ copies of $X_{r-j}^{\mathfrak{h}}$, and in a similar
way every eigenform of $\Delta^{+}\big|_{\mathcal{S}_{++}^{\left(r-j\right)}}$
contributes $2^{j+1}$ eigenforms to $X^{r}$. We recall that for
$f\in\mathcal{S}_{++}^{\left(r-j\right)}$ we always have $f\left(e_{0}\right)=0$,
and observe that due to the recursion relations if $f\neq0$ then
$f\left(e_{1}\right)\neq0$. Therefore, the eigenforms obtained from
copies of $X_{r-j}^{\mathfrak{h}}$ for various $j$ are all linearly
independent, as they are supported outside different balls in $X^{r}$.
Together with the $4r+1$ spherical eigenforms, this accounts for
\[
4r+1+\sum_{j=1}^{r}2^{j+1}\cdot\left|\Spec\Delta^{+}\big|_{\mathcal{S}_{++}^{\left(r-j\right)}}\right|=4r+1+\sum_{j=1}^{r-1}2^{j+1}\left(r-j\right)=2^{r+2}-3
\]
independent eigenforms, and since this is the dimension of $\Omega^{1}\left(X_{r}\right)$
we are done.\end{proof}
\begin{prop}
For every $r\in\mathbb{N}$ and $\lambda\in\Spec M_{\pm\pm}^{\left(r\right)}$,
either $\lambda=0$ or $\frac{3}{2}-\sqrt{2}<\lambda$.\end{prop}
\begin{proof}
Let $p_{++}^{\left[r\right]}\left(\lambda\right)=\det\left(M_{++}^{\left(r\right)}-\lambda I\right)$,
and similarly for the other $\pm\pm$. Expanding $M_{--}^{\left(r\right)}-\lambda I$
by minor gives 
\begin{gather*}
p_{--}^{\left[1\right]}\left(\lambda\right)=1-\lambda,\quad p_{--}^{\left[2\right]}\left(\lambda\right)=\lambda^{2}-\frac{7}{2}\lambda+2,\quad p_{--}^{\left[3\right]}=-\lambda^{3}+5\lambda^{2}-\frac{27}{4}\lambda+2\\
p_{--}^{\left[r\right]}\left(\lambda\right)=\left(\frac{3}{2}-\lambda\right)p_{--}^{\left[r-1\right]}\left(\lambda\right)-\frac{1}{2}p_{--}^{\left[r-2\right]}\left(\lambda\right)\qquad\left(r\geq4\right).
\end{gather*}
This yields a quadratic recurrence formula in $\mathbb{Q}\left[\lambda\right]$
whose solution (for $r\geq2$) is $p_{--}^{\left[r\right]}\left(\lambda\right)=\alpha\left(\lambda\right)\mu_{+}\left(\lambda\right)^{r}+\beta\left(\lambda\right)\mu_{-}\left(\lambda\right)^{r}$,
where 
\begin{align*}
\alpha\left(\lambda\right) & =2-\beta\left(\lambda\right)=\frac{\left(2\lambda-2\right)\sqrt{4\lambda^{2}-12\lambda+1}+4\lambda^{2}-10\lambda-2}{\left(2\lambda-3\right)\sqrt{4\lambda^{2}-12\lambda+1}+4\lambda^{2}-12\lambda+1},\\
\mu_{\pm}\left(\lambda\right) & =\frac{3}{4}-\frac{\lambda}{2}\pm\frac{1}{4}\sqrt{4\lambda^{2}-12\lambda+1}.
\end{align*}
For $0<\lambda<\frac{3}{2}-\sqrt{2}$ one can verify that $\beta\left(\lambda\right)<0<\alpha\left(\lambda\right)$
and $0<\mu_{-}\left(\lambda\right)<\mu_{+}\left(\lambda\right)$,
and for $r\geq2$ 
\begin{align*}
p_{--}^{\left[r\right]}\left(\lambda\right) & =\mu_{+}\left(\lambda\right)^{r}\left(\alpha\left(\lambda\right)+\beta\left(\lambda\right)\left(\frac{\mu_{-}\left(\lambda\right)}{\mu_{+}\left(\lambda\right)}\right)^{r}\right)\geq\mu_{+}\left(\lambda\right)^{r}\left(\alpha\left(\lambda\right)+\beta\left(\lambda\right)\left(\frac{\mu_{-}\left(\lambda\right)}{\mu_{+}\left(\lambda\right)}\right)^{2}\right)\\
 & =\mu_{+}\left(\lambda\right)^{r-2}\left(\alpha\left(\lambda\right)\mu_{+}\left(\lambda\right)^{2}+\beta\left(\lambda\right)\mu_{-}\left(\lambda\right)^{2}\right)=\mu_{+}\left(\lambda\right)^{r-2}p_{--}^{\left[2\right]}\left(\lambda\right)>0.
\end{align*}
Using the solution for $p_{--}^{\left[r\right]}$ one can write $p_{+-}^{\left[r\right]}$,
for $r\geq4$, as 
\begin{align*}
p_{+-}^{\left[r\right]}\left(\lambda\right) & =\left(\frac{1}{2}-\lambda\right)p_{--}^{\left[r-1\right]}\left(\lambda\right)-\frac{1}{2}p_{--}^{\left[r-2\right]}\left(\lambda\right)\\
 & =\alpha\left(\lambda\right)\left(\left(\frac{1}{2}-\lambda\right)\mu_{+}\left(\lambda\right)-\frac{1}{2}\right)\mu_{+}\left(\lambda\right)^{r-2}+\beta\left(\lambda\right)\left(\left(\frac{1}{2}-\lambda\right)\mu_{-}\left(\lambda\right)-\frac{1}{2}\right)\mu_{-}\left(\lambda\right)^{r-2}
\end{align*}
Now $\alpha\left(\lambda\right)\left(\left(\frac{1}{2}-\lambda\right)\mu_{+}\left(\lambda\right)-\frac{1}{2}\right)<0<\beta\left(\lambda\right)\left(\left(\frac{1}{2}-\lambda\right)\mu_{-}\left(\lambda\right)-\frac{1}{2}\right)$
for $0<\lambda<\frac{3}{2}-\sqrt{2}$, and it follows that $p_{+-}^{\left[r\right]}\left(\lambda\right)$
does not vanish in this interval. This takes care of $p_{++}^{\left[r\right]}\left(\lambda\right)$
as well, since $M_{++}^{\left[r\right]}=M_{+-}^{\left[r\right]}$.
The considerations for $p_{-+}^{\left[r\right]}\left(\lambda\right)$
are similar, and we leave them to the reader.
\end{proof}
We can conclude now that $\left\{ X_{r}\right\} _{r\in\mathbb{N}}$
constitute a counterexample for high-dimensional Alon-Boppana:
\begin{proof}[Proof of Theorem \ref{thm:T_2_2-counterexample}]
By the results in this section, the spectrum of $\Delta_{X_{r}}^{+}$
is contained in $\left\{ 0\right\} \cup\left(\frac{3}{2}-\sqrt{2},3\right]$.
Since $X_{r}$ is contractible, its first homology is trivial and
thus the zeros in the spectrum all belong to coboundaries, i.e., $\Spec\Delta_{X_{r}}^{+}\big|_{Z_{1}}\subseteq\left(\frac{3}{2}-\sqrt{2},3\right]$.
Therefore, $\liminf\limits _{r\rightarrow\infty}\lambda\left(X_{r}\right)\geq\frac{3}{2}-\sqrt{2}$.
In fact, $\liminf\limits _{r\rightarrow\infty}\lambda\left(X_{r}\right)=\frac{3}{2}-\sqrt{2}$.
This follows from $\frac{3}{2}-\sqrt{2}\in\Spec T_{2}^{2}$ (by Theorem
\ref{thm:The-spectrum-of-trees}), together with Theorem \ref{thm:cont-spec},
which asserts that there exist $\lambda_{r}\in\Spec\Delta_{X_{r}}^{+}$
such that $\lambda_{r}\rightarrow\frac{3}{2}-\sqrt{2}$. As $\lambda_{r}$
can be assumed to be nonzero, they are in fact in $\Spec\Delta_{X_{r}}^{+}\big|_{Z_{1}}$,
so that $\liminf\limits _{r\rightarrow\infty}\lambda\left(X_{r}\right)\leq\lim\limits _{r\rightarrow\infty}\lambda_{r}=\frac{3}{2}-\sqrt{2}$.
Finally, by Lemma \ref{lem:d-2-of-inf-deg} and Theorem \ref{thm:The-spectrum-of-trees}
we have $\lambda\left(T_{2}^{2}\right)=0$.
\end{proof}

\subsection{\label{sub:Spectral-radius-and}Spectral radius and random walk}

The spectral radius of an operator $T$ is $\rho\left(T\right)=\max\left\{ \left|\lambda\right|\,\middle|\,\lambda\in\Spec T\right\} $.
If $T$ is a self-adjoint operator on a Hilbert space then $\rho\left(T\right)=\left\Vert T\right\Vert $.
In this section we observe the transition operator $A=A\left(X,p\right)$
acting on $\Omega_{L^{2}}^{d-1}$, and relate it to the asymptotic
behavior of the expectation process on $X$. Under additional conditions,
this can be translated to a result on the spectral gap of the complex.
\begin{prop}
\label{prop:spectral_radius}Let $\mathcal{E}_{n}^{\sigma}$ be the
expectation process associated with the $p$-lazy $\left(d-1\right)$-walk
on a finite or countable $d$-complex $X$ with bounded $\left(d-1\right)$-degrees.
\begin{enumerate}
\item For all values of $p$
\[
\sup_{\sigma\in X_{\pm}^{d-1}}\limsup_{n\rightarrow\infty}\sqrt[n]{\left|\mathcal{E}_{n}^{\sigma}\left(\sigma\right)\right|}=\|A\|=\rho\left(A\right).
\]

\item If $\frac{1}{2}\leq p\leq1$ then 
\[
\sup_{\sigma\in X_{\pm}^{d-1}}\limsup_{n\rightarrow\infty}\sqrt[n]{\mathcal{E}_{n}^{\sigma}\left(\sigma\right)}=\|A\|=\max\Spec A=\frac{p\left(d-1\right)+1}{d}-\frac{1-p}{d}\cdot\min\Spec\Delta^{+}.
\]

\item If $\frac{1}{2}\leq p\leq1$ and all $\left(d-2\right)$-cells in
$X$ are of infinite degree, then
\[
\sup_{\sigma\in X_{\pm}^{d-1}}\limsup_{n\rightarrow\infty}\sqrt[n]{\mathcal{E}_{n}^{\sigma}\left(\sigma\right)}=\frac{p\left(d-1\right)+1}{d}-\frac{1-p}{d}\cdot\lambda\left(X\right).
\]

\end{enumerate}
\end{prop}
\begin{proof}
For an oriented $\left(d-1\right)$-cell $\sigma$, 
\[
\mathcal{E}_{n}^{\sigma}\left(\sigma\right)=A^{n}\mathcal{E}_{0}^{\sigma}\left(\sigma\right)=\deg\sigma\left\langle A^{n}\mathbbm{1}_{\sigma},\mathbbm{1}_{\sigma}\right\rangle =\deg\sigma\int_{\mathbb{C}}z^{n}d\mu_{\sigma}\left(z\right)=\deg\sigma\int_{\Spec A}z^{n}d\mu_{\sigma}\left(z\right),
\]
where $\mu_{\sigma}$ is the spectral measure of $A$ with respect
to $\mathbbm{1}_{\sigma}$. It follows that 
\[
\limsup_{n\rightarrow\infty}\sqrt[n]{\left|\mathcal{E}_{n}^{\sigma}\left(\sigma\right)\right|}=\limsup_{n\rightarrow\infty}\sqrt[n]{\deg\sigma\left|\int_{\supp\mu_{\sigma}}\negthickspace\negthickspace\negthickspace\negthickspace\negthickspace\negthickspace\negthickspace\negthickspace z^{n}d\mu_{\sigma}\left(z\right)\right|}=\max\left\{ \left|\lambda\right|\,\middle|\,\lambda\in\supp\mu_{\sigma}\right\} ,
\]
and by $\Spec\left(A\right)=\bigcup\limits _{\sigma\in X_{\pm}^{d-1}}\text{supp}\left(\mu_{\sigma}\right)$
(see (\ref{eq:spectral_stuff})) 
\[
\sup_{\sigma\in X_{\pm}^{d-1}}\limsup_{n\rightarrow\infty}\sqrt[n]{\left|\mathcal{E}_{n}^{\sigma}\left(\sigma\right)\right|}=\sup_{\sigma\in X_{\pm}^{d-1}}\max\left\{ \left|\lambda\right|\,\middle|\,\lambda\in\supp\mu_{\sigma}\right\} =\rho\left(A\right),
\]
settling \emph{(1)}. Since $\Spec\left(A\right)\subseteq\left[2p-1,\frac{p\left(d-1\right)+1}{d}\right]$,
in the case $p\geq\frac{1}{2}$ the spectrum of $A$ is nonnegative.
Therefore, 
\[
\mathcal{E}_{n}^{\sigma}\left(\sigma\right)=A^{n}\mathcal{E}_{0}^{\sigma}\left(\sigma\right)=\deg\sigma\left\langle A^{n}\mathbbm{1}_{\sigma},\mathbbm{1}_{\sigma}\right\rangle \geq0
\]
so that $\left|\mathcal{E}_{n}^{\sigma}\left(\sigma\right)\right|=\mathcal{E}_{n}^{\sigma}\left(\sigma\right)$,
and in addition $\rho\left(A\right)=\max\Spec A$. This accounts for
\emph{(2)}, and combining this with Lemma \ref{lem:d-2-of-inf-deg}
gives \emph{(3).}
\end{proof}
This proposition is a generalization of the classic connection between
return probability and spectral radius in an infinite connected graph.
Namely, for any vertex $v$ the non-lazy walk on this graph satisfies
\[
\lim_{n\rightarrow\infty}\sqrt[n]{\mathbf{p}_{n}^{v}\left(v\right)}=1-\lambda\left(G\right)=\max\Spec A=\left\Vert A\right\Vert =\rho\left(A\right),
\]
where $A$ is the transition operator of the walk. There are slight
differences, though: in general dimension $p\geq\frac{1}{2}$ is needed
for some of these equalities, and in addition one must take the supremum
over all possible starting points for the process. For graphs this
is not necessary (provided the graph is connected), and we do not
know whether the same is true in general dimension. One case in which
this is not necessary is when the complex is $\left(d-1\right)$-transitive,
in the sense that its symmetry group acts transitively on $X^{d-1}$.
This (together with Theorem \ref{thm:The-spectrum-of-trees}) leads
to the following corollary:
\begin{cor}
For the $k$-regular arboreal $d$-complex $T_{k}^{d}$, the non-lazy
random walk starting at any $\left(d-1\right)$-cell $\sigma$ satisfies
\[
\limsup_{n\rightarrow\infty}\sqrt[n]{\left|\mathbf{p}_{n}^{\sigma}\left(\sigma\right)-\mathbf{p}_{n}^{\sigma}\left(\overline{\sigma}\right)\right|}=\frac{d-1+2\sqrt{d(k-1)}}{kd}.
\]
For $p\geq\frac{1}{2}$, the $p$-lazy walk satisfies 
\[
\limsup_{n\rightarrow\infty}\sqrt[n]{\mathbf{p}_{n}^{\sigma}\left(\sigma\right)-\mathbf{p}_{n}^{\sigma}\left(\overline{\sigma}\right)}=\begin{cases}
p+\frac{1-p}{d} & \quad2\leq k\leq d+1\\
p+\left(1-p\right)\frac{1-d+2\sqrt{d(k-1)}}{kd} & \quad d+1\leq k
\end{cases}.
\]

\end{cor}
Another corollary of Proposition \ref{prop:spectral_radius} is the
following:
\begin{cor}
If $\dim X=d$ and there exists some $\tau\in X^{d-2}$ of finite
degree (in particular, if $X$ is finite), then the $p\geq\frac{1}{2}$
lazy random walk satisfies 
\[
\sup_{\sigma\in X_{\pm}^{d-1}}\limsup_{n\rightarrow\infty}\sqrt[n]{\mathbf{p}_{n}^{\sigma}\left(\sigma\right)-\mathbf{p}_{n}^{\sigma}\left(\overline{\sigma}\right)}=p+\frac{1-p}{d}.
\]
\end{cor}
\begin{proof}
The form $\delta_{d-1}\mathbbm{1}_{\tau}$ is in $\Omega_{L^{2}}^{d-1}$
and in $\ker\delta_{d}$, so that $0\in\Spec\Delta^{+}$.
\end{proof}

\subsection{\label{sub:Amenability-transiency}Amenability, transience and recurrence}

An infinite connected graph with finite degrees is said to be \emph{amenable}
if its Cheeger constant 
\[
h\left(X\right)=\min_{{A\subseteq V\atop 0<\left|A\right|<\infty}}\frac{\left|E\left(A,V\backslash A\right)\right|}{\left|A\right|}
\]
is zero. It is called \emph{recurrent} if with probability one the
random walk on it returns to its starting point, and \emph{transient}
otherwise. A nonamenable graph is always transient.

All three notions have many equivalent characterizations. Among these
are the following, which relate to the Laplacian of the graph:
\begin{enumerate}
\item If $X$ has bounded degrees, then it is amenable if and only if $\lambda\left(X\right)=\min\Spec\Delta^{+}=0$.
This follows from the so-called discrete Cheeger inequalities due
to Tanner, Dodziuk, and Alon-Milman \cite{Dod84,Tan84,AM85,Alo86}.
\item $X$ is transient if and only if $\mathbb{E}^{v}\left[{\mbox{number of}\atop \mbox{visits to }v}\right]=\sum_{n=0}^{\infty}\mathbf{p}_{n}^{v}\left(v\right)<\infty$
for some $v$, or equivalently for all $v$.
\item $X$ is transient if and only if there exists $f\in\Omega_{L^{2}}^{1}\left(X\right)$
such that $\partial f=\mathbbm{1}_{v}$ for some $v$, or equivalently
for all $v$ \cite{lyons1983simple}.
\end{enumerate}
This suggests observing the following generalizations of these notions
for a simplicial complex of dimension $d$:
\begin{lyxlist}{00.00.0000}
\item [{\textbf{$\mathbf{\left(A\right)}$}}] $\lambda\left(X\right)=0$.
\item [{$\mathbf{\left(A'\right)}$}] $\min\Spec\Delta^{+}=0$.
\item [{$\left(\mathbf{T}\right)$}] $\sum_{n=0}^{\infty}\widetilde{\mathcal{E}}_{n}^{\sigma}\left(\sigma\right)<\infty$
for every $\sigma\in X^{d-1}$, where $\widetilde{\mathcal{E}}$ is
the normalized expectation process of laziness $p$ on $X$, for some
$\frac{1}{2}\leq p<1$ (see Proposition \ref{pro:amen-trans}\emph{(\ref{enu:some-any-p})}).
\item [{$\left(\mathbf{T}'\right)$}] For every $\sigma\in X^{d-1}$ there
exists $f\in\Omega_{L^{2}}^{d}\left(X\right)$ such that $\partial_{d}f=\mathbbm{1}_{\sigma}$.
\end{lyxlist}
For infinite graphs, $\mathbf{\left(A\right)}$ and $\mathbf{\left(A'\right)}$
are the same and are equivalent to amenability, and $\left(\mathbf{T}\right)$
(for any $p$) and $\left(\mathbf{T}'\right)$ are equivalent to transience.
These definitions suggests many questions, some of which are presented
in the next section. The next proposition points out some observations
regarding them. Let us also define the property:
\begin{lyxlist}{00.00.0000}
\item [{$\left(\mathbf{S}\right)$}] All $\left(d-2\right)$-cells in $X$
have infinite degrees,
\end{lyxlist}
which holds in any infinite graph.
\begin{prop}
\label{pro:amen-trans}Let $X$ be a complex of dimension $d$ with
bounded $\left(d-1\right)$-degrees. Then
\begin{enumerate}
\item $\left(\mathbf{A}\right)\Rightarrow\left(\mathbf{A}'\right)$.
\item $\left(\mathbf{A}'\right)+\left(\mathbf{S}\right)\Rightarrow\left(\mathbf{A}\right)$.
\item $\neg\left(\mathbf{A}'\right)\Rightarrow\left(\mathbf{T}'\right)\Rightarrow\left(\mathbf{S}\right)$.
\item $\neg\left(\mathbf{A}'\right)\Rightarrow\left(\mathbf{T}\right)$.
\item \label{enu:some-any-p}If $\left(\mathbf{T}\right)$ holds for some
$\frac{1}{2}\leq p<1$, then it holds for every such $p$.
\item If zero is an isolated point in $\Spec\Delta^{+}$ then $\neg\left(\mathbf{T}\right)$.
\end{enumerate}
\end{prop}
\begin{proof}
\emph{(1)} is trivial and \emph{(2)} follows from Lemma \ref{lem:d-2-of-inf-deg}.

\emph{(3) }If $\left(\mathbf{A}'\right)$ fails then $0\notin\Spec\Delta^{+}$,
which means that $\Delta^{+}$ is invertible on $\Omega_{L^{2}}^{d-1}\left(X\right)$.
Thus, for every $\sigma\in X^{d-1}$ there exists $\psi\in\Omega_{L^{2}}^{d-1}$
s.t.\ $\Delta^{+}\psi=\mathbbm{1}_{\sigma}$, and taking $f=\delta_{d}\psi$
gives $\left(\mathbf{T}'\right)$. If $\left(\mathbf{S}\right)$ fails,
then some $\tau\in X^{d-2}$ has finite degree. In this case for any
$f\in\Omega^{d}$ one has 
\[
\left(\partial_{d-1}\partial_{d}f\right)\left(\tau\right)=\sum_{v\triangleleft\tau}\left(\partial_{d}f\right)\left(v\tau\right)=\sum_{v\triangleleft\tau}\sum_{w\triangleleft v\tau}f\left(wv\tau\right)=0,
\]
since this sums over every $d$-cell containing $\tau$ exactly twice,
with opposite orientations. If $\sigma$ is any $\left(d-1\right)$-cell
containing $\tau$, then $\partial_{d}f=\mathbbm{1}_{\sigma}$ would
give $0=\left(\partial_{d-1}\partial_{d}f\right)\left(\tau\right)=\left(\partial_{d-1}\mathbbm{1}_{\sigma}\right)\left(\tau\right)=1$,
so that $\left(\mathbf{T}'\right)$ fails.

\emph{(4)} If $\min\Spec\Delta^{+}>0$ then by Proposition \ref{prop:spectral_radius}\emph{(2)}
\[
\sup_{\sigma\in X_{\pm}^{d-1}}\limsup_{n\rightarrow\infty}\sqrt[n]{\widetilde{\mathcal{E}}_{n}^{\sigma}\left(\sigma\right)}=1-\frac{1-p}{p\left(d-1\right)+1}\cdot\min\Spec\Delta^{+}<1
\]
which gives $\sum_{n=0}^{\infty}\widetilde{\mathcal{E}}_{n}^{\sigma}\left(\sigma\right)<\infty$
for every $\sigma$.

\emph{(5)} Let $\frac{1}{2}\leq p$, and denote by $\left\{ \widetilde{\mathcal{E}}_{n}^{p,\sigma}\right\} _{n=0}^{\infty}$
the $p$-lazy normalized expectation process starting from $\sigma$
and $\widetilde{A}_{p}=\frac{p\left(d-1\right)+1}{d}\cdot A_{p}$.
Recall that $\widetilde{\mathcal{E}}_{n}^{p,\sigma}=\widetilde{A}_{p}^{n}\widetilde{\mathcal{E}}_{0}^{p,\sigma}=\widetilde{A}_{p}^{n}\one_{\sigma}$,
and let $\mu_{p}=\mu_{\sigma}^{\widetilde{A}_{p}}$ be the spectral
measure of $\widetilde{A}_{p}$ w.r.t.\ \textbf{$\sigma$}. Then

\[
\sum_{n=0}^{\infty}\widetilde{\mathcal{E}}_{n}^{p,\sigma}\left(\sigma\right)=\sum_{n=0}^{\infty}\widetilde{A}_{p}^{n}\one_{\sigma}\left(\sigma\right)=\deg\sigma\sum_{n=0}^{\infty}\left\langle \widetilde{A}_{p}^{n}\one_{\sigma},\one_{\sigma}\right\rangle =\deg\sigma\sum_{n=0}^{\infty}\int_{\Spec\widetilde{A}_{p}}\negthickspace\negthickspace\negthickspace\negthickspace\negthickspace\negthickspace\negthickspace\negthickspace\lambda^{n}d\mu_{p}\left(\lambda\right).
\]
Since $\Spec\widetilde{A}_{p}\subseteq\left[\frac{d\left(2p-1\right)}{p\left(d-1\right)+1},1\right]\subseteq\left[0,1\right]$,
by monotone convergence 
\begin{equation}
\sum_{n=0}^{\infty}\widetilde{\mathcal{E}}_{n}^{p,\sigma}\left(\sigma\right)=\deg\sigma\sum_{n=0}^{\infty}\int_{\Spec\widetilde{A}_{p}}\negthickspace\negthickspace\negthickspace\negthickspace\negthickspace\negthickspace\negthickspace\negthickspace\lambda^{n}d\mu_{p}\left(\lambda\right)=\deg\sigma\int_{\Spec\widetilde{A}_{p}}\frac{d\mu_{p}\left(\lambda\right)}{1-\lambda}\label{eq:spectral-equation}
\end{equation}
which is to be understood as $\infty$ if $\mu_{p}$ has an atom at
$\lambda=1$. Given $p<q<1$ one has $\widetilde{A}_{q}=\pi\widetilde{A}_{p}+\left(1-\pi\right)I$,
where $\pi=\pi\left(p,q,d\right)=\frac{1-q}{1-p}\cdot\frac{p\left(d-1\right)+1}{q\left(d-1\right)+1}\in\left(0,1\right)$.
The spectral measure of $\widetilde{A}_{q}$ w.r.t.\ $\sigma$ is
thus given by $\mu_{q}=\mu_{\sigma}^{\widetilde{A}_{q}}=\mu_{\sigma}^{\widetilde{A}_{p}}\circ g^{-1}$
where $g\left(\lambda\right)=\pi\lambda+1-\pi$, so that 
\[
\sum_{n=0}^{\infty}\widetilde{\mathcal{E}}_{n}^{q,\sigma}\left(\sigma\right)=\deg\sigma\int_{\Spec\left(\widetilde{A}_{q}\right)}\frac{d\mu_{q}\left(\lambda\right)}{1-\lambda}=\deg\sigma\int_{\Spec\widetilde{A}_{p}}\frac{d\mu_{p}\left(\lambda\right)}{1-\left(\pi\lambda+1-\pi\right)}=\frac{1}{\pi}\sum_{n=0}^{\infty}\widetilde{\mathcal{E}}_{n}^{p,\sigma}\left(\sigma\right)
\]
which completes the proof. Finally, \emph{(6)} follows from (\ref{eq:spectral-equation})
as an isolated point in the spectrum implies an atom at $1$.

\end{proof}

\section{\label{sec:Open-Questions}Open Questions}
\begin{enumerate}
\item In an infinite connected graph, the limit of $\sqrt[n]{\mathbf{p}_{n}^{v}\left(v\right)}$
(which describes the spectral radius of the transition operator, see
§\ref{sub:Spectral-radius-and}) is independent of the starting point
$v$. Is the same true in higher dimension? Namely, is $\limsup_{n\rightarrow\infty}\sqrt[n]{\mathcal{E}_{n}^{\sigma}\left(\sigma\right)}$
independent of $\sigma$ for a $\left(d-1\right)$-connected complex?
\item If $X\twoheadrightarrow Y$ is a covering map of graphs, then $\lambda\left(X\right)\geq\lambda\left(Y\right)$
(see e.g.\ \cite[Lemma 3.1]{MR0109367}, but beware - Kesten uses
$\lambda\left(X\right)$ for what we denote by $1-\lambda\left(X\right)$).
Does the same holds in higher dimension? If $\pi:X\twoheadrightarrow Y$
is a covering map of $d$-complexes, then the same argumentation as
in graphs shows that for any $\widetilde{\sigma}\in X^{d-1}$ and
$\sigma=\pi\left(\widetilde{\sigma}\right)\in Y^{d-1}$ one has $\mathbf{p}_{n}^{\widetilde{\sigma}}\left(\widetilde{\sigma}\right)\leq\mathbf{p}_{n}^{\sigma}\left(\sigma\right)$
and also $\mathbf{p}_{n}^{\widetilde{\sigma}}\left(\overline{\widetilde{\sigma}}\right)\leq\mathbf{p}_{n}^{\sigma}\left(\overline{\sigma}\right)$.
This, however, does not suffice to show that $\mathcal{E}_{n}^{\widetilde{\sigma}}\left(\widetilde{\sigma}\right)\leq\mathcal{E}_{n}^{\sigma}\left(\sigma\right)$.
Showing that this hold (or even that it holds after taking $n^{\mathrm{th}}$-roots
and letting $n\rightarrow\infty$) would give the desired result.
\item It is not hard to see that a $\left(d+1\right)$-partite $d$-complex
is disorientable, but for $d\geq2$ one can also construct examples
of disorientable complexes which are not $\left(d+1\right)$-partite.
It seems reasonable to conjecture that for simply connected complexes
these properties coincide. Is this indeed the case?
\item The suggestions for higher-dimensional analogues of amenability and
transience raise several questions:

\begin{enumerate}
\item Can high amenability and transience be characterized in non-spectral
terms (i.e.\ combinatorial expansion, or some $1-0$ event in the
$\left(d-1\right)$-walk model)? 
\item Are the transience properties $\left(\mathbf{T}\right)$ and\textbf{
$\left(\mathbf{T}'\right)$ }equivalent under some conditions?
\item Are all $d$-complexes with degrees bounded by $d+1$ $d$-amenable?
\end{enumerate}
\item In classical settings, the Brownian motion on a Riemannian manifold
constitute a continuous limit of the discrete random walk. Can one
define a continuous process, say, on the $\left(d-1\right)$-sphere
bundle of a Riemannian manifold, which relates to its $\left(d-1\right)$-homology
and to the spectrum of the Laplace-Beltrami operator on $\left(d-1\right)$-forms?
\item There are surprising and useful connections between random walks on
graphs and electrical networks (see e.g.\ \cite{doyle1984random,lyons2005probability}).
Can a parallel theory be devised for the random $\left(d-1\right)$-walk
on $d$-complexes?
\end{enumerate}

\bibliographystyle{amsalpha}
\bibliography{Random_walks_on_complexes}

\newcommand{\etalchar}[1]{$^{#1}$}
\providecommand{\bysame}{\leavevmode\hbox to3em{\hrulefill}\thinspace}
\providecommand{\MR}{\relax\ifhmode\unskip\space\fi MR }
\providecommand{\MRhref}[2]{%
  \href{http://www.ams.org/mathscinet-getitem?mr=#1}{#2}
}
\providecommand{\href}[2]{#2}
\begin{thebibliography}{FGL{\etalchar{+}}10}

\bibitem[Alo86]{Alo86}
N.~Alon, \emph{Eigenvalues and expanders}, Combinatorica \textbf{6} (1986),
  no.~2, 83--96.

\bibitem[AM85]{AM85}
N.~Alon and V.D. Milman, \emph{$\lambda_1$, isoperimetric inequalities for
  graphs, and superconcentrators}, Journal of Combinatorial Theory, Series B
  \textbf{38} (1985), no.~1, 73--88.

\bibitem[CS{\.Z}03]{cartwright2003ramanujan}
D.I. Cartwright, P.~Sol{\'e}, and A.~{\.Z}uk, \emph{Ramanujan geometries of
  type $\tilde{A}_n$}, Discrete mathematics \textbf{269} (2003), no.~1, 35--43.

\bibitem[DK10]{dotterrer2010coboundary}
D.~Dotterrer and M.~Kahle, \emph{Coboundary expanders}, arXiv preprint
  arXiv:1012.5316 (2010).

\bibitem[Dod84]{Dod84}
J.~Dodziuk, \emph{Difference equations, isoperimetric inequality and transience
  of certain random walks}, Trans. Amer. Math. Soc \textbf{284} (1984).

\bibitem[DS84]{doyle1984random}
P.G. Doyle and J.L. Snell, \emph{Random walks and electric networks}, no.~22,
  Mathematical Association of America, 1984.

\bibitem[Eck44]{Eck44}
B.~Eckmann, \emph{Harmonische funktionen und randwertaufgaben in einem
  komplex}, Commentarii Mathematici Helvetici \textbf{17} (1944), no.~1,
  240--255.

\bibitem[Fel66]{feller1966introduction}
W.~Feller, \emph{An introduction to probability theory, vol. ii}, 1966.

\bibitem[FGL{\etalchar{+}}10]{fox2011}
J.~Fox, M.~Gromov, V.~Lafforgue, A.~Naor, and J.~Pach, \emph{Overlap properties
  of geometric expanders}, arXiv preprint arXiv:1005.1392 (2010).

\bibitem[Gar73]{Gar73}
H.~Garland, \emph{p-adic curvature and the cohomology of discrete subgroups of
  p-adic groups}, The Annals of Mathematics \textbf{97} (1973), no.~3,
  375--423.

\bibitem[Gro10]{Gromov2010}
M.~Gromov, \emph{Singularities, expanders and topology of maps. part 2: From
  combinatorics to topology via algebraic isoperimetry}, Geometric and
  Functional Analysis \textbf{20} (2010), no.~2, 416--526.

\bibitem[GW12]{Gundert2012}
A.~Gundert and U.~Wagner, \emph{On laplacians of random complexes}, Proceedings
  of the 2012 symposuim on Computational Geometry, ACM, 2012, pp.~151--160.

\bibitem[G{\.Z}99]{grigorchuk1999asymptotic}
R.I. Grigorchuk and A.~{\.Z}uk, \emph{On the asymptotic spectrum of random
  walks on infinite families of graphs}, Random walks and discrete potential
  theory (Cortona, 1997), Sympos. Math \textbf{39} (1999), 188--204.

\bibitem[Kes59]{MR0109367}
H.~Kesten, \emph{Symmetric random walks on groups}, Trans. Amer. Math. Soc.
  \textbf{92} (1959), 336--354. \MR{0109367 (22 \#253)}

\bibitem[Li04]{li2004ramanujan}
W.C.W. Li, \emph{Ramanujan hypergraphs}, Geometric and Functional Analysis
  \textbf{14} (2004), no.~2, 380--399.

\bibitem[LM06]{Linial2006}
N.~Linial and R.~Meshulam, \emph{Homological connectivity of random
  2-complexes}, Combinatorica \textbf{26} (2006), no.~4, 475--487.

\bibitem[LP05]{lyons2005probability}
R.~Lyons and Y.~Peres, \emph{Probability on trees and networks}, 2005.

\bibitem[LSV05]{Lubotzky2005a}
A.~Lubotzky, B.~Samuels, and U.~Vishne, \emph{{ }{R}amanujan complexes of type
  $\tilde{A}_{d}$}, Israel Journal of Mathematics \textbf{149} (2005), no.~1,
  267--299.

\bibitem[Lub13]{Lubotzky2013}
A.~Lubotzky, \emph{Ramanujan complexes and high dimensional expanders}, arXiv
  preprint arXiv:1301.1028 (2013).

\bibitem[Lyo83]{lyons1983simple}
T.~Lyons, \emph{A simple criterion for transience of a reversible {M}arkov
  chain}, The Annals of Probability (1983), 393--402.

\bibitem[MW09]{meshulam2009homological}
R.~Meshulam and N.~Wallach, \emph{Homological connectivity of random
  k-dimensional complexes}, Random Structures \& Algorithms \textbf{34} (2009),
  no.~3, 408--417.

\bibitem[MW11]{Matouvsek2011}
J.~Matou\v{s}ek and U.~Wagner, \emph{On {G}romov's method of selecting heavily
  covered points}, Arxiv preprint arXiv:1102.3515 (2011).

\bibitem[Nil91]{nilli1991second}
A.~Nilli, \emph{On the second eigenvalue of a graph}, Discrete Mathematics
  \textbf{91} (1991), no.~2, 207--210.

\bibitem[PRT12]{parzanchevski2012isoperimetric}
O.~Parzanchevski, R.~Rosenthal, and R.J. Tessler, \emph{Isoperimetric
  inequalities in simplicial complexes}, arXiv preprint arXiv:1207.0638 (2012).

\bibitem[SKM12]{mukherjee2012cheeger}
J.~Steenbergen, C.~Klivans, and S.~Mukherjee, \emph{A {C}heeger-type inequality
  on simplicial complexes}, arXiv preprint arXiv:1209.5091 (2012).

\bibitem[Tan84]{Tan84}
R.M. Tanner, \emph{Explicit concentrators from generalized $ n $-gons}, SIAM
  Journal on Algebraic and Discrete Methods \textbf{5} (1984), 287.

\end{thebibliography}

\end{document}